\documentclass{siamltex}

\newcommand{\mathbb}[1]{\mathsf{I\!#1}}

\newcommand{\BYDEF}{\,\shortstack{\textup{\tiny def} \\ = }\,}

\newtheorem{a-1}{Assumption}[section]
\newtheorem{a-2}[a-1]{Assumption}
\newtheorem{a-3}[a-1]{Assumption}
\newtheorem{a-4}[a-1]{Assumption}
\newtheorem{a-5}[a-1]{Assumption}

\newtheorem{l-1a}[a-1]{Lemma}
\newtheorem{t-1new}[a-1]{Theorem}
\newtheorem{t-2new}[a-1]{Theorem}
\newtheorem{d-1}[a-1]{Definition}
\newtheorem{t-1}[a-1]{Theorem}
\newtheorem{p-1opt}[a-1]{Proposition}
\newtheorem{p-limav}[a-1]{Proposition}

\newtheorem{l-1}[a-1]{Lemma}
\newtheorem{l-2}[a-1]{Lemma}
\newtheorem{c-1}[a-1]{Corollary}
\newtheorem{c-2}[a-1]{Corollary}
\newtheorem{t-2}[a-1]{Theorem}
\newtheorem{p-1}[a-1]{Proposition}
\newtheorem{t-3new}[a-1]{Theorem}

\newtheorem{a-1x}[a-1]{Assumption}
\newtheorem{a-2x}[a-1]{Assumption}
\newtheorem{a-3x}[a-1]{Assumption}
\newtheorem{a-4x}[a-1]{Assumption}
\newtheorem{a-5x}[a-1]{Assumption}
\newtheorem{a-6x}[a-1]{Assumption}
\newtheorem{a-7x}[a-1]{Assumption}
\newtheorem{p-1x}[a-1]{Proposition}
\newtheorem{t-1x}[a-1]{Theorem}

\newtheorem{r-1}[a-1]{Remark}
\newtheorem{r-2}[a-1]{Remark}
\newtheorem{r-3}[a-1]{Remark}

\newtheorem{l-81}[a-1]{Lemma}
\newtheorem{l-82}[a-1]{Lemma}

\usepackage{graphicx,color}
\usepackage{latexsym}
\usepackage{amsfonts}
\def\Re{{\rm I\kern-0.2em R}}

\def\R{\Re}

\newtheorem{Theorem}{Theorem}[section]
\newtheorem{Definition}[Theorem]{Definition}
\newtheorem{Proposition}[Theorem]{Proposition}
\newtheorem{Lemma}[Theorem]{Lemma}
\newtheorem{Corollary}[Theorem]{Corollary}
\newtheorem{Remark}[Theorem]{Remark}

\newtheorem{Assumption}[Theorem]{Assumption}

\usepackage{amsmath}
\usepackage{amssymb}
\usepackage{epsf}

\title{On average control generating families for singularly perturbed
optimal control problems  with
 long run  average optimality criteria}
\author{Vladimir Gaitsgory\thanks{Department of Mathematics, Macquarie University, Sydney 2109, Australia, vladimir.gaitsgory@mq.edu.au; The work of V. Gaitsgory was supported by the Australian Research Council Discovery Grants DP130104432 and
DP120100532} \and Ludmila Manic \thanks{Department of Mathematics, Macquarie University, NSW 2109, Australia, ludmila.manic@mq.edu.au}
\and
Sergey Rossomakhine
\thanks{Flinders Mathematical Sciences Laboratory, School of Computer Science, Engineering and Mathematics, Flinders University, GPO Box 2100, Adelaide SA 5001, Australia, serguei.rossomakhine@flinders.edu.au; The work of S. Rossomakhine was
supported by the Australian Research Council Discovery Grant
DP120100532}}
\begin{document}

\maketitle
\begin{abstract}
The paper aims at the development of tools  for analysis and construction of near optimal solutions   of singularly
perturbed (SP) optimal controls problems with
 long run average optimality criteria.
The idea that we exploit is to first
asymptotically approximate a given problem of optimal control of the SP system by a certain averaged
optimal control problem, then reformulate this averaged problem as an infinite-dimensional (ID) linear programming (LP) problem,
and then approximate the latter by semi-infinite LP problems. We show that the optimal solution of these semi-infinite LP problems and their duals
(that can be found with the help of a modification of an available LP software) allow one to construct near optimal controls of the SP system. We demonstrate the construction with a numerical example.
\end{abstract}

\bigskip


\bigskip


\bigskip
\bigskip

\section{Introduction and preliminaries}\label{Sec-Contents}

Problems of optimal control of singularly perturbed (SP) systems  have been studied
intensively in both deterministic and stochastic settings (see \cite{Alv-1}, \cite{Art3}, \cite{Ben},
\cite{Bor-Gai}, \cite{Col}, \cite{Dmitriev}, \cite{Dont-Don}, \cite{DonDonSla},  \cite{Fil4}, \cite{Gai8}, \cite{Gra2}, \cite{Kab2}, \cite{Kok1},
\cite{Kus3}, \cite{Lei}, \cite{Nai}, \cite{Plot}, \cite{QW}, \cite{Reo1}, \cite{Vel}, \cite{Vig},
\cite{Yin} for a sample of the literature).
Originally, the most common approaches to SP control systems, especially in the deterministic
case, were related to an approximation of the slow dynamics by the solutions of the systems obtained via
equating of the singular perturbations parameter to zero, with further application of
the boundary layer method (see \cite{Reo}, \cite{Vas}) for an asymptotical description of the fast
dynamics. This type of approaches  were successfully applied to a number of important
classes of problems (see, e.g,
 \cite{Ben}, \cite{DonDonSla}, \cite{Kab2}, \cite{Kok1}, \cite{Kus3}, \cite{Nai},  \cite{Reo1}, \cite{Vel},
 \cite{Yin}).

Various averaging type approaches allowing a consideration of more general classes
of SP problems, in which the optimal and near optimal controls  take the form of rapidly oscillating functions and in which equating of the small parameter to zero does not lead to a right
approximation, were studied in \cite{Alv}, \cite{Alv-1}, \cite{Art2}, \cite{Art0},  \cite{Art3}, \cite{Art-Bright}, \cite{Art1},  \cite{Bor-Gai}, \cite{Bor-Gai-1}, \cite{Bor-Gai-2}, \cite{Don}, \cite{Dont-Don}, \cite{Fil4}, \cite{Filatov},
\cite{Gai0},
 \cite{Gai1}, \cite{Gai8}, \cite{Gai-Leiz}, \cite{Gai-Ng}, \cite{Gra}, \cite{Gra2}, \cite{Plot}, \cite{QW},  \cite{Serea}, \cite{Vig} (see also references therein).
 This research
lead to a good understanding of what the \lq\lq true limit" problems, optimal solutions of which approximate optimal solutions of the SP problems, are. However, till recently, no algorithms
for finding such approximating solutions (in case fast oscillations may lead to a significant improvement of the performance) have been discussed in the literature, and (to the best of our knowledge) first  steps in this direction have been made in the recent publication \cite{GR-long}.

The present paper continues the line of research started in \cite{GR-long}. As in \cite{GR-long}, our consideration is based on earlier results on averaging of SP control systems obtained in  \cite{Gai1}, \cite{Gai8}, \cite{Gai-Leiz}, \cite{Gai-Ng}, \cite{GR} (see also \cite{Art2}, \cite{Art0},  \cite{Art3}, \cite{Art1}, \cite{Gra}, \cite{Gra2}, \cite{QW}) and
  on results obtained in \cite{FinGaiLeb}, \cite{GQ},  \cite{GQ-1},  \cite{GR} that establish the equivalence of optimal control
problems to certain infinite dimensional (ID) linear programming (LP) problems (related results on IDLP formulations of optimal control problems in both deterministic and stochastic settings can be found
in \cite{Borkar1}, \cite{Borkar}, \cite{BhBo}, \cite{BGQ}, \cite{Evans}, \cite{F-V}, \cite{Goreac-Serea}, \cite{Her-Her-Lasserre}, \cite{Kurtz}, \cite{Lass-Trelat},
\cite{Rubio}, \cite{Stockbridge}, \cite{Stockbridge1} and \cite{Vinter}). In contrast to \cite{GR-long}, where mostly optimal control problems with time discounting criteria were dealt with, this paper is devoted to the consideration of the problems with long run  average optimality criteria.

As in  \cite{GR-long}, we, first,
asymptotically approximate a given problem of optimal control of the SP system by a certain averaged
optimal control problem, then reformulate this averaged problem as an IDLP problem,
and then approximate the latter by semi-infinite LP problems. We show that the optimal solution of these semi-infinite LP problems and their duals
(that can be found with the help of a modification of an available LP software) allow one to construct near optimal controls of the SP system.
 Note that, while the approach we exploit  is similar to that of  \cite{GR-long}, the results of this paper are obtained under different assumptions and require a more elaborated  argument than those used in  \cite{GR-long}.

The paper is organized as follows. It consists of eight sections. Section \ref{Sec-Contents} is this introduction. In Section 2, we establish some basic relationships between
the SP and the averaged optimal control problems and their IDLP counterparts (Propositions \ref{Prop-averaging-result} and \ref{Prop-inequalities-equalities}). In Section \ref{Sec-ACG}, the concept of the average control generating (ACG) families for SP problems with long run  average criteria is introduced (Definition \ref{Def-ACG}) and sufficient and necessary conditions for an ACG family to be optimal are established
under the assumption that solutions of the averaged and associated dual problems exist (Proposition \ref{Prop-necessary-opt-cond} and Remark \ref{Remark-nec-opt-cond}). In Section \ref{N-approx-dual-opt}, an approximating averaged semi-infinite LP problem is introduced and it is shown
 that solutions of the corresponding averaged and associated  dual problems exist under natural controllability conditions (Proposition \ref{Prop-existence-disc}). In Section \ref{Sec-ACG-construction}, it is established that, if certain assumptions are satisfied, then
 solutions of the approximating averaged and associated dual problems can be used for the construction of near optimal ACG families (Theorem \ref{Main-SP-Nemeric}). In Section \ref{Sect-asympt-near-opt}, we indicate a way how asymptotically near optimal controls of the SP problems with long run time average criteria can be constructed on the basis of  near optimal ACG families (Theorem \ref{Main-SP-SP-Nemeric}), the construction being illustrated with a numerical example. In Sections \ref{Sec-Selected-proofs} and  \ref{Sec-Selected-proofs-new}, we give proofs of
  Theorem \ref{Main-SP-Nemeric} and  Theorem \ref{Main-SP-SP-Nemeric}.

Let us conclude this section with some notations and definitions. Given a compact metric space $X$, $\mathcal{B}(X)$ will stand for the $\sigma$-algebra of its Borel subsets and $\mathcal{P}(X)$ will denote the set of probability measures defined on $\mathcal{B}(X)$. The set $\mathcal{P}(X)$ will always be treated as a compact metric space with a metric $\rho$, which is  consistent with its weak$^*$  topology. That is, a sequence
              $\nu^k \in \mathcal{P}(X), k =1,2,... ,$
converges to $\nu\in \mathcal{P}(X)$ in this metric if and only if
      $$
          \lim_{k\rightarrow \infty}\int_{X} h(x) \nu^k (dx) \ = \
          \int_{X}h(x) \nu (dx)
      $$
for any continuous $h(x): X \rightarrow \mathbb{R}^1$.
Using this metric $\rho$, one can define the Hausdorff metric $\rho_H$ on the set of closed subsets of $\mathcal{P}(X)$ as follows:
         $\forall \Gamma_i \subset \mathcal{P}(X) \ , \ i=1,2 \ ,$
\vspace{-0.15cm}
\begin{equation}\label{e:intro-3}
               \rho_H(\Gamma_1, \Gamma_2) \stackrel{def}{=} \max \{\sup_{\nu \in \Gamma_1} \rho(\nu,\Gamma_2), \sup_{\nu \in \Gamma_2} \rho(\nu,\Gamma_1)\}, \ \end{equation}
where
      $\rho(\nu, \Gamma_i) \BYDEF \inf_{\nu' \in \Gamma_i} \rho(\nu,\nu') \ .$


Given a measurable function $x(\cdot): [0,\infty) \rightarrow X$, the {\it occupational measure} generated by  this function on the interval $[0,S] $ is the probability measure $\nu^{x(\cdot),S}\in \mathcal{P}(X)$ defined by the equation
\begin{equation}\label{e:occup-S}
              \nu^{x(\cdot),S} (\texttt{B}) \BYDEF \frac{1}{S} \int_0^{S}  1_{\texttt{B}}(x(t))dt , \ \ \forall \texttt{B} \in \mathcal{B}(X), \end{equation}
where $1_{\texttt{B}}(\cdot) $ is the indicator function. The occupational measure  generated by this function on the interval $[0,\infty) $  is the probability measure $\nu^{x(\cdot)}\in \mathcal{P}(X)$ defined as the limit (assumed to exist)
\begin{equation}\label{e:occup-S-infty}
         \nu^{x(\cdot)} (\texttt{B}) \BYDEF \lim_{S\rightarrow\infty}\frac{1}{S} \int_0^{S}  1_{\texttt{B}}(x(t))dt ,  \ \ \forall \texttt{B} \in \mathcal{B}(X). \end{equation}
Note that  (\ref{e:occup-S}) is equivalent to that
\begin{equation}\label{e:oms-0-1}
          \int_{X} h(x) \nu^{x(\cdot),S}(dx) = \frac{1}{S} \int _0 ^ S h (x(t)) dt \end{equation}
for any $h(\cdot)\in C(X)$, and (\ref{e:occup-S-infty}) is equivalent to that
\begin{equation}\label{e:oms-0-1-infy}
         \int_{X} h(x) \nu^{x(\cdot)}(dx) =  \lim_{S\rightarrow\infty}\frac{1}{S} \int _0 ^ S h (x(t)) dt \end{equation}
for any $h(\cdot)\in C(X)$.

\section{Singularly perturbed and averaged optimal control problems and the related IDLP problems}\label{Sec-ACG-nec-opt}
   Consider the SP control system
\begin{eqnarray}\label{e:intro-0-1}
     \epsilon y'(t) &=&f(u(t),y(t),z(t)) ,   \\
              z'(t)&=&g(u(t),y(t),z(t)),   \label{e:intro-0-2} \end{eqnarray}
where  $\epsilon > 0$ is a small parameter; $\ f(\cdot) :U\times
     \mathbb{R}^m \times \mathbb{R}^n \to \mathbb{R}^m, \ g(\cdot):U\times
     \mathbb{R}^m \times \mathbb{R}^n\to \mathbb{R}^n$ are continuous
vector functions satisfying Lipschitz conditions in $z$ and $y$; and where controls $u(\cdot) $ are  measurable functions of time  satisfying  the inclusion
\begin{equation}\label{u-constraint}
       u(t) \in U, \end{equation}
$U$ being a given compact metric space.

Let $Y$ be a given  compact subset of $\mathbb{R}^m$ and $Z $  be a given compact subset of $\mathbb{R}^n$ such that the system (\ref{e:intro-0-1})-(\ref{e:intro-0-2}) is viable in $\ Y\times Z\ $ for any $\epsilon > 0 $ small enough (see the definition of viability in \cite{aubin0}).

 \begin{Definition}\label{Def-adm-SP} Let $u(\cdot)$ be a control and let $(y_{\epsilon}(\cdot), z_{\epsilon}(\cdot))$ be the corresponding solution of the system (\ref{e:intro-0-1})-(\ref{e:intro-0-2}). The triplet $(u(\cdot),y_{\epsilon}(\cdot), z_{\epsilon}(\cdot)) $ will be called admissible  if
\begin{equation}\label{equ-Y}
      (y_{\epsilon}(t),z_{\epsilon}(t))\in Y\times Z  \ \ \ \ \forall t\geq 0.\end{equation}
\end{Definition}
 In this paper we will be dealing with
 the  optimal control problem
  \begin{equation}\label{Vy-perturbed-per-new-1-long-run}
     \displaystyle{ \inf _{(u(\cdot),y_{\epsilon}(\cdot),z_{\epsilon}(\cdot)) }\liminf_{\mathcal{T}\rightarrow\infty}\frac{1}{\mathcal{T}}\int_0 ^ {\mathcal{T}}G(u(t), y_{\epsilon}(t), z_{\epsilon}(t))dt} \stackrel{\rm def}{=}V^*(\epsilon) ,\end{equation}
where $G(\cdot) $ is a continuous function and  $inf$  is over all admissible triplets of the SP system. Note that the initial conditions are not fixed in (\ref{e:intro-0-1})-(\ref{e:intro-0-2}) and they are, in fact, a part of the optimization problem. Note also that, under natural conditions, the optimal value of the problem (\ref{Vy-perturbed-per-new-1-long-run}) is equal to the optimal value of the periodic optimization problem
\begin{equation}\label{Vy-perturbed-per-new-1-long-run-per}
 \displaystyle{ \inf _{(\mathcal{T}, u(\cdot),y_{\epsilon}(\cdot),z_{\epsilon}(\cdot)) }\frac{1}{\mathcal{T}}\int_0 ^ {\mathcal{T}}G(u(t), y_{\epsilon}(t), z_{\epsilon}(t))dt}  ,\end{equation}
 where $inf$ is over the length $\mathcal{T} $ of the time interval and over the admissible triplets that are defined on this interval and that satisfy the periodicity condition:
 $\ (y_{\epsilon}(\mathcal{T}),z_{\epsilon}(\mathcal{T}))=(y_{\epsilon}(0),z_{\epsilon}(0))$. Although the periodic optimization formulation seems to be simpler, a more general statement of the problem  in the form (\ref{Vy-perturbed-per-new-1-long-run}) is more convenient for our consideration.

 The SP optimal control problem (\ref{Vy-perturbed-per-new-1-long-run}) is related to the infinite dimensional linear programming problem
 \begin{equation}\label{equ-Y-per-equal-LP}
 \min_{\gamma\in \mathcal{W}(\epsilon)}\{\int_{U\times Y\times Z}G(u,y,z)\gamma(du,dy,dz)\}\BYDEF G^*(\epsilon),
\end{equation}
where
\begin{equation}\label{Vy-perturbed-per-new-5}
\begin{array}{c}
\displaystyle \mathcal{W}(\epsilon ) \BYDEF  \{\gamma \in \mathcal{P}(U \times Y\times Z) \ : \ \displaystyle \int_{U\times Y\times Z} \nabla (\phi(y)\psi(z) )^T \chi_{\epsilon}(u,y,z)
\gamma(du,dy,dz) =  0
\\[12pt]
 \forall \phi(\cdot)\in C^1 (\R^m), \ \ \forall \psi(\cdot)\in C^1 (\R^n) \},
\end{array}
\end{equation}
with  $\chi_{\epsilon}(u,y,z)^T\BYDEF ( \frac{1}{\epsilon}\ f(u,y,z)^T,
\ g(u,y,z)^T) \ $.
Namely, the optimal values of these two problems are related by the inequality
\begin{equation}\label{equ-Y-per-inequality}
 V^*(\epsilon) \geq G^*(\epsilon) \ \ \ \forall \epsilon > 0,
\end{equation}
and also, under certain conditions (see \cite{FinGaiLeb} and \cite{GR}),
 \begin{equation}\label{equ-Y-per-equal}
 V^*(\epsilon) =G^*(\epsilon) \ \ \ \forall \epsilon > 0.
\end{equation}

Along with the SP system (\ref{e:intro-0-1})-(\ref{e:intro-0-2}), let us
consider a so-called associate system
 \begin{equation}\label{e:intro-0-3}
y'(\tau) = f(u(\tau),y(\tau), z)\ , \ \ \ \ \ \ z =
\mbox{const} .
\end{equation}
Note that the associated system (\ref{e:intro-0-3}) looks similar to the \lq\lq fast" subsystem (\ref{e:intro-0-1}) but, in contrast
to (\ref{e:intro-0-1}), it is evolving in the \lq\lq stretched" time
scale $\tau=\frac{t}{\epsilon}$, with   $z$ being a vector of
fixed parameters. Everywhere in what follows, it is assumed that the associated system is  viable in $Y$.

 \begin{Definition}\label{Def-adm-associate}
 A pair $(u(\cdot), y(\cdot))$ will be called
{\it admissible for the associated system} if
(\ref{e:intro-0-3}) is satisfied for almost all $\tau $ ($u(\cdot) $ being measurable and $y(\cdot) $ being absolutely continuous functions) and if
 \begin{equation}\label{e:intro-0-3-1}
u(\tau)\in U, \ \ \ \ y(\tau)\in Y\ \ \ \ \ \ \ \ \ \forall \tau \geq 0.
\end{equation}
\end{Definition}

Denote by $\mathcal{M}(z,S,y) $ the set of occupational measures generated on the interval $[0,S] $ by the admissible pairs of the
associated system that satisfy the initial
conditions $y(0)=y $. That is,
$$
\mathcal{M}(z,S,y) \BYDEF \bigcup_{(u(\cdot),y(\cdot))}\{ \mu^{(u(\cdot), y(\cdot)), S}\}
\subset \mathcal{P}(U\times Y),
$$
where $\mu^{(u(\cdot), y(\cdot)),S}$ is the occupational measure generated on the interval $[0,S]$ by an admissible pair of the associated system
$(u(\cdot), y(\cdot)) $  satisfying the initial condition $y(0)=y $ and the union is over such admissible pairs.
Also, denote by  $\mathcal{M}(z,S) $ the union of $\mathcal{M}(z,S,y) $ over all $y\in Y $,
$$
\mathcal{M}(z,S)\BYDEF \bigcup_{y\in Y}\{\mathcal{M}(z,S,y)  \}.
$$

In \cite{Gai8} it has been established
 that
 \begin{equation}\label{e:intro-0-3-2-extra}
\limsup_{S\rightarrow\infty}\bar{co}\mathcal{M}(z,S)\ \subset\ W(z)
\end{equation}
 and that,
 under mild conditions,
\begin{equation}\label{e:intro-0-3-2}
\lim_{S\rightarrow\infty}\rho_H(\bar{co}\mathcal{M}(z,S), W(z))=0,
\end{equation}
where $\bar{co}$ stands for the closed convex hull of the corresponding set and $W(z)\subset \mathcal{P}(U \times Y) $ is defined by the equation
\begin{equation}\label{e:2.4}
\begin{array}{c}
\displaystyle
W(z) \BYDEF  \{\mu \in \mathcal{P}(U \times Y) \ : \
\int_{U\times Y} \nabla\phi(y)^{T} f(u,y,z)
\mu(du,dy) =  0 \ \ \forall \phi(\cdot) \in C^1(\R^m) \} \
\end{array}
\end{equation}
(see Theorem 2.1(i) in \cite{Gai8}). Also, it has been established that, under some additional conditions (see Theorem 2.1(ii),(iii) and Proposition 4.1 in \cite{Gai8})),
\begin{equation}\label{e:intro-0-3-3}
\lim_{S\rightarrow\infty}\rho_H(\mathcal{M}(z,S,y), W(z))=0 \ \ \ \forall \ y\in Y,
\end{equation}
with the convergence being uniform with respect to $y\in Y $.

Define the function $\tilde g(\mu,z): \mathcal{P}(U \times Y)\times Z\rightarrow \R^n  $ by the equation
\begin{equation}\label{e:g-tilde}
\tilde{g}(\mu ,z) \BYDEF  \int_{U\times Y} g(u,y,z)
\mu(du,dy) \ \ \ \forall \mu\in \mathcal{P}(U \times Y)
 \end{equation}
and consider the system
\begin{equation}\label{e:intro-0-4}
z'(t)=\tilde{g}(\mu (t),z(t)) ,
\end{equation}
in which the role of controls  is played by measure valued functions $\mu(\cdot)$ that satisfy the inclusion
\begin{equation}\label{e:intro-0-5}
\mu (t) \in W(z(t)) .
\end{equation}

The system (\ref{e:intro-0-4}) will be referred to as {\it the averaged
system}. In what follows, it is assumed that the averaged system is  viable in $Z$.
\begin{Definition}\label{Def-adm-averaged}
A pair $(\mu(\cdot), z(\cdot))$ will be referred to as
{\it admissible for the averaged system} if
(\ref{e:intro-0-4}) and (\ref{e:intro-0-5}) are satisfied for almost all $t$ ($\mu(\cdot) $ being measurable and $z(\cdot) $ being absolutely
continuous functions) and if
 \begin{equation}\label{e:intro-0-3-6}
 z(t)\in Z \ \ \ \ \ \forall t\geq 0.
\end{equation}
\end{Definition}
From Theorem 2.8 of \cite{Gai-Ng} (see also Corollary 3.1 in  \cite{Gai8}) it follows that, under the assumption that (\ref{e:intro-0-3-3})
is satisfied (and under other assumptions including the Lipschitz continuity of the multi-valued map
$V(z)\BYDEF \cup_{\mu\in W(z)}\{\tilde g(\mu,z) \} $),
the averaged system approximates the SP dynamics on the infinite time horizon in the  sense that the following two statements are valid:

{\it (i) Given an admissible triplet $(u(\cdot), y_{\epsilon}(\cdot),z_{\epsilon}(\cdot))$  of the SP system (\ref{e:intro-0-1})-(\ref{e:intro-0-2}) that
satisfies the initial condition
 \begin{equation}\label{e-initial-SP}
 (y_{\epsilon}(0),
z_{\epsilon}(0)) =(y_0,
z_0), \end{equation}
there exists an admissible pair of the averaged system $(\mu(\cdot),z(\cdot)) $  satisfying the initial
condition
\begin{equation}\label{e:intro-0-3-6-1}
 z(0)=z_0
\end{equation}
such that
\begin{equation}\label{e:intro-0-3-7}
 \sup_{t\in [0,\infty)}||z_{\epsilon}(t) -z(t)||\leq \alpha (\epsilon ), \ \ \ \ \ {\rm where} \ \ \ \ \ \lim_{\epsilon\rightarrow 0}\alpha (\epsilon ) = 0
\end{equation}
and, for any Lipschitz continuous functions $h(u,y,z) $,
\begin{equation}\label{e:intro-0-3-8}
 \sup_{\mathcal{T}>0}\ |\frac{1}{\mathcal{T}}\int_0^{\mathcal{T}} h(u(t),y_{\epsilon}(t),z_{\epsilon}(t))dt - \frac{1}{\mathcal{T}}\int_0^{\mathcal{T}}\tilde h(\mu(t),z(t))dt|\ \leq \alpha_h (\epsilon ), \ \ \ \ \ {\rm where} \ \ \ \ \
 \lim_{\epsilon\rightarrow 0}\alpha_h (\epsilon) = 0
\end{equation}
where}
\begin{equation}\label{e:intro-0-3-9}
 \tilde{h}(\mu ,z) \BYDEF  \int_{U\times Y} h(u,y,z)
\mu(du,dy) \ \ \ \forall \mu\in \mathcal{P}(U \times Y).
\end{equation}
{\it (ii) Let $(\mu(\cdot),z(\cdot)) $ be an admissible pair of the averaged system satisfying the initial
condition (\ref{e:intro-0-3-6-1}). There exists an admissible triplet  $(u(\cdot), y_{\epsilon}(\cdot),z_{\epsilon}(\cdot))$
of the SP system satisfying  the initial condition (\ref{e-initial-SP})   such that the estimates
 (\ref{e:intro-0-3-7}) and (\ref{e:intro-0-3-8}) are true.}

 Without going into technical details, let us  introduce the following definition.

\begin{Definition}\label{Def-Average-Approximation} The averaged system will be said to uniformly approximate the SP system
 if the statements (i) and (ii) are valid, with the estimates (\ref{e:intro-0-3-7}) and (\ref{e:intro-0-3-8}) being uniform with respect to the initial conditions $(y_0,z_0)\in Y\times Z $.
\end{Definition}

Consider the optimal control problem
\begin{equation}\label{Vy-ave-opt}
\inf_{(\mu(\cdot),z(\cdot))}\liminf_{\mathcal{T}\rightarrow\infty}\frac{1}{\mathcal{T}}\int_0^\mathcal{T} \tilde G(\mu(t),z(t))dt\BYDEF \tilde{V}^*,
\end{equation}
where
\begin{equation}\label{e:G-tilde}
\tilde{G}(\mu ,z) \BYDEF  \int_{U\times Y} G(u,y,z)
\mu(du,dy)
 \end{equation}
 and where $inf $ is sought over all admissible pairs of the averaged system (\ref{e:intro-0-4}). This will be referred to as {\it averaged optimal control problem}

\begin{Proposition}\label{Prop-averaging-result}
If the averaged system  uniformly approximates the SP system, then
\begin{equation}\label{e-limit-value}
\lim_{\epsilon\rightarrow 0} V^*(\epsilon)= \tilde{V}^*.
\end{equation}
\end{Proposition}

\begin{proof}
The proof follows from the validity of (\ref{e:intro-0-3-8}) (taken  with $h(u,y,z) = G(u,y,z)$).
\end{proof}

The optimal control problem (\ref{Vy-ave-opt}) is related to the infinite dimensional linear programming problem
\begin{equation}\label{e-ave-LP-opt-ave}
 \tilde{G}^*\BYDEF \min_{p\in \tilde{\mathcal{W}}}\int_{F}\tilde G(\mu,z)p(d\mu , dz)
\end{equation}
where
$F $ is the graph of $W(\cdot) $,
\begin{equation}\label{e:graph-w}
F\BYDEF \{(\mu , z) \ : \ \mu\in W(z), \ \ z\in Z \} \ \subset \mathcal{P}(U\times Y)\times Z  ,
 \end{equation}
and
 the set $\tilde{\mathcal{W}}$ is defined by the equation
\begin{equation}\label{D-SP-new}
\tilde{\mathcal{W}}\BYDEF \{ p \in {\cal P} (F): \;
 \int_{F }\nabla\psi(z)^T \tilde{g}(\mu,z) p(d\mu,dz) = 0 \ \ \ \forall \psi(\cdot) \in C^1(R^n)\}.
  \end{equation}
  For convenience, this will be referred to as {\it averaged IDLP problem}.
The relationships between the problems (\ref{Vy-ave-opt}) and (\ref{e-ave-LP-opt-ave}) include, in particular, the inequality between the optimal values
\begin{equation}\label{e-ave-LP-opt-ave-inequality}
\tilde{V}^* \geq \tilde{G}^*,
\end{equation}
which, under certain conditions (see \cite{GQ-1}), takes the form of the equality
\begin{equation}\label{e-ave-LP-opt-ave-equality}
\tilde{V}^* = \tilde{G}^*.
\end{equation}

\begin{Proposition}\label{Prop-inequalities-equalities}
 The following relationships are valid
\begin{equation}\label{lim-inequalities}
\liminf_{\epsilon\rightarrow 0}V^*(\epsilon)\geq \liminf_{\epsilon\rightarrow 0}G^*(\epsilon)\geq \tilde{G}^*.
\end{equation}
If the averaged system uniformly approximates the SP system and if (\ref{e-ave-LP-opt-ave-equality}) is valid, then
\begin{equation}\label{lim-equalities}
\lim_{\epsilon\rightarrow 0}V^*(\epsilon)= \lim_{\epsilon\rightarrow 0}G^*(\epsilon) =\tilde{G}^*.
\end{equation}
\end{Proposition}

Note that the first inequality in (\ref{lim-inequalities}) follows from (\ref{equ-Y-per-inequality}).  The validity of (\ref{lim-equalities}) follows from (\ref{e-limit-value}), (\ref{e-ave-LP-opt-ave-equality}) and the second inequality in (\ref{lim-inequalities}). The latter is proved on the basis of the two lemmas stated below.

Having in mind the fact that an arbitrary probability measure $\gamma\in \mathcal{P}(U\times Y\times Z) $
can be  \lq\lq disintegrated"
  as follows
\begin{equation}\label{e:SP-W-4-extra}
\gamma(du,dy,dz) = \mu(du,dy|z)\nu(dz),
\end{equation}
let us define the set of probability measures
$\mathcal{W}\subset \mathcal{P}(U\times Y\times Z)$  by the equation
\begin{equation}\label{tilde-W-1}
\begin{array}{c}
\mathcal{W}= \{\gamma = \mu(du,dy|z)\nu(dz) \ : \ \mu(\cdot  |z)\in W(z) \ \ for \ \  \nu - almost \ all \ z\in Z,
\\[12pt]
 \int_{Z}\  \nabla \psi(z)^T \tilde g(\mu(\cdot |z),z) \nu(dz)
 = 0 \ \ \forall   \ \psi(\cdot)\in C^1 (\R^n)\},
\end{array}
\end{equation}
where
\begin{equation}\label{e-tilde-g-z}  \tilde g(\mu(\cdot |z),z)= \int_{U\times Y}g(u,y,z)\mu(du,dy |z).
\end{equation}
Note that the  disintegration (\ref{e:SP-W-4-extra})
is understood in the sense that, for any continuous $h(u,y,z)$,
$$
\int_{U\times Y\times Z} h(u,y,z)\gamma (du,dy,dz) = \int_Z \left(\int_{U\times Y}h(u,y,z)\mu(du,dy|z)\right)\ \nu(dz),
$$
where
 $ \ \nu(dz)\BYDEF \gamma(U\times Y,dz)\ $,  $\ \mu(du,dy|z)$ is a probability measure on $U\times Y$  such that the integral\\
$\ \int_{U\times Y}h(u,y,z)\mu(du,dy|z)$ is Borel measurable on $Z$.

\begin{Lemma} The following relationship is valid:
\begin{equation}\label{e-two-sets-1}
\limsup_{\epsilon\rightarrow 0} \mathcal{W}(\epsilon)\subset \mathcal{W}.
\end{equation}
\end{Lemma}
\begin{proof}
The proof follows from Propositions 4.1 and 4.2 in \cite{GR-long}.
\end{proof}

Define  the
map $\Phi(\cdot):\mathcal{P}(F) \rightarrow \mathcal{P}(U \times Y \times Z)$
as follows.
 For any $p \in \mathcal{P}(F) $, let   $\Phi(p)\in \mathcal{P}(U \times Y \times Z)$ be such that
 \begin{equation}\label{e:h&th-1}
  \int_{U\! \times Y \times Z} \!h(u,y,z) \Phi(p) (du,dy,dz) =
\int_{F} \!\tilde h (\mu ,z)
p(d\mu,dz)\ \ \ \ \ \forall h(\cdot)\in \ C(U\times Y\times Z),
\end{equation}
where $\ \tilde h(\mu,z)=\int_{U\times Y}h(u,y,z)\mu(du,dy) $ (this definition is legitimate since the right-hand side of the
above expression defines a linear continuous functional on $C(U\times Y\times Z) $, the latter being associated with an element
of $\mathcal{P}(U \times Y \times Z) $ that makes the equality (\ref{e:h&th-1}) valid).
Note that the map $\Phi(\cdot):\mathcal{P}(F) \rightarrow \mathcal{P}(U \times Y \times Z)$ is linear and it is continuous in the sense that
\begin{equation}\label{e-continuity-Psi}
\lim_{p_l\rightarrow p}\Phi(p_l)=\Phi(p),
\end{equation}
with $p_l$ converging to $p$ in the weak$^*$ topology of $\mathcal{P}(F)$ and $\Phi(p_l)$ converging to $\Phi(p)$
in the weak$^*$ topology of $\mathcal{P}(U \times Y \times Z)$
 (see Lemma 4.3 in \cite{Gai-Ng}).

 \begin{Lemma} The following equality is true:
 \begin{equation}\label{e-two-sets-3}
\mathcal{W} = \Phi (\tilde{\mathcal{W}}).
\end{equation}
\end{Lemma}
\begin{proof}
 The proof follows from Proposition 5.6 in \cite{GR-long}.
\end{proof}

 {\it Proof of Proposition \ref{Prop-inequalities-equalities}. }
As was mentioned above, to prove the proposition, it is sufficient to establish the validity of the second inequality in (\ref{lim-inequalities}). Note that, by (\ref{e-two-sets-1}),
\begin{equation}\label{e-two-sets-2}
\liminf_{\epsilon\rightarrow 0} G^*(\epsilon)\geq \min_{\gamma\in \mathcal{W}}\{\int_{U\times Y\times Z}G(u,y,z)\gamma(du,dy,dz)\}.
\end{equation}
 Also, by (\ref{e-two-sets-3}),
 \begin{equation}\label{e-two-sets-4}
\min_{\gamma\in\mathcal{W}}\{\int_{U\times Y\times Z}G(u,y,z)\gamma(du,dy,dz) \} = \min_{p\in\mathcal{\tilde{W}}}\{\int_{U\times Y\times Z}G(u,y,z)\Phi(p)(du,dy,dz) \} \ = \ \tilde{G}^*
\end{equation}
(the last equality being due to (\ref{e:h&th-1})). By comparing (\ref{e-two-sets-2}) and (\ref{e-two-sets-4}), one obtains the second inequality in (\ref{lim-inequalities}).
\endproof

\section{Average control generating (ACG) families}\label{Sec-ACG}
For any $z\in Z$, let  $(u_z(\cdot),y_z(\cdot))$ be an admissible pair
of  the associated system (\ref{e:intro-0-3})
and  $\mu(du,dy|z)$ be the
occupational measure generated by this pair on $[0,\infty) $ (see (\ref{e:oms-0-1-infy})),
 with the integral
$\int_{U\times Y}h(u,y,z)\mu(du,dy|z) $ being a measurable function of $z$ and
\begin{equation}\label{e-opt-OM-1-0}
 | \frac{1}{S}\int_0^S h(u_{z}(\tau),y_{z}(\tau),z)d\tau - \int_{U\times Y}h(u,y,z)\mu(du,dy|z))|\leq \phi_h(S)\ \ \forall z\in Z, \ \  \ \ \ \ \ \ \
 \lim_{S\rightarrow\infty}\phi_h(S)=0
\end{equation}
 for any
continuous function $h(u,y,z): U\times Y\times Z \rightarrow \R^1$.
Note that, due to (\ref{e:intro-0-3-2-extra}),
\begin{equation}\label{e-opt-OM-1-0-1002}
\ \mu(du,dy|z)\in W(z) \ \forall \ z\in Z.
\end{equation}

\begin{Definition}\label{Def-ACG}
The family $(u_z(\cdot),y_z(\cdot)$ will be called {\it average control generating} (ACG) if
the system
\begin{equation}\label{e-opt-OM-1}
z'(t)=\tilde g_{\mu}(z(t)), \ \ \  \  z(0) = z_0 ,
\end{equation}
where
\begin{equation}\label{e-opt-OM-1-101}
\tilde g_{\mu}(z)\BYDEF \tilde g(\mu(\cdot |z),z)= \int_{U\times Y}g(u,y,z)\mu(du,dy |z),
\end{equation}
has a unique solution  $z(t)\in  Z\ \forall t\in [0,\infty)$ and, for any continuous function $\ \tilde{h}(\mu,z): F \rightarrow \R^1$,
there exists a limit
\begin{equation}\label{Vy-ave-h-def}
\lim_{\mathcal{T}\rightarrow\infty}\frac{1}{\mathcal{T}}\int_0^\mathcal{T} \tilde h(\mu(t),z(t))dt,
\end{equation}
where $\mu(t)\BYDEF \mu(du,dy|z(t)) $.

 \end{Definition}

Note that, according to this definition, if $(u_z(\cdot),y_z(\cdot))$ is an ACG family,
with  $\mu(du,dy|z) $ being the family of occupational measures
 generated by this family, and if $ z(\cdot)$ is the corresponding solution of (\ref{e-opt-OM-1}),
 then the  pair $(\mu(\cdot), z(\cdot)) $, where  $\mu(t)\BYDEF \mu(du,dy|z(t)) $, is an admissible pair of the averaged
 system (for convenience, this admissible pair will also be referred to as one generated by the ACG family). From the fact that the limit
 (\ref{Vy-ave-h-def}) exists for any continuous $\ \tilde{h}(\mu,z)$ it follows that the pair $(\mu(\cdot), z(\cdot)) $ generates the occupational measure
 $p\in \mathcal{P}(F) $ defined by the equation
 \begin{equation}\label{Vy-ave-h-def-averaged-measure-p}
\int_{F}\tilde{h}(\mu,z)p(d\mu,dz) = \lim_{\mathcal{T}\rightarrow\infty}\frac{1}{\mathcal{T}}\int_0^\mathcal{T} \tilde h(\mu(t),z(t))dt\ \ \ \ \forall \tilde{h}(\mu,z)\in C(F) .
\end{equation}
Also, the state trajectory $z(\cdot) $ generates the occupation measure $\nu\in \mathcal{P}(Z) $ defined by the equation
\begin{equation}\label{Vy-ave-h-def-averaged-measure-nu}
\int_{Z} h(z)\nu(dz) = \lim_{\mathcal{T}\rightarrow\infty}\frac{1}{\mathcal{T}}\int_0^\mathcal{T} h(z(t))dt\ \ \ \ \forall h(z)\in C(Z) .
\end{equation}

\begin{Proposition}\label{Prop-clarification-1}
Let  $(u_z(\cdot),y_z(\cdot))$ be an ACG family and let $\mu(du,dy|z) $ and $(\mu(\cdot), z(\cdot)) $ be, respectively, the family of occupational measures
 and the admissible pair of the averaged system generated by this family.  Let $p$ be the occupational measure generated by $(\mu(\cdot), z(\cdot)) $
 and $\nu $  be the occupational measure generated by $z(\cdot) $ (in accordance with (\ref{Vy-ave-h-def-averaged-measure-p}) and (\ref{Vy-ave-h-def-averaged-measure-nu}) respectively).
Then
\begin{equation}\label{e-opt-OM-1-extra-101-per}
p\in \tilde{\mathcal{W}}
\end{equation}
and
\begin{equation}\label{e-opt-OM-2-per}
\Phi(p) =  \mu(du,dy|z)\nu(dz),
\end{equation}
where $\Phi(\cdot) $ is defined by (\ref{e:h&th-1}).

\end{Proposition}
\begin{proof}
For an arbitrary $\ \psi(\cdot)\in C^1 (\R^n)$,
$$
\lim_{\mathcal{T}\rightarrow\infty} \frac{1}{\mathcal{T}}\int_0^\mathcal{T} \nabla \psi(z(t))^T \tilde g(\mu(t),z(t))dt =
\lim_{\mathcal{T}\rightarrow\infty}\frac{1}{\mathcal{T}}(\psi(z(\mathcal{T})) -\psi(z(0))) \ = \ 0.
$$
Hence, by (\ref{Vy-ave-h-def-averaged-measure-p}),
$$
\int_{F}\nabla \psi(z)^T \tilde g(\mu,z)p(d\mu,dz) = 0 \ \ \ \ \psi(\cdot)\in C^1 (\R^n).
$$
The latter implies  (\ref{e-opt-OM-1-extra-101-per}). To prove (\ref{e-opt-OM-2-per}), note that, for an arbitrary
continuous $h(u,y,z) $  and  $\ \tilde h(\mu,z)$ defined in accordance with (\ref{e:intro-0-3-9}), one can write down
$$
\int_{F}\tilde h(\mu,z)p(d\mu , dz)= \lim_{\mathcal{T}\rightarrow\infty} \frac{1}{\mathcal{T}}\int_0^\mathcal{T}\tilde h (\mu(t),z(t))dt
$$
\vspace{-.2in}
$$
= \lim_{\mathcal{T}\rightarrow\infty} \frac{1}{\mathcal{T}}\int_0^\mathcal{T} \left(\int_{U\times Y}h(u,y,z)\mu(du,dy |z(t)) \right) dt = \int_Z \left(\int_{U\times Y}h(u,y,z)\mu(du,dy |z) \right) \nu(dz) .
$$
By the definition of $\Phi(\cdot) $ (see (\ref{e:h&th-1})), the latter implies (\ref{e-opt-OM-2-per}).
\end{proof}

 \begin{Definition}\label{Def-ACG-opt}
 An ACG family  $(u_z(\cdot),y_z(\cdot))$   will be called  optimal  if the
admissible pair $(\mu(\cdot), z(\cdot)) $ generated by this family is optimal in the averaged problem (\ref{Vy-ave-opt}). That is,
\begin{equation}\label{e-opt-OM-2-101}
\lim_{\mathcal{T}\rightarrow\infty}\frac{1}{\mathcal{T}}\int_0^\mathcal{T} \tilde G(\mu(t), z(t))dt \ = \
  \tilde{V}^*.
 \end{equation}
An ACG family  $(u_z(\cdot),y_z(\cdot))$  will be called   {\it $\alpha $-near optimal} ($\alpha >0$)
  if
 \begin{equation}\label{e-opt-OM-2-101-1}
\lim_{\mathcal{T}\rightarrow\infty}\frac{1}{\mathcal{T}}\int_0^\mathcal{T} \tilde G(\mu(t), z(t))dt \ \
 \leq \ \tilde{V}^* + \alpha.
 \end{equation}
 \end{Definition}
\begin{Corollary}\label{Corollary-iff}
Let the equality (\ref{e-ave-LP-opt-ave-equality}) be valid. An ACG family   $(u_z(\cdot),y_z(\cdot))$  generating the admissible pair  $(\mu(\cdot),(z(\cdot)) $ will be optimal if and only if the occupational measure generated by this pair (according to
(\ref{Vy-ave-h-def-averaged-measure-p})) is an optimal solution of the  averaged IDLP problem (\ref{e-ave-LP-opt-ave}).
\end{Corollary}

Let $\tilde H(p,z)$ be the
 Hamiltonian corresponding to the averaged
optimal control problem (\ref{Vy-ave-opt})
\begin{equation}\label{e:H-tilde}
\tilde H(p,z)\BYDEF \min_{\mu\in W(z)} \{\tilde G(\mu, z)+p^T \tilde g (\mu , z)\},
\end{equation}
where $\tilde g(\mu, z) $ and $\tilde G(\mu , z) $ are defined by (\ref{e:g-tilde}) and (\ref{e:G-tilde}).
Consider the problem
\begin{equation}\label{e:DUAL-AVE}
\sup_{\zeta(\cdot)\in C^1} \{\theta : \theta\leq \tilde H(\nabla \zeta(z), z) \ \forall
z\in Z \}= \tilde{G}^*  ,
\end{equation}
where $sup $ is sought over all continuously differentiable functions $\zeta(\cdot):\R^n\rightarrow \R^1 $. Note that
 the optimal value of the  problem (\ref{e:DUAL-AVE}) is equal to the optimal value of the averaged IDLP problem (\ref{e-ave-LP-opt-ave}).
 The former is in fact dual with respect to the later, the equality of the optimal values
 being one of the duality relationships between the two (see Theorem 4.1 in \cite{FinGaiLeb}).
 For brevity, (\ref{e:DUAL-AVE}) will be referred to as just {\it averaged dual problem}.
 Note that the averaged dual problem can be equivalently rewritten in the form
\begin{equation}\label{e:DUAL-AVE-0}
\sup_{\zeta(\cdot)\in C^1(\R^n)} \{\theta : \theta\leq  \tilde G(\mu , z)+ \nabla \zeta(z)^T \tilde g (\mu , z) \ \ \forall
(\mu ,z) \in F \}=\tilde{G}^*, \ \ \ \ \
\end{equation}
where  $F$ is the graph of $W(\cdot) $ (see (\ref{e:graph-w})).
A function $\zeta^*(\cdot)\in C^1 $ will be called a solution of the averaged dual problem  if
\begin{equation}\label{e:DUAL-AVE-sol-1}
\tilde{G}^* \leq \tilde H(\nabla \zeta^* (z), z)\ \ \forall
z\in Z \ ,
\end{equation}
or, equivalently, if
\begin{equation}\label{e:DUAL-AVE-sol-1-1001}
\tilde{G}^* \leq \tilde G(\mu, z)+\nabla \zeta^*(z)^T \tilde g (\mu , z)\ \ \forall
(\mu ,z) \in F \ .
\end{equation}
Note that, if $\zeta^*(\cdot)\in C^1 $ satisfies (\ref{e:DUAL-AVE-sol-1}), then $\zeta^*(\cdot)+ const $ satisfies
(\ref{e:DUAL-AVE-sol-1}) as well.

Assume that a solution of the averaged dual problem  (that is, a functions $\zeta^*(\cdot) $ satisfying
(\ref{e:DUAL-AVE-sol-1})) exists and consider the problem in the right hand side of (\ref{e:H-tilde}) with $p=\nabla \zeta^* (z)$ rewriting it in the form
\begin{equation}\label{e:H-tilde-10-1}
 \min_{\mu\in W(z)} \{\int_{U\times Y}[G(u,y,z)+\nabla \zeta^* (z)^T  g (u,y , z)]\mu(du,dz)\} = \tilde H (\nabla \zeta^* (z),z).
\end{equation}
The latter is an IDLP problem, with the dual of it having
 the form
\begin{equation}\label{e:dec-fast-4}
\ \ \ \ \ \ \sup_{\eta(\cdot)\in C^1(\R^m)} \{\theta : \theta \leq  G(u,y,z)+\nabla \zeta^* (z)^T  g (u,y , z) + \nabla \eta (y)^T  f(u,y,z)
\ \ \forall
(u,y) \in U\times Y \}
\end{equation}
\vspace{-.3in}
$$
=\tilde H (\nabla \zeta^* (z),z),
$$
where $sup $ is sought over all continuously differentiable functions $\eta(\cdot): \R^m\rightarrow \R^1 $.
 The  optimal values of the problems (\ref{e:H-tilde-10-1}) and (\ref{e:dec-fast-4}) are equal,
  this being one of the duality relationships between these two problems
 (see Theorem 4.1 in \cite{FinGaiLeb}). The problem (\ref{e:dec-fast-4}) will be referred to as {\it associated dual problem}.
 A function $\eta^*_z(\cdot)\in C^1(\R^m) $ will be called a solution of the  problem (\ref{e:dec-fast-4}) if
\begin{equation}\label{e:H-tilde-10-3}
   \tilde H (\nabla \zeta^* (z),z)\leq G(u,y,z)+\nabla \zeta^* (z)^T  g (u,y , z) + \nabla \eta^*_z (y)^T  f(u,y,z) \ \ \forall (u,y)\in U\times Y.
\end{equation}
Note that from (\ref{e:DUAL-AVE-sol-1}) and from (\ref{e:H-tilde-10-3}) it follows that
\begin{equation}\label{e:H-tilde-10-33}
  G(u,y,z)+\nabla \zeta^* (z)^T  g (u,y , z) + \nabla \eta^*_z (y)^T  f(u,y,z) \geq \tilde G^* \ \ \forall (u,y,z)\in U\times Y\times Z.
\end{equation}

The following result gives  sufficient and also (under additional periodicity assumptions) necessary  conditions for an ACG family $(u_z(\cdot),y_z(\cdot))$ to be optimal and for the equality (\ref{e-ave-LP-opt-ave-equality}) to be valid.

\begin{Proposition}\label{Prop-necessary-opt-cond}
Let a solution $\zeta^* (z) $ of the averaged dual problem exist and  a solution $\eta^*_z (y) $ of the associated dual problem
exist for any $z\in Z$. Then an ACG family $(u_z(\cdot),y_z(\cdot))$  generating the admissible pair  of the averaged system $(\mu(\cdot), z(\cdot)) $ is optimal and  the equality (\ref{e-ave-LP-opt-ave-equality}) is valid if
 $$
 G(u_{z(t)}(\tau),y_{z(t)}(\tau),z(t))+\nabla \zeta^* (z(t))^T  g (u_{z(t)}(\tau),y_{z(t)}(\tau) , z(t))
 $$
 \vspace{-.2in}
  \begin{equation}\label{e:H-tilde-10-3-1002-1}
 + \nabla \eta^*_{z(t)} (y_{z(t)}(\tau))^T  f(u_{z(t)}(\tau),y_{z(t)}(\tau) , z(t)) \ = \ \tilde G^* \ \ \ \ \ \ \forall \ \tau\in P_t, \ \ \forall \ t\in A,
\end{equation}
for some $P_t\subset \R^1 $ and $A\subset \R^1 $ such that
  \begin{equation}\label{e:H-tilde-10-3-1002-1-1}
 meas\{\R^1\setminus P_t\} = 0 \ \ \ \forall t\in A \ \ \ \ \ {\it and} \ \ \ \ \  meas\{\R^1\setminus A\} = 0.
\end{equation}
Under the additional assumptions that an ACG family $(u_z(\cdot),y_z(\cdot))$ is periodic, that is,
\begin{equation}\label{e:H-tilde-10-3-1001}
(u_z(\tau + T_z),y_z(\tau + T_z)) = (u_z(\tau ),y_z(\tau )) \ \ \ \ \forall \ \tau \geq 0
\end{equation}
for some $\ T_z > 0 $  and that the admissible pair of the averaged system $(\mu(\cdot), z(\cdot)) $ generated by this family is periodic as well, that is,
\begin{equation}\label{e:H-tilde-10-3-1001-1}
(\mu(t+\tilde T),z(t+\tilde T)) = (\mu(t),z(t )) \ \ \ \ \forall \ t \geq 0
\end{equation}
for some $\tilde T > 0 $, the fulfillment of (\ref{e:H-tilde-10-3-1002-1}) is also necessary for  $(u_z(\cdot),y_z(\cdot))$ to be optimal and for the equality (\ref{e-ave-LP-opt-ave-equality}) to be valid.

 \end{Proposition}

 \begin{proof}
 Assume (\ref{e:H-tilde-10-3-1002-1}) is true. Then
   $$
\lim_{S\rightarrow\infty}\frac{1}{S}\int_0^S [G(u_{z(t)}(\tau),y_{z(t)}(\tau),z(t))+\nabla \zeta^* (z(t))^T  g (u_{z(t)}(\tau),y_{z(t)}(\tau) , z(t))
$$
\vspace{-.2in}
$$
+ \ \nabla \eta^*_{z(t)} (y_{z(t)}(\tau))^T  f(u_{z(t)}(\tau),y_{z(t)}(\tau) , z(t))\ ]d\tau
$$
\vspace{-.2in}
 \begin{equation}\label{e:H-tilde-10-3-1002-3}
 = \ \tilde G(\mu(t),z(t)) + \nabla \zeta^* (z(t))^T  \tilde g(\mu(t),z(t))  \ = \ \tilde G^* \ \ \ \forall \ t\in A,
\end{equation}
where it has been taken into account that
$$
\lim_{S\rightarrow\infty}\frac{1}{S}\int_0^S \nabla \eta^*_{z(t)} (y_{z(t)}(\tau))^T  f(u_{z(t)}(\tau),y_{z(t)}(\tau),z(t))d\tau =\lim_{S\rightarrow\infty}\frac{1}{S}  [\ \eta^*_{z(t)}(y_{z(t)}(S)) - \eta^*_{z(t)}(y_{z(t)}(0))] \ = \ 0.
$$
Since
$$
\lim_{\mathcal{T}\rightarrow\infty}\frac{1}{\mathcal{T}}\int_0^\mathcal{T} \nabla \zeta^* (z(t))^T  \tilde g(\mu(t),z(t))dt =
\lim_{\mathcal{T}\rightarrow\infty}\frac{1}{\mathcal{T}}(\zeta^* (z(\mathcal{T}))-\zeta^* (z(0))) = 0,
$$
from (\ref{e:H-tilde-10-3-1002-3}) it follows that
$$
\lim_{\mathcal{T}\rightarrow\infty}\frac{1}{\mathcal{T}}\int_0^\mathcal{T}\tilde G(\mu(t),z(t)) dt = \tilde G^*.
$$
By (\ref{e-ave-LP-opt-ave-inequality}), the latter implies that $(u_z(\cdot), y_z(\cdot)) $ is optimal and that the equality (\ref{e-ave-LP-opt-ave-equality})
is valid.

Let us now prove (assuming that (\ref{e:H-tilde-10-3-1001}) and (\ref{e:H-tilde-10-3-1001-1}) are true) that the fulfillment of (\ref{e:H-tilde-10-3-1002-1}) is necessary  for  an ACG family $(u_z(\cdot), y_z(\cdot)) $ to be optimal and for the equality  (\ref{e-ave-LP-opt-ave-equality}) to be valid.
In fact, let an ACG $(u_z(\cdot), y_z(\cdot)) $  be optimal and let (\ref{e-ave-LP-opt-ave-equality}) be true.
Then
$$
\frac{1}{\tilde T}\int_0^{\tilde T} \tilde G(\mu(t) , z(t))dt = \tilde{G}^*.
$$
Since (by(\ref{e:H-tilde-10-3-1001-1}))
$$
\int_0^{\tilde T} \nabla \zeta^*(z(t))^T \tilde g(\mu(t) , z(t))dt = \zeta^*(z(\tilde T))-\zeta^*(z(0)) = 0,
$$
it follows that
$$
 \frac{1}{\tilde T}\int_0^{\tilde T} [\ \tilde G(\mu(t) , z(t)) + \nabla \zeta^*(z(t))^T \tilde g(\mu(t) , z(t)) - \tilde{G}^* ]dt = 0
$$
and, hence, by (\ref{e:DUAL-AVE-sol-1-1001}),
\begin{equation}\label{e:H-tilde-10-3-1002-5}
 \tilde G(\mu(t) , z(t)) + \nabla \zeta^*(z(t))^T \tilde g(\mu(t) , z(t))= \tilde{G}^*
\end{equation}
for almost all $\ t\in [0,\tilde T]$. Note that (due to the periodicity condition (\ref{e:H-tilde-10-3-1001-1})) the equality
above is also valid for almost all $t\in [0,\infty) $.

Let the set $A$ ($meas\{\R^1\setminus A\} = 0 $) be such that the equality (\ref{e:H-tilde-10-3-1002-5}) is valid and let
$t\in A$. Due to the periodicity condition (\ref{e:H-tilde-10-3-1001}), to prove the required statement it is sufficient to show that the equality
(\ref{e:H-tilde-10-3-1002-1}) is satisfied for almost all $\tau\in [0,T_{z(t)}] $. Assume it is not the case and there exists a set $Q_t\subset [0,T_{z(t)}]$,
with $meas\{Q_t\} >0 $, on which (\ref{e:H-tilde-10-3-1002-1}) is not satisfied, the latter implying (due to (\ref{e:H-tilde-10-33})) that
$$
 G(u_{z(t)}(\tau),y_{z(t)}(\tau),z(t))+\nabla \zeta^* (z(t))^T  g (u_{z(t)}(\tau),y_{z(t)}(\tau) , z(t))
 $$
  \vspace{-.2in}
 $$
 + \nabla \eta^*_{z(t)} (y_{z(t)}(\tau))^T  f(u_{z(t)},y_{z(t)}(\tau) , z(t)) \ > \ \tilde G^*\ \ \ \forall \tau\in Q_t.
$$
From the above inequality and from (\ref{e:H-tilde-10-33}) it follows that
$$
\frac{1}{T_{z(t)}}\int_0^{T_{z(t)}} [\ G(u_{z(t)}(\tau),y_{z(t)}(\tau),z(t))+\nabla \zeta^* (z(t))^T  g (u_{z(t)}(\tau),y_{z(t)}(\tau) , z(t))
$$
  \vspace{-.2in}
  \begin{equation}\label{e:H-tilde-10-3-1002-7}
+ \ \nabla \eta^*_{z(t)} (y_{z(t)}(\tau))^T  f(u_{z(t)},y_{z(t)}(\tau) , z(t))d\tau \ > \ \tilde{G}^*.
\end{equation}
By (\ref{e:H-tilde-10-3-1001}),
\begin{equation}\label{e:H-tilde-10-3-1002-8}
\int_0^{T_{z(t)}}\nabla \eta^*_{z(t)} (y_{z(t)}(\tau))^T  f(u_{z(t)}(\tau),y_{z(t)}(\tau) , z(t))d\tau=
\eta^*_{z(t)} (y_{z(t)}(T_{z(t)})) - \eta^*_{z(t)} (y_{z(t)}(0))\ =\ 0.
\end{equation}
Hence, from (\ref{e:H-tilde-10-3-1002-7}) it follows that
$$
\frac{1}{T_{z(t)}}\int_0^{T_{z(t)}} [\ G(u_{z(t)}(\tau),y_{z(t)}(\tau),z(t))+\nabla \zeta^* (z(t))^T  g (u_{z(t)}(\tau),y_{z(t)}(\tau) , z(t))\ ] d\tau
 \ > \ \tilde{G}^*,
$$
which is equivalent to
$$
\tilde G(\mu(t) , z(t)) + \nabla \zeta^*(z(t))^T \tilde g(\mu(t) , z(t))> \tilde{G}^*.
$$
This contradicts  to the fact that $t$ was chosen to belong to the set $A$ on which (\ref{e:H-tilde-10-3-1002-5}) is satisfied. This completes the proof of the proposition.
 \end{proof}

 \begin{Remark}\label{Remark-nec-opt-cond}
 Note that, due to (\ref{e:H-tilde-10-33}), the validity of  (\ref{e:H-tilde-10-3-1002-1}) implies the validity of the inclusion
 $$
(u_{z(t)}(\tau), y_{z(t)}(\tau), z(t))\in {\rm Argmin}_{(u,y,z)\in U\times Y\times Z}\{  G(u,y,z)+\nabla \zeta^* (z)^T  g (u,y , z)
$$
\vspace{-.2in}
\begin{equation}\label{e:H-tilde-10-3-1002-0}
+ \nabla \eta^*_{z} (y)^T  f(u,y , z)  \} \  \ \ \ \ \ \ \forall \ \tau\in P_t, \ \ \forall \ t\in A
\end{equation}
which, in turn, implies
 $$
u_{z(t)}(\tau)\in {\rm Argmin}_{u\in U}\{  G(u,y_{z(t)}(\tau),z(t))+\nabla \zeta^* (z(t))^T  g (u,y_{z(t)}(\tau) , z(t))
$$
\vspace{-.2in}
\begin{equation}\label{e:H-tilde-10-3-1002}
+ \nabla \eta^*_{z(t)} (y_{z(t)}(\tau))^T  f(u,y_{z(t)}(\tau) , z(t))  \} \ \ \ \ \ \ \forall \ \tau\in P_t, \ \ \forall \ t\in A.
\end{equation}
That is, if the equality (\ref{e-ave-LP-opt-ave-equality}) is valid, then for an ACG family $(u_z(\cdot),y_z(\cdot))  $ satisfying the periodicity conditions  (\ref{e:H-tilde-10-3-1001}) and (\ref{e:H-tilde-10-3-1001-1}) to be optimal, it is necessary that the inclusion (\ref{e:H-tilde-10-3-1002}) is satisfied.
 \end{Remark}

 \section{Approximating averaged semi-infinite LP problem;  conditions for the existence of solutions of the approximating averaged/associated  dual problems}\label{N-approx-dual-opt}
Let $\psi_i(\cdot)\in C^1(\R^n) \ , \ \ i=1,2,...,$ be  a sequence of functions such that any $\zeta(\cdot)\in C^1(\R^n)$ and its gradient are simultaneously approximated by a linear
combination of $\psi_i(\cdot)$ and their gradients. Also, let $\phi_i(\cdot)\in C^1(\R^m) \ , \ \ i=1,2,...,$ be  a sequence of functions such that any $\eta(\cdot)\in C^1(\R^m)$ and its gradient are simultaneously approximated by a linear
combination of $\phi_i(\cdot)$ and their gradients. Examples of such sequences are monomials
 $z_1^{i_1}...z_n^{i_n}$, $i_1,...,i_n =0,1,...$ and, respectively, $y_1^{i_1}...y_m^{i_m}$, $i_1,...,i_m =0,1,...$, with $z_k \ (k=1,...,n)$, and
 $y_l\ (l=1,...,m) $ standing for the components of $z$ and $y$ (see, e.g., \cite{Llavona}).

Let us introduce the following notations:
\begin{equation}\label{e:2.4-M}
W_M(z) \BYDEF  \{\mu \in \mathcal{P}(U \times Y) \ : \
\int_{U\times Y} \nabla\phi_i(y)^{T} f(u,y,z)
\mu(du,dy) =  0, \ \ \ i=1,...,M \} ,
\end{equation}
 \vspace{-.2in}
 \begin{equation}\label{e:graph-w-M}
F_M\BYDEF \{(\mu , z) \ : \ \mu\in W_M(z), \ \ z\in Z \} \ \subset \mathcal{P}(U\times Y)\times Z \ ,
 \end{equation}
 \vspace{-.2in}
 \begin{equation}\label{D-SP-new-MN}
\tilde{\mathcal{W}}^{N,M}\BYDEF \{ p \in {\cal P} (F_M): \;
 \int_{F_M }\nabla\psi_i(z)^T \tilde{g}(\mu,z)  p(d\mu,dz) = 0, \ \ \ i=1,...,N\},
  \end{equation}
(compare with (\ref{e:2.4}), (\ref{e:graph-w}) and (\ref{D-SP-new}), respectively) and let us consider the following semi-infinite  LP problem
\begin{equation}\label{e-ave-LP-opt-di-MN}
\min_{p\in \tilde{\mathcal{W}}^{N,M}}\int_{F_M}\tilde G(\mu,z)p(d\mu , dz)\BYDEF \tilde{G}^{N,M}
\end{equation}
(compare with (\ref{e-ave-LP-opt-ave})).
This problem will be referred to as {\it $(N,M)$-approximating averaged problem}.

It is obvious that
\begin{equation}\label{e:graph-w-M-103-0}
W_1(z)\supset W_2(z) \supset ...\supset W_M(z)\supset ...\supset W(z) \ \ \ \Rightarrow \ \ \ \ F_1\supset F_2 \supset ... \supset F_M\supset ... \supset F.
\end{equation}
Defining the set $\tilde{\mathcal{W}}^{N}$ by the equation
 \begin{equation}\label{D-SP-new-N}
\tilde{\mathcal{W}}^{N}\BYDEF \{ p \in {\cal P} (F): \;
 \int_{F }\nabla\psi_i(z)^T \tilde{g}(\mu,z) p(d\mu,dz) = 0, \ \ \ i=1,...,N\},
  \end{equation}
one can also see that
\begin{equation}\label{e:graph-w-M-103-1}
\tilde{\mathcal{W}}^{N,M}\supset
\tilde{\mathcal{W}}^{N}\supset \tilde{\mathcal{W}}\ \ \ \ \ \ \forall \ N, M =1,2,...
\end{equation}
(with $\mathcal{W}^{N,M}$, $\tilde{\mathcal{W}}^{N}$ and $\tilde{\mathcal{W}}$ being considered as  subsets of
 $\mathcal{P}(\mathcal{P}(U \times Y)\times Z) $), the latter implying, in particular, that
 \begin{equation}\label{e:graph-w-M-103-101}
\tilde{G}^{N,M}\leq \tilde{G}^* \ \ \ \ \forall \ N, M =1,2,...\ .
\end{equation}

It can be readily verified that (see, e.g.,  the proof of Proposition 7 in \cite{GR}) that
\begin{equation}\label{e:graph-w-M-100}
\lim_{M\rightarrow\infty}W_M(z)= W(z),  \ \ \ \ \ \ \lim_{M\rightarrow\infty}F_M= F,
 \end{equation}
where, in the first case, the convergence is in the Hausdorff metric generated by the weak convergence in $\mathcal{P}(U \times Y)$ and,
in the second, it is in the Hausdorff metric generated by the weak$^*$ convergence in $\mathcal{P}(U \times Y)$  and the convergence in $Z$.

\begin{Proposition}\label{Prop-LM-convergence}
The following relationships are valid:
\begin{equation}\label{e:graph-w-M-101}
\lim_{M\rightarrow\infty}\tilde{\mathcal{W}}^{N,M}= \tilde{\mathcal{W}}^{N},\ \ \ \ \ \ \ \lim_{N\rightarrow\infty}\tilde{\mathcal{W}}^{N} = \tilde{\mathcal{W}},
 \end{equation}
 where the convergence in both cases is in Hausdorff metric generated by the  weak$^*$ convergence  in $\mathcal{P}(\mathcal{P}(U \times Y)\times Z)$.
Also,
\begin{equation}\label{e:graph-w-M-102}
\lim_{N\rightarrow\infty} \lim_{M\rightarrow\infty}\tilde{G}^{N,M}= \tilde{G}^{*}.
 \end{equation}
If the optimal solution $p^* $ of the averaged IDLP problem (\ref{e-ave-LP-opt-ave}) is unique, then, for an
 an arbitrary optimal solution $p^{N,M} $ of the $(N,M) $-approximating problem (\ref{e-ave-LP-opt-di-MN}),
 \begin{equation}\label{e:graph-w-M-102-101}
\lim_{N\rightarrow\infty }\limsup_{M\rightarrow\infty}\rho(p^{N,M},p^*) = 0.
 \end{equation}
 \end{Proposition}
\begin{proof}
The proof is  similar to that of Proposition 3.5 in \cite{GR-long}.
\end{proof}

Define the finite dimensional space $\mathcal{Q}_N\subset C^1(\R^n) $ by the equation
\begin{equation}\label{e:HJB-8}
\mathcal{Q}_N\BYDEF \{\zeta(\cdot)\in C^1(\R^n): \zeta(z)= \sum_{i=1}^N \lambda_i \psi_i(z), \ \ \lambda=
(\lambda_i)\in \mathbb{R}^N\}
\end{equation}
 and consider the following problem
\begin{equation}\label{e:DUAL-AVE-0-approx-MN}
\sup_{\zeta(\cdot)\in \mathcal{Q}_N} \{\theta : \theta\leq  \tilde G(\mu , z)+ \nabla \zeta(z)^T \tilde g (\mu , z) \ \ \forall
(\mu ,z) \in F_M \}=\tilde{G}^{N,M}.
\end{equation}
This problem is dual with respect to the problem (\ref{e-ave-LP-opt-di-MN}),
the equality of the optimal values of these two problems being a part of the duality relationships. Note that the problem (\ref{e:DUAL-AVE-0-approx-MN})
looks similar to the averaged dual problem (\ref{e:DUAL-AVE-0}). However, in contrast to the latter, in (\ref{e:DUAL-AVE-0-approx-MN}), $\sup$ is sought over the finite dimensional subspace $\ \mathcal{Q}_N$ of $\ C^1(\R^n)$ and $F_M $ is used instead of $F$.
 The problem (\ref{e:DUAL-AVE-0-approx-MN}) will be referred to as {\it $(N,M) $-approximating averaged dual problem}.
A function $\zeta^{N,M}(\cdot)\in  \mathcal{Q}_N$,
\begin{equation}\label{e:DUAL-AVE-0-approx-2}
\zeta^{N,M}(z) = \sum_{i=1}^N \lambda_i^{N,M} \psi_i(z),
\end{equation}
will be called a solution of the $(N,M)$-approximating  averaged dual problem if
\begin{equation}\label{e:DUAL-AVE-0-approx-2-1}
\tilde{G}^{N,M}\leq  \tilde G(\mu , z)+ \nabla \zeta^{N,M}(z)^T \tilde g (\mu , z) \ \ \forall
(\mu ,z) \in F_M .
\end{equation}

Define the finite dimensional space $\mathcal{V}_M\subset C^1(\R^m) $ by the equation
\begin{equation}\label{e:HJB-8-Associated}
\mathcal{V}_M\BYDEF \{\eta(\cdot)\in C^1(\R^m): \eta(y)= \sum_{i=1}^M \omega_i \phi_i(y), \ \ \omega=
(\omega_i)\in \mathbb{R}^M\}
\end{equation}
 and, assuming that a solution $\zeta^{N,M}(z)$ of the $(N,M) $-approximating averaged dual problem exists,  consider the following problem
\begin{equation}\label{e:dec-fast-4-Associated}
\sup_{\eta(\cdot)\in \mathcal{V}_M} \{\theta : \theta \leq  G(u,y,z)+\nabla \zeta^{N,M} (z)^T  g (u,y , z) + \nabla \eta (y)^T  f(u,y,z) \ \ \forall
(u,y) \in U\times Y \}\BYDEF \sigma^{N,M}(z).
\end{equation}
While the problem (\ref{e:dec-fast-4-Associated}) looks similar to the associated dual problem
 (\ref{e:dec-fast-4}), it
differs from the latter, firstly, by that
$sup $ is sought over the finite dimensional subspace  $\mathcal{V}_M $ of   $C^1(\R^m) $ and, secondly, by that
a solution
$\zeta^{N,M} (z) $ of  (\ref{e:DUAL-AVE-0-approx-MN}) is used instead of a solution $\zeta^* (z) $ of (\ref{e:DUAL-AVE}) (the later
may not exist).  The problem (\ref{e:dec-fast-4-Associated}) will be referred to as
{\it $(N,M)$-approximating  associated dual problem}.
It can be shown that
it is, indeed, dual with respect to the semi-infinite LP problem
\begin{equation}\label{e:dec-fast-4-Associated-1}
\min_{\mu\in W_M(z)}\{\int_{U\times Y}[G(u,y,z)+\nabla \zeta^{N,M} (z)^T  g (u,y , z)]\mu(du,dy)= \sigma^{N,M}(z),
\end{equation}
the duality relationships including the equality of the optimal values (see Theorem 5.2(ii) in \cite{FinGaiLeb}).
A function $\eta^{N,M}_z(\cdot)\in \mathcal{V}_M$,
\begin{equation}\label{e:DUAL-AVE-0-approx-1-Associate-1}
\eta^{N,M}_z(y)= \sum_{i=1}^M \omega_{z,i}^{N,M} \phi_i(y),
\end{equation}
will be called a solution of the $(N,M)$-approximating  associated dual problem if
\begin{equation}\label{e:DUAL-AVE-0-approx-1-Associate-2}
\sigma^{N,M}(z) \leq G(u,y,z)+\nabla \zeta^{N,M} (z)^T  g (u,y , z) + \nabla \eta_z^{N,M} (y)^T  f(u,y,z) \ \ \forall (u,y)\in U\times Y.
\end{equation}

Let us now introduce two controllability type  assumptions under which
 solutions of the\\ $(N,M)$-approximating averaged and associated dual problems exists.

 \begin{Assumption}\label{Ave-associated-local-controllability}
There exists a set $Z^0\subset Z $, the closure of which has a nonempty interior,
$$
int (cl Z^0)\neq \emptyset ,
$$
such that any two points in  $Z^0$ can be connected by an admissible trajectory of the averaged system (that is, for
any $ z', z''\in Z^0$ , there exists an admissible pair $(\mu(\cdot ), z(\cdot))$ of the averaged system defined on some interval
$[0, \mathcal{T}]$ such that $z(0) = z'$ and $z(\mathcal{T}) = z''$).
\end{Assumption}

\begin{Assumption}\label{Ass-associated-local-controllability}
There exists a set $\ Y^0(z)\subset Y $, the closure of which has a nonempty interior,
$$
int (cl Y^0(z))\neq \emptyset ,
$$
such that any two points in  $Y^0 (z)$ can be connected by an admissible trajectory of the associated system (that is, for
any $ y', y''\in Y^0(z)$ , there exists an admissible pair $(u(\cdot ), y(\cdot))$ of the associated system defined on some interval
$[0, S]$ such that $y(0) = y'$ and $y(S) = y''$).
\end{Assumption}

Assume also that, for any $N=1,2,...,$ and $M=1,2,..., $ the gradients $\nabla \psi_i(z), \ i =1,2,... N,$ and $\nabla \phi_i(y), \ i =1,2,... M, $  are linearly independent on
any open subset of $\R^N $ and, respectively, $\R^M $. That is, if $Q $ is an open subset of $\R^N$, then  the equality
$$
\ \sum_{i=1}^N v_i \nabla \psi_i(z) = 0 \ \ \forall z\in Q
$$
is valid only if $v_i=0, \ i=1,...,N $, and, similarly, if $D $ is an open subset of $\R^M$, then the equality
$$
\ \sum_{i=1}^M w_i \nabla \phi_i(y) = 0 \ \forall y\in D
$$
is valid only if $w_i=0, \ i=1,...,M $.

\begin{Proposition}\label{Prop-existence-disc}

(i) If Assumption \ref{Ave-associated-local-controllability} is satisfied, then a solution of the $(N,M) $-approximating averaged
dual problem exists for any $N$ and $M$.

(ii) If Assumption \ref{Ass-associated-local-controllability} is satisfied for any $z\in Z$, then
 a solution  of the
$(N,M)$-approximating  associated dual problem exists for any $N$ and $M$, and for any $z\in Z$.
\end{Proposition}

The proof of the propositions is based on  Lemma 3.10 of  \cite{GR-long} reproduced below

\begin{Lemma}\label{Lemma-Important-simple-duality}
Let $X$ be a compact metric space and let $\Psi_i(\cdot): X\rightarrow \R^1, \ i=0,1,...,K, $ be continuous functional on $X $. Let
\begin{equation}\label{e:Important-simple-duality-Lemma-1}
\sigma^*\BYDEF \sup_{\{\lambda_i\}}\{\theta \ : \ \theta \leq \Psi_0(x) + \sum_{i=1}^{K}\lambda_i \Psi_i(x) \ \forall x\in X\} ,
\end{equation}
where $sup$ is sought over $\lambda\BYDEF \{\lambda_i\}\in \R^K $. A solution of the problem (\ref{e:Important-simple-duality-Lemma-1}), that is
$\lambda^*\BYDEF \{\lambda_i^*\}\in \R^K $ such that
\begin{equation}\label{e:Important-simple-duality-Lemma-2}
\sigma^* \leq \Psi_0(x) + \sum_{i=1}^{K}\lambda_i^* \Psi_i(x) \ \ \forall x\in X
\end{equation}
exists if the inequality
\begin{equation}\label{e:Important-simple-duality-Lemma-3}
0 \leq  \sum_{i=1}^{K}v_i \Psi_i(x) \ \forall x\in X
\end{equation}
is valid only with  $v_i = 0 , \ i=1,...,K $.
\end{Lemma}
\begin{proof}
See the proof of Lemma 3.10 in  \cite{GR-long}.
\end{proof}

{\it Proof of Proposition \ref{Prop-existence-disc}}.
By Lemma \ref{Lemma-Important-simple-duality}, to prove the Proposition \ref{Prop-existence-disc}(i), it is sufficient to show that,
under Assumption \ref{Ave-associated-local-controllability}, the inequality
\begin{equation}\label{e:existence-NM-associated-dual-1}
0\leq \sum_{i=1}^N v_i [\nabla\psi_i(z)^T\tilde{g}(\mu,z)]  \ \ \forall (\mu,z)\in F_M
\end{equation}
can be valid only with $\ v_i=0, \ i=1,...,N$. Let us assume that (\ref{e:existence-NM-associated-dual-1}) is valid
and let us rewrite it in the form
\begin{equation}\label{e:existence-NM-associated-dual-1-1}
0\leq \nabla\psi(z)^T \tilde{g}(\mu,z) \ \ \ \forall (\mu,z)\in F_M , \ \ \ \ \ \ {\rm where} \ \ \ \ \ \  \psi(z)\BYDEF \sum_{i=1}^N v_i \psi_i(z).
\end{equation}
Let $ z', z''\in Z^0$  and let an admissible pair $(\mu(\cdot ), z(\cdot))$ of the averaged system be such that
 $z(0) = z'$ and $z(\mathcal{T}) = z''$ for some $\mathcal{T}>0$. Since $\ (\mu(t ), z(t))\in F \ \ \forall t\in [0,\mathcal{T}]$ and since
 $F\subset F_M $ (see (\ref{e:graph-w-M-103-0})), from
 (\ref{e:existence-NM-associated-dual-1-1}) it follows that
 $$
 \psi(z'') - \psi(z')= \int_0^{\mathcal{T}} \nabla \psi (z(t))^T\tilde{g}(\mu(t),z(t))dt \ \geq \ 0 \ \ \ \ \ \Rightarrow \ \ \ \ \ \psi(z'') \geq \psi(z').
 $$
Since $ z', z''$ can be arbitrary points in $Z^0 $, it follows that
$$
\psi(z) = {\rm const} \ \ \forall z\in Z^0 \ \ \ \ \ \Rightarrow \ \ \ \ \ \psi(z) = {\rm const} \ \ \forall z\in cl Z^0.
$$
The latter implies that
$$
\nabla \psi(z) = \sum_{i=1}^N v_i \nabla \psi_i(z) = 0  \ \ \ \ \forall z\in int(clZ^0),
$$
which, in turn, implies that $v_i= 0, \ i=1,...,N $  (due to linear independence of $\nabla \psi_i(\cdot) $). The statement (i) of the proposition is proved.
The proof of the statement (ii) is similar (see also the proof of Proposition 3.9 in \cite{GR-long}).
\qquad
\endproof

\section{Construction of near optimal ACG families}\label{Sec-ACG-construction}
 Let us assume that, for any $N$ and $ M$, a solution $\zeta^{N,M} (z) $ of the $(N,M)$-approximating averaged dual problem exists  and a solution
  $ \eta^{N,M}_z (y)$ of the
 $(N,M)$-approximating associated problem exists for any $z\in Z$ (as follows from Proposition \ref{Prop-existence-disc},
 these  exist if Assumptions \ref{Ave-associated-local-controllability}
  and \ref{Ass-associated-local-controllability} are satisfied).

 Define a control $u^{N,M}(y,z) $  as an optimal solution of the problem
\begin{equation}\label{e-NM-minimizer-0}
\min_{u\in U}\{G(u,y,z)+\nabla \zeta^{N,M} (z)^T  g (u,y , z) + \nabla \eta^{N,M}_z (y)^T  f(u,y,z)\}.
\end{equation}
That is,
\begin{equation}\label{e-NM-minimizer-1}
u^{N,M}(y,z)=argmin_{u\in U}\{G(u,y,z)+\nabla \zeta^{N,M} (z)^T  g (u,y , z) + \nabla \eta^{N,M}_z (y)^T  f(u,y,z)\}.
\end{equation}
Assume that the system
\begin{equation}\label{e:opt-cond-AVE-1-fast-100-2-NM}
y'_z(\tau) = f(u^{N,M}(y_z(\tau),z),y_z(\tau), z), \ \ \ \ \ y_z(0) = y \in Y,
\end{equation}
has a unique  solution $y_z^{N,M}(\tau)\in Y$. Below, we  introduce
assumptions under which it will be established that
 $(u_z^{N,M}(\cdot),y_z^{N,M}(\cdot)) $, where  $u_z^{N,M}(\tau)  \BYDEF u_z^{N,M}(y_z^{N,M}(\tau),z) $, is
a near optimal ACG family (see Theorem \ref{Main-SP-Nemeric}).

\begin{Assumption}\label{SET-1} The following conditions are satisfied:

(i) The optimal solution $p^*$ of the  IDLP problem (\ref{e-ave-LP-opt-ave}) is unique, and  the equality (\ref{e-ave-LP-opt-ave-equality}) is valid.

(ii)  The optimal solution of the averaged problem (\ref{Vy-ave-opt}) (that is, an admissible pair $(\mu^*(\cdot), z^*(\cdot))$ that delivers minimum in (\ref{Vy-ave-opt})) exists and, for any continuous function $\ \tilde{h}(\mu,z): F \rightarrow \R^1$,
there exists a limit
\begin{equation}\label{Vy-ave-h-def-opt}
\lim_{\mathcal{T}\rightarrow\infty}\frac{1}{\mathcal{T}}\int_0^\mathcal{T} \tilde h(\mu^*(t),z^*(t))dt.
\end{equation}

(iii) For almost all $t\in [0,\infty)$ and any $r>0$, the $p^* $-measure of the  set
$$
\mathcal{B}_r(\mu^*(t), z^*(t)) \BYDEF \{(\mu,z) \ : \ \rho(\mu,\mu^*(t)) + ||z-z^*(t)||< r \}
$$
 is not zero. That is,
\begin{equation}\label{e:convergence-important-3}
p^*(\mathcal{B}_r(\mu^*(t), z^*(t)))> 0.
\end{equation}
\end{Assumption}
Note that from Assumption \ref{SET-1}(ii) it follows that the pair  $\ (\mu^*(\cdot), z^*(\cdot))$ generates an occupational measure
and from Assumptions \ref{SET-1}(i) it follows that this measure coincides with $p^*$ (see Corollary \ref{Corollary-iff}). That is,
\begin{equation}\label{Vy-ave-h-def-opt-1}
\lim_{\mathcal{T}\rightarrow\infty}\frac{1}{\mathcal{T}}\int_0^{\mathcal{T}} \tilde h(\mu^*(t),z^*(t))dt= \int_F \tilde h(\mu,z)p^*(d\mu,dz).
\end{equation}
The following statement gives sufficient conditions for the validity of
 Assumption \ref{SET-1}(iii).

\begin{Proposition}\label{error-fixed-1}
Let Assumptions \ref{SET-1}(i) and \ref{SET-1}(ii) be satisfied. Then Assumption \ref{SET-1}(iii)
will be satisfied if the pair $\ (\mu^*(\cdot), z^*(\cdot))$ is $\tilde{\mathcal{T}}$-periodic ($\tilde{\mathcal{T}}$ is some positive number) and if
$\ \mu^*(\cdot)$ is piecewise continuous on $[0,\tilde{\mathcal{T}}]$.
\end{Proposition}

\begin{proof}
Let $t$ be a continuity point of $\mu^*(\cdot) $. Due to the assumed periodicity of the pair $\ (\mu^*(\cdot), z^*(\cdot))$,
$$
\frac{1}{\tilde{\mathcal{T}}}\int_0^{\tilde{\mathcal{T}}} \tilde h(\mu^*(t),z^*(t))dt= \int_F \tilde h(\mu,z)p^*(d\mu,dz)
$$
and
\begin{equation}\label{Vy-ave-h-def-opt-2}
p^*(\mathcal{B}_r(\mu^*(t),z^*(t))) = \frac{1}{\tilde{\mathcal{T}}}\ meas \{t': \ t'\in [0,\tilde{\mathcal{T}}], \ \ (\mu^*(t'),z^*(t'))\in \mathcal{B}_r(\mu^*(t),z^*(t))\} .
\end{equation}
Since $t$ is a continuity point of $\mu^*(\cdot) $ and since $z^*(\cdot)$ is continuous, there exists $\alpha >0$ such that
$(\mu^*(t'), z^*(t'))\in \mathcal{B}_r(\mu^*(t),  z^*(t))\ \ \forall t'\in [t-\alpha, t+\alpha] $. Hence, the right-hand-side in (\ref{Vy-ave-h-def-opt-2}) is greater than $\frac{2\alpha}{\tilde{\mathcal{T}}} $. This proves the required statement as the number of discontinuity points of $\mu^*(\cdot) $ is finite (due to the assumed piecewise continuity).
\end{proof}

\begin{Assumption}\label{SET-1-1}

(i) For almost all $t\in [0,\infty) $, there exists
an  admissible pair $(u_t^{*}(\tau),y_t^{*}(\tau)) $ of the associated system (considered with $z=z^*(t)$) such that
 $\mu^*(t) $ is the occupational measure generated by this pair on the interval $[0,\infty) $. That is,  for any
continuous $h(u,y) $,
\begin{equation}\label{e-opt-OM-1-0-NM-star}
\lim_{S\rightarrow\infty} S^{-1}\int_0^S h(u_t^{*}(\tau),y_t^{*}(\tau))d\tau = \int_{U\times Y}h(u,y)\mu^*(t)(du,dy).
 \end{equation}
(ii) For almost all $t\in [0,\infty) $, for almost all $\tau\in [0,\infty) $ and for any $r>0$, the $\mu^*(t) $-measure
of the set
$$
\ B_r(u_t^{*}(\tau),y_t^{*}(\tau))\BYDEF \{(u,y)\ : \ ||u-u_t^{*}(\tau)|| + ||y - y_t^{*}(\tau)||< r\}
$$
 is not zero. That is,
\begin{equation}\label{e:convergence-important-7}
\mu^*(t)(B_r(u_t^{*}(\tau),y_t^{*}(\tau))) > 0.
\end{equation}

\end{Assumption}

The following proposition gives sufficient conditions for the validity of Assumption \ref{SET-1-1}(ii).

\begin{Proposition}\label{error-fixed-2}
Let Assumptions \ref{SET-1-1}(i) be valid. Then Assumption \ref{SET-1-1}(ii)
will be satisfied if, for almost all $t\in [0,\infty) $,  the pair $\ (u^*_t(\tau), y^*_t(\tau))$ is $T_t$-periodic ($T_t$ is some positive number) and if
$\ u^*(\cdot)$ is piecewise continuous on $[0,T_t]$.
\end{Proposition}
\begin{proof}
The proof is similar to that of Proposition \ref{error-fixed-1}
\end{proof}

\begin{Assumption}\label{SET-2-0}

(i) The pair $(u_z^{N,M}(\tau),y_z^{N,M}(\tau)) $, where $y_z^{N,M}(\tau)$ is the solution of (\ref{e:opt-cond-AVE-1-fast-100-2-NM})
and $u_z^{N,M}(\tau)  = u_z^{N,M}(y_z^{N,M}(\tau),z) $ is an ACG family that
 generates  the occupational measure
$\mu^{N,M}(du,dy|z)$ on the interval $[0,\infty) $, the latter being independent of the initial conditions $y_z^{N,M}(0)=y$ for $y$ in a neighbourhood of $y_t^{*}(\cdot) $. Also, for any
continuous $h(u,y,z): U\times Y \times Z\rightarrow \R^1 $,
\begin{equation}\label{e-opt-OM-1-0-NM}
 | S^{-1}\int_0^S h(u_{z}^{N,M}(\tau),y_{z}^{N,M}(\tau),z)d\tau - \int_{U\times Y}h(u,y,z)\mu^{N,M}(du,dy|z))|\leq \phi_h(S)\ \ \forall z\in Z, \ \  \ \ \ \ \ \ \
 \lim_{S\rightarrow\infty}\phi_h(S)= 0.
\end{equation}

(ii) The admissible pair of the averaged system $(\mu^{N,M}(\cdot), z^{N,M}(\cdot)) $ generated by $(u_z^{N,M}(\cdot),y_z^{N,M}(\cdot)) $ generates the occupational measure $\lambda^{N,M}\in \mathcal{P}(F) $, the latter being independent of the initial conditions $\ z^{N,M}(0)=z $ for $z$ in a neighbourhood of $z^*(\cdot) $. Also, for any continuous function $\ \tilde{h}(\mu,z): F\rightarrow \R^1$,
 \begin{equation}\label{Vy-ave-h-def-averaged-measure-p-NM}
 |\lim_{\mathcal{T}\rightarrow\infty}\frac{1}{\mathcal{T}}\int_0^\mathcal{T} \tilde h(\mu^{N,M}(t),z^{N,M}(t))dt- \int_{F}\tilde{h}(\mu,z)\lambda^{N,M}(d\mu,dz)|\ \leq \ \phi_{\tilde h}(\mathcal{T}), \ \ \ \ \ \ \ \lim_{\mathcal{T}\rightarrow\infty}\phi_{\tilde h}(\mathcal{T}) = 0.
\end{equation}

\end{Assumption}

To state our next assumption, let us re-denote the occupational measure $\mu^{N,M}(du,dy|z)$ (introduced in Assumption \ref{SET-2-0} above)
as $\mu^{N,M}(z) $ (that is, $\ \mu^{N,M}(du,dy|z)=\mu^{N,M}(z) $).

\begin{Assumption}\label{SET-2} For almost all $t\in [0,\infty)$, there exists
  an open ball  $Q_t\subset \mathbb{R}^n $ centered at $z^*(t) $ such that:

(i) The occupational measure $\mu^{N,M}(z) $   is continuous on  $ Q_t$.
Namely, for any $z', z''\in  Q_t$,
\begin{equation}\label{e:HJB-1-17-1-per-cont-mu-NM-g-1}
  \rho(\mu^{N,M}(z'), \mu^{N,M}(z''))\leq \kappa(||z'-z''||),
  \end{equation}
   where $\kappa(\theta) $ is a function tending to zero when $\theta $ tends to zero ($\lim_{\theta\rightarrow 0}\kappa   (\theta)=0$). Also,
 for any $z', z''\in  Q_t$,
       \begin{equation}\label{e:HJB-1-17-1-per-Lip-g}
||\int_{U\times Y} g(u,y,z')\mu^{N,M}(z')(du,dy) -\int_{U\times Y} g(u,y,z'')\mu^{N,M}(z'')(du,dy) ||\leq L ||z'-z''||,
\end{equation}
where $L$ is a constant.

(ii) Let  $z^{N,M}(\cdot)$ be the solution of the system
     \begin{equation}\label{e-opt-OM-1-MN}
z'(t)=\tilde{g}(\mu^{N,M}(z(t)), z(t)) \ , \ \ \  \  z(0) = z^*(0).
\end{equation}
We assume that, for any $t > 0 $,
 \begin{equation}\label{e:HJB-1-17-1-N}
\lim_{N\rightarrow\infty}\limsup_{M\rightarrow\infty }meas \{A_t(N,M)\}=0,
\end{equation}
 where
 \begin{equation}\label{e:intro-0-4-N-A-t}
 A_t(N,M)\BYDEF \{t'\in [0,t] \ : \ z^{N,M}(t')\notin Q_{t'}\}.
 \end{equation}
\end{Assumption}

In addition  to the assumptions above, let us also introduce

\begin{Assumption}\label{SET-3}
For each $t\in [0,\infty)$ such that
  $Q_t\neq \emptyset $, the following conditions are satisfied:

(i) For almost all $\tau\in [0,\infty ) $, there exists an open ball $B_{t,\tau}\subset \mathbb{R}^m $ centered at $y_t^{*}(\tau) $
such that $u^{N,M}(y,z)  $ is uniquely defined (the problem (\ref{e-NM-minimizer-0}) has a unique solution) for $(y,z)\in B_{t,\tau}\times Q_t  $.

(ii) The function $u^{N,M}(y,z)  $
 satisfies Lipschitz conditions on $B_{t,\tau}\times Q_t$.  That is,
 \begin{equation}\label{e:HJB-1-17-0-per-Lip-u-NM}
||u^{N,M}(y',z')- u^{N,M}(y'',z'')||\leq L ( ||y'-y''|| + ||z'-z''||) \ \ \ \ \ \  \forall (y',z'), (y'',z'')\in B_{t,\tau}\times Q_t,
\end{equation}
where $L$ is a constant.

(iii) Let  $ y^{N,M}_{t} (\tau)\BYDEF y^{N,M}_{z^*(t)} (\tau)$ be the solution of the system  (\ref{e:opt-cond-AVE-1-fast-100-2-NM})
considered with $z=z^*(t) $ and with the initial condition
$y_z(0)=y_t^*(0) $. We assume that, for any $\tau >0 $,
\begin{equation}\label{e:HJB-1-17-1-N-z-const}
\lim_{N\rightarrow\infty}\limsup_{M\rightarrow\infty }meas \{P_{t,\tau}(N,M)\}=0,
\end{equation}
where
\begin{equation}\label{e:HJB-1-17-1-N-z-const-def}
P_{t,\tau}(N,M)\BYDEF \{\tau'\in [0,\tau] \ : \  y^{N,M}_{t}(\tau')\notin B_{t, \tau'}\}.
\end{equation}

\end{Assumption}

\begin{Theorem}\label{Main-SP-Nemeric}
Let Assumptions \ref{SET-1}, \ref{SET-1-1}, \ref{SET-2-0}, \ref{SET-2} and \ref{SET-3} be satisfied.
Then
the family $(u^{N,M}_z(\cdot),  y^{N,M}_z(\cdot)) $ introduced in Assumption  \ref{SET-2-0}(i) is
a $\beta(N,M) $-
near optimal ACG family, with
\begin{equation}\label{e:HJB-1-17-1-N-z-const-def-near-opt}
\lim_{N\rightarrow\infty}\limsup_{M\rightarrow\infty}\beta(N,M) = 0.
\end{equation}
Also,
\begin{equation}\label{e:HJB-1-17-1-N-z-const-def-near-opt-new}
\lim_{N\rightarrow\infty}\limsup_{M\rightarrow\infty}\rho(\lambda^{N,M},p^*) = 0,
\end{equation}
where $\lambda^{N,M} $ is defined by (\ref{Vy-ave-h-def-averaged-measure-p-NM}).
\end{Theorem}

\begin{proof} The proof is given in Section \ref{Sec-Selected-proofs}. It is based on Lemma \ref{fast-convergence} stated at the end
of this section.    Note  that in
the process of the proof of the theorem it is established that
\begin{equation}\label{e:HJB-1-19-2-NM}
\lim_{N\rightarrow\infty}\limsup_{M\rightarrow\infty}\max_{t'\in [0,t]} ||z^{N,M}(t')- z^*(t')|| = 0 \ \ \ \forall t\in [0,\infty),
\end{equation}
where $z^{N,M}(\cdot) $ is the solution of (\ref{e-opt-OM-1-MN}).
Also, it is shown that
  \begin{equation}\label{e:HJB-1-19-1-NM}\  \lim_{N\rightarrow\infty}\limsup_{M\rightarrow\infty}\rho (\mu^{N,M}(z^{N,M}(t)),\mu^*(t)) = 0   \end{equation}
  for almost all  $t\in [0,\infty) $,
and
 \begin{equation}\label{e:HJB-19-NM}
\lim_{N\rightarrow\infty}\limsup_{M\rightarrow\infty}|\tilde V^{N,M}- \tilde G^*| = 0,
\end{equation}
where
\begin{equation}\label{e:HJB-16-NM}
\tilde V^{N,M} \BYDEF \lim_{\mathcal{T}\rightarrow\infty}\frac{1}{\mathcal{T}}\int_0^\mathcal{T}\tilde G(\mu^{N,M}(z^{N,M}(t)),
z^{N,M}(t))dt .
\end{equation}
The relationship  (\ref{e:HJB-19-NM}) implies the statement of the theorem
 with
\begin{equation}\label{e:HJB-16-NM-beta}
\beta(N,M)\BYDEF \tilde V^{N,M}-  \tilde G^*
\end{equation}
(see Definition \ref{Def-ACG-opt}).
 \end{proof}

\begin{Lemma}\label{fast-convergence}
Let the assumptions of Theorem \ref{Main-SP-Nemeric} be satisfied and
let $t\in [0,\infty) $ be such that $Q_t\neq\emptyset $. Then
\begin{equation}\label{e:HJB-1-17-1-N-z-const-fast-1}
\lim_{N\rightarrow\infty}\limsup_{M\rightarrow\infty }\max_{\tau'\in [0,\tau]}| y^{N,M}_{t} (\tau') - y_t^{*}(\tau') ||=0
\ \ \ \forall \tau \in [0,\infty).
\end{equation}
Also,
\begin{equation}\label{e:HJB-1-17-1-N-z-const-fast-2}
\lim_{N\rightarrow\infty}\limsup_{M\rightarrow\infty }||u^{N,M}(y^{N,M}_{t} (\tau), z^*(t) ) - u_t^{*}(\tau) ||=0
\end{equation}
for almost all $\tau\in [0,\infty) $.
\end{Lemma}
\begin{proof}
The proof is given in Section  \ref{Sec-Selected-proofs}.
\end{proof}
\begin{Remark}
Note that from (\ref{e:graph-w-M-102-101}) and (\ref{e:HJB-1-17-1-N-z-const-def-near-opt-new}) it follows that
\begin{equation}\label{e:HJB-1-17-1-N-z-const-def-near-opt-new-1}
\lim_{N\rightarrow\infty}\limsup_{M\rightarrow\infty}\rho(\lambda^{N,M},p^{N,M}) = 0,
\end{equation}
where $p^{N,M}$ is an arbitrary  optimal solution of the $(N,M)$-approximating averaged problem (\ref{e-ave-LP-opt-di-MN}).
\end{Remark}

\section{Construction of asymptotically near optimal controls of the SP problem}\label{Sect-asympt-near-opt}
In the previous section we discussed a construction of near optimal ACG families. In this section we will discuss a way how the latter can be used for  a construction
of asymptotically near optimal controls of the SP problem.

\begin{Definition}\label{Def-asympt-opt}
A control $u_{\epsilon}(\cdot) $  will be called asymptotically  $\alpha$-near optimal ($\alpha>0$)  in the SP problem  (\ref{Vy-perturbed-per-new-1-long-run})
 if the solution $(y_{\epsilon}(\cdot),z_{\epsilon}(\cdot)) $ of the system (\ref{e:intro-0-1})-(\ref{e:intro-0-2}) obtained with this control satisfies (\ref{equ-Y})
 (that is the triplet $(u_{\epsilon}(\cdot), y_{\epsilon}(\cdot),z_{\epsilon}(\cdot)) $ is admissible) and if
\begin{equation}\label{e-opt-convergence-1}
\lim_{\epsilon\rightarrow 0}\liminf_{\mathcal{T}\rightarrow\infty}\frac{1}{\mathcal{T}}\int_0 ^ {\mathcal{T}}G(u_{\epsilon}(t), y_{\epsilon}(t), z_{\epsilon}(t))dt\leq
\liminf_{\epsilon\rightarrow 0}V^*(\epsilon)+ \alpha .
\end{equation}

\end{Definition}

For simplicity, we will be dealing with a special case when
$f(u,y,z) = f(u,y) $. That is, the right hand side in (\ref{e:intro-0-1}) is independent of $z$ (the SP systems that have such a property are called \lq\lq weakly coupled"). Note that in this case the set $W(z)$ defined in (\ref{e:2.4}) does not depend on $z$ too. That is, $W(z)=W$.

Let us also introduce the following  assumptions about the functions $f(u,y)$ and $\tilde g(\mu,z)$.

\begin{Assumption}\label{Ass-stability-1}

(i) There exists a positive definite matrix $A_1$
 such that  its eigenvalues are greater than a positive constant on $Z$ and such that
\begin{equation}\label{Case 1-1}
(f(u, y')- f(u, y''))^T A_1 (y'-y'')
\end{equation}
\vspace{-.3in}
$$
\leq -(y'-y'')^T(y'-y'') \ \ \ \forall y',y''\in \R^m \ , \
 \forall u\in U .
$$

(ii) There exists a positive definite matrix $A_2$
 such that  its eigenvalues are greater than a positive constant on $Z$ and such that
\begin{equation}\label{Case 1-1-1}
(\tilde g(\mu,z')- \tilde g(\mu,z'')^T A_2 (z'-z'')
\end{equation}
\vspace{-.3in}
$$
\leq -(z'-z'')^T(z'-z'') \ \ \ \forall z',z''\in \R^n \ , \
 \forall \mu\in W .
$$
\end{Assumption}

Note that these are Liapunov type stability conditions and, as has been established in \cite{Gai-Ng}, their fulfillment is sufficient for the validity of the statement that the SP system is  uniformly approximated
by the averaged system (see Definition \ref{Def-Average-Approximation}). Also, as can be readily verified, Assumption \ref{Ass-stability-1}(i) implies
that the solutions $y(\tau, u(\cdot), y_1)$ and $y(\tau, u(\cdot), y_2)$ of
the associated system (\ref{e:intro-0-3})  obtained with an arbitrary control $u(\cdot)$
and with  initial values $y(0)= y_1 $ and  $ y(0)= y_2$ ($ y_1 $ and  $ y_2 $ being arbitrary vectors in $Y$) satisfy the inequality
\begin{equation}\label{e-opt-OM-2-106}
||y(\tau, u(\cdot), y_1) - y(\tau, u(\cdot), y_2)||\leq c_1e^{-c_2\tau}||y_1-y_2||,
\end{equation}
where $c_1, c_2$ are some positive constants. Similarly, Assumption \ref{Ass-stability-1}(ii) implies that the solutions
$z(t, \mu(\cdot), z_1)$ and $z(t, \mu(\cdot), z_2)$ of the averaged
system (\ref{e:intro-0-4}) obtained with an arbitrary control $\mu(\cdot)$
and with  initial values $z(0)= z_1 $ and  $ z(0)= z_2$ ($ z_1 $ and  $ z_2 $ being arbitrary vectors in $Z$) satisfy the inequality
\begin{equation}\label{e-opt-OM-2-106-z}
||z(t, \mu(\cdot), z_1) - z(t, \mu(\cdot), z_2)||\leq c_3e^{-c_4 t}||z_1-z_2||,
\end{equation}
where $c_3, c_4$ are some positive constants.

From the validity of (\ref{e-opt-OM-2-106}) and (\ref{e-opt-OM-2-106-z}) it follows that the associated system (\ref{e:intro-0-3}) and the averaged system
(\ref{e:intro-0-4}) have unique forward invariant sets which also are global attractors for the solutions of these systems (see Theorem 3.1(ii) in \cite{Gai1}). For simplicity, we will assume that $Y$ and $Z$ are these sets.


Let $\ (u_z^{N,M}(\tau),y_z^{N,M}(\tau)) $ be the ACG family introduced in Assumptions \ref{SET-2-0}(i)
and let $\ \ \mu^{N,M}(du,dy|z)=\mu^{N,M}(z) $, $\ \ z^{N,M}(t)$ and  $\ \ \mu^{N,M}(z^{N,M}(t))$  be generated by this family as assumed in Section \ref{Sec-ACG-construction} (all the assumptions made in that section are supposed to be satisfied in the consideration below). Let $y_z^{N,M}(\tau , y) $
stand for the solution of the associated system (\ref{e:intro-0-3}) obtained with the control $u_z^{N,M}(\tau)$ and with the initial condition $y_z^{N,M}(0 , y)=y\in Y $. From (\ref{e-opt-OM-2-106}) it follows that
$$
||y_z^{N,M}(\tau , y)-y_{z}^{N,M}(\tau)||\leq c_1e^{-c_2\tau}\max_{y',y''\in Y}||y'-y''||.
$$
The latter implies that, for any Lipschitz continuous function $h(u,y,z)$, there exists $\bar{\phi}_h(S), \ \lim_{S\rightarrow\infty}\bar{\phi}_h(S)=0,  \  $ such that
\begin{equation}\label{e-opt-OM-2-106-1}
|S^{-1}\int_0^S h(u^{N,M}_{z}(\tau),y_z^{N,M}(\tau , y),z)d\tau - S^{-1}\int_0^S h(u^{N,M}_{z}(\tau),y_z^{N,M}(\tau ),z)d\tau | \leq \bar{\phi}_h(S),
\end{equation}
which, due to (\ref{e-opt-OM-1-0-NM}), implies that
\begin{equation}\label{e-opt-OM-2-106-1-0}
|S^{-1}\int_0^S h(u^{N,M}_{z}(\tau),y_z^{N,M}(\tau , y),z)d\tau -  \int_{U\times Y}h(u,y,z)\mu^{N,M}(du,dy|z))|\leq \phi_h(S) + \bar{\phi}_h(S)\BYDEF
\bar{\bar{\phi}}_h(S),
\end{equation}
with $ \ \lim_{S\rightarrow\infty}\bar{\bar{\phi}}_h(S)=0. $
Hence,
\begin{equation}\label{e-inf-hor-1}
\lim_{S\rightarrow\infty}\rho(\mu^{N,M} (S,y), \mu^{N,M}(z))
\ \leq \ \phi(S), \ \ \ \ \ \ \lim_{S\rightarrow\infty}{\phi}(S)=0,
\end{equation}
where $\mu^{N,M} (S,y) $ is the occupational measure generated by the pair $(u^{N,M}_{z}(\tau), y_z^{N,M}(\tau , y)) $ on the interval $ [0,S]$.
That is, the family of measures $\mu^{N,M}(z)$ is uniformly attainable by the associated system with the use of the control $u^{N,M}_{z}(\tau)$ (see Definition 4.3 in \cite{GR-long}).

Partition the interval $[0,\infty) $ by the points
\begin{equation}\label{e:contr-rev-100-1}
t_l = l\Delta(\epsilon), \ l=0,1,....,
\end{equation}
where $\ \Delta(\epsilon)> 0 $ is such that
 \begin{equation}\label{e:contr-rev-100-2}
\lim_{\epsilon\rightarrow 0}\Delta(\epsilon)=0, \ \ \ \ \  \ \ \ \lim_{\epsilon\rightarrow 0}\frac{\Delta(\epsilon)}{\epsilon}=\infty .
\end{equation}
Define the control $u_{\epsilon}^{N,M}(t) $   by the equation
\begin{equation}\label{e:contr-rev-100-3}
u^{N,M}_{\epsilon}(t) \BYDEF u^{N,M}_{z^{N,M}(t_l)}(\frac{t-t_l}{\epsilon}) \ \ \ \forall \ t\in [t_l,t_{1+1}), \ \ \ \ l=0,1,... \ .
\end{equation}

\begin{Theorem}\label{Main-SP-SP-Nemeric}
Let the assumptions of Theorem \ref{Main-SP-Nemeric} be satisfied and let the function
\begin{equation}\label{e:contr-rev-100-3-1001}
\mu^{N,M}(t) \BYDEF \mu^{N,M}(z^{N,M}(t))
\end{equation}
has the following piecewise continuity property: for any $\mathcal{T} > 0 $, there may exist no more than a finite number of points
$\ \mathcal{T}_i\in (0,\mathcal{T}), \ i=1,...k,$ , with
\begin{equation}\label{e:contr-rev-100-3-1002-11}
 \ k\leq c\mathcal{T}, \ \ \ \ \ \ c={\rm const},
\end{equation}
 such that,
for any $t\neq \mathcal{T}_i $,
\begin{equation}\label{e:contr-rev-100-3-1002}
\max \{||\tilde g(\mu^{N,M}(t'),z) - \tilde g(\mu^{N,M}(t),z)||, \ |\tilde G(\mu^{N,M}(t'),z) - \tilde G(\mu^{N,M}(t),z)|\} \leq \nu(t-t')\ \ \ \ \forall \ t'\in (t-a_t, t+a_t)
\end{equation}
where $\nu(\cdot) $ is monotone decreasing, with $ \ \lim_{\theta\rightarrow 0} \nu(\theta) = 0$, and where $a_t>0$, with $r_{\delta}$,
 $$
 r_{\delta} \BYDEF \inf \{ a_t \ : \ t\notin\cup_{i=1}^k ( \mathcal{T}_i - \delta ,\mathcal{T}_i+ \delta )\},
 $$
being a positive
continuous function of $\delta $ (which
  may tend to zero when  $\delta $ tends to zero).
Let also Assumption \ref{Ass-stability-1} be valid and the solution $(y_{\epsilon}^{N,M}(\cdot),z_{\epsilon}^{N,M}(\cdot)) $ of the system (\ref{e:intro-0-1})-(\ref{e:intro-0-2}) obtained with the control  $u_{\epsilon}^{N,M}(\cdot) $  and with the initial conditions
$\ (y_{\epsilon}^{N,M}(0),z_{\epsilon}^{N,M}(0))= (y^{N,M}(0),z^{N,M}(0))  \ $ satisfies the inclusion (\ref{equ-Y}). Then
the control $u_{\epsilon}^{N,M}(\cdot) $ is $\beta(N,M) $-asymptotically near optimal in the problem (\ref{Vy-perturbed-per-new-1-long-run}), where $\beta(N,M)$ is defined in (\ref{e:HJB-16-NM-beta}). Also,
\begin{equation}\label{e:contr-rev-100-3-1002-101}
\lim_{\epsilon\rightarrow 0}\sup_{t\in [0,\infty)}||z_{\epsilon}^{N,M}(t)-z^{N,M}(t)|| \ = \ 0
\end{equation}
and, if the triplet $\ (u_{\epsilon}^{N,M}(\cdot), y_{\epsilon}^{N,M}(\cdot),z_{\epsilon}^{N,M}(\cdot)) \  $ generates the occupational measure $\gamma_{\epsilon}^{N,M } $ on the interval $[0,\infty)$ (see (\ref{e:occup-S-infty})), then
\begin{equation}\label{e:contr-rev-100-3-1002-101-101}
\rho(\gamma_{\epsilon}^{N,M }, \Phi(\lambda^{N,M}))\ \leq \ \kappa(\epsilon)  \ \ \ \ {\rm where} \ \ \ \
\lim_{\epsilon\rightarrow 0}\kappa(\epsilon) = 0,
\end{equation}
 with $\lambda^{N,M}$ being the occupational measure generated
 by $\ (\mu^{N,M}(\cdot), z^{N,M}(\cdot))\ $  (see (\ref{Vy-ave-h-def-averaged-measure-p-NM})) and the map $\Phi(\cdot) $ being defined by (\ref{e:h&th-1}).
\end{Theorem}
\begin{proof}
The proof of the theorem is given in Section \ref{Sec-Selected-proofs-new}.
\end{proof}

Note that from (\ref{e:HJB-1-17-1-N-z-const-def-near-opt-new-1}) and (\ref{e:contr-rev-100-3-1002-101-101}) it follows (due to continuity of $\Phi(\cdot) $) that
\begin{equation}\label{e:contr-rev-100-3-1002-101-101-1}
\rho(\gamma_{\epsilon}^{N,M }, \Phi(p^{N,M}))\ \leq \ \kappa(\epsilon)  \ + \ \theta(N,M),\ \ \ \ {\rm where} \ \ \ \ \lim_{N\rightarrow\infty}\limsup_{M\rightarrow\infty}\theta(N,M) = 0
\end{equation}
and where $p^{N,M}$ is an arbitrary optimal solution of the $(N,M)$-approximating averaged problem (\ref{e-ave-LP-opt-di-MN}).  This problem
always has an optimal solution that can be presented in the form (see Section \ref{Sec-Selected-proofs} below)

\begin{equation}\label{e:ex-4-1-rep-101-1-2-3}
p^{N,M}\BYDEF\sum^{K}_{k=1}p_k \delta_{(\mu_k,z_k)}, \ \ \ \ \ \sum^{K}_{k=1}p_k =1, \ \ \ \ \ p_k> 0, \ k=1,...,K,
\end{equation}
where
\begin{equation}\label{e:ex-4-1-rep-101-1-2-4}
\mu_k = \sum_{j=1}^{J_k} q_j^{k}\delta_{(u_j^{k},y_j^{k})}, \ \ \ \ \ \sum_{j=1}^{J_k} q_j^{k} =1, \ \ \ \ \ q_j^{k}> 0, \ j=1,..., J_k,
\end{equation}
$\delta_{(u_j^{k},y_j^{k})} $ being the Dirac measure concentrated at $(u_j^{k},y_j^{k})\in U\times Y \ (j=1,...,J_k) $ and $\delta_{(\mu_k,z_k)} $ being the  Dirac measure concentrated at $(\mu_k,z_k)\in \mathcal{P}(U\times Y)\times Z\ $ ($k=1,...K)$. As can be readily verified,
\begin{equation}\label{e:ex-4-1-rep-101-1-2-4-11}
\Phi(\sum^{K}_{k=1}p_k \delta_{(\mu_k,z_k)})= \sum^{K}_{k=1}\sum_{j=1}^{J_k}p_kq_j^{k}\delta_{(u_j^{k},y_j^{k},z_k)},
\end{equation}
where $\ \delta_{(u_j^{k},y_j^{k},z_k)} \ $ is the Dirac measure concentrated at $(u_j^{k},y_j^{k}, z_k)$. Thus, by (\ref{e:contr-rev-100-3-1002-101-101-1}),
\begin{equation}\label{e:contr-rev-100-3-1002-101-101-11}
\rho(\gamma_{\epsilon}^{N,M },\ \sum^{K}_{k=1}\sum_{j=1}^{J_k}p_kq_j^{k}\delta_{(u_j^{k},y_j^{k},z_k)} )\ \leq \ \kappa(\epsilon)  \ + \ \theta(N,M).
\end{equation}
That is, for $N,M $ large enough and $\epsilon$ small enough, the occupational measure $\gamma_{\epsilon}^{N,M } $ is approximated by a convex combination of
the Dirac measures, which implies, in particular, that the state trajectory $\ (y_{\epsilon}^{N,M}(\cdot),z_{\epsilon}^{N,M}(\cdot)) \  $ spends
a non-zero proportion of time in a vicinity of each of the points $(u_j^{k},y_j^{k},z_k) $.

To illustrate the construction of asymptotically near optimal controls,  let us conclude this section  with   an example (which was briefly discussed in \cite{GR-long}).
 Consider the  optimal control problem
\begin{equation}\label{e:ex-4-1-rep-101-2-0}
\inf _{(u(\cdot),y_{\epsilon}(\cdot),z_{\epsilon}(\cdot)) }\liminf_{\mathcal{T}\rightarrow\infty}\frac{1}{\mathcal{T}}\int_0 ^ {\mathcal{T}}
(0.1u_1^2(t)+ 0.1u_2^2(t) -z_1^2(t))dt = V^*(\epsilon),
\end{equation}
where minimization is  over the controls $u(\cdot) =(u_1(\cdot),u_2(\cdot)) $,
\begin{equation}\label{e:ex-4-1-rep-1-0}
(u_1(t),u_2(t))\in U\BYDEF \{(u_1,u_2)\ : \ |u_i|\leq 1, \ i=1,2\},
\end{equation}
   and the corresponding solutions $ y_{\epsilon}(\cdot)=(y_{1,\epsilon}(\cdot),y_{2,\epsilon}(\cdot))$ and $z_{\epsilon}(\cdot)=(z_{1,\epsilon}(\cdot),z_{2,\epsilon}(\cdot))$  )
   of the SP  system
\begin{equation}\label{e:ex-4-2-repeat}
\epsilon y'_i(t)= -y_i(t) + u_i(t), \ \ \ i=1,2,
\end{equation}
\vspace{-.2in}
\begin{equation}\label{e:ex-4-1-rep-101-1}
z_1'(t)=z_2(t), \ \ \ \ \ \ \ z_2'(t)= - 4z_1(t) - 0.3z_2(t)-y_1(t)u_2(t)+y_2(t)u_1(t),
\end{equation}
with
$$
(y_1(t),y_2(t))\in Y\BYDEF\{(y_1,y_2)\ : \ |y_i|\leq 1, \ i=1,2\}\
$$
and with
$$
(z_1(t), z_2(t))\in Z \BYDEF \{(z_1,z_2)\ : \ |z_1|\leq 2.5, \ \ |z_2|\leq 4.5\}.
$$
The averaged system (\ref{e:intro-0-4}) takes in this case the form
\begin{equation}\label{e:ex-4-1-rep-101-1-1}
z_1'(t)=z_2(t), \ \ \ \ \ \ \ z_2'(t)= - 4z_1(t) + \int_{U\times Y}(-y_1u_2+y_2u_1)\mu(t)(du,dy)
\end{equation}
where
\begin{equation}\label{e:ex-4-1-rep-101-1-2}
\mu(t)\in W\BYDEF \{\mu \in \mathcal{P}(U \times Y) \ : \ \int_{U\times Y} [\frac{\partial \phi(y)}{\partial y_1}(-y_1 + u_1)\ + \
\frac{\partial \phi(y)}{\partial y_2}(-y_2 + u_2)]\mu(du,dy) =  0 \ \ \forall \phi(\cdot) \in C^1(\R^2)\},
\end{equation}
Note that, as can be readily verified, the function $\ f(u,y) = (-y_1 + u_1, \  -y_2 + u_2)$ satisfies Assumption \ref{Ass-stability-1}(i)
and the function $\ \tilde g(\mu,z) = (z_2, \ -4z_1 + \int_{U\times Y}(-y_1u_2+y_2u_1)\mu(du,dy)\ ) $ satisfies Assumption \ref{Ass-stability-1}(ii).

The $(N,M)$-approximating averaged problem (\ref{e-ave-LP-opt-di-MN})
was constructed in this example  with the use of the monomials $\ z^{j_1}_1z^{j_2}_2\ (1\leq j_1 + j_2\leq 5)$
as the test functions in (\ref{D-SP-new-MN}) and the monomials $\ y_1^{i_1}y_2^{i_2}\ (1\leq i_1 + i_2\leq 5)$  as the test functions
in (\ref{e:2.4-M}). Note that $N,M=35$ in this case (recall that $N$ stands for the number of constraints in (\ref{D-SP-new-MN})
and $ M$ stands for the number of constraints in (\ref{e:2.4-M})).
This  problem was solved numerically  with the help of the linear programming based algorithm described in Section 4.3 of \cite{GR-long},
its output including the optimal value of the problem, an optimal solution of the problem and solutions of the corresponding averaged and associated dual problems.

The optimal value of the problem was obtained to be approximately equal to $-1.186 $:
\begin{equation}\label{e:near-opt-SP-10-6-per}
\tilde{G}^{35,35}\approx -1.186.
\end{equation}
Along with the optimal value,
 the
 points
\begin{equation}\label{e:ex-4-1-rep-101-1-2-1}
z_k = (z_{1,k}, z_{2,k})\in Z, \ \ \ \ k=1,..., K,
\end{equation}
 and weights $\{p_k\} $ that enter the expansion (\ref{e:ex-4-1-rep-101-1-2-3}) as well as the points
 \begin{equation}\label{e:ex-4-1-rep-101-1-2-2}
u_j^{k}= (u_{1,j}^{k}, u_{2,j}^{k} )\in U ,\  \ \ \ y_j^{k}= (y_{1,j}^{k}, y_{2,j}^{k} )\in Y, \ \ \ \ \    \ j=1,..., J_k, \ \ \ \ k=1,..., K,
\end{equation}
and the corresponding weights $\{q_j^{k} \} $ that enter the expansions (\ref{e:ex-4-1-rep-101-1-2-4}) were numerically found.
 In Figure 1, the points $\{z_k\} $ that enter the  expansion (\ref{e:ex-4-1-rep-101-1-2-3}) are marked with dotes on the \lq\lq $z$-plane". Corresponding
 to each such a point $z_k $, there are points $\{y_j^{k} \}$ that enter the expansion (\ref{e:ex-4-1-rep-101-1-2-4}). These points are  marked with dots on the \lq\lq $y$-plane" in Figure 2 for  $\ z_k\approx (1.07, -0.87)$ (which is one of the points  marked in Fig.1; for other points marked in  Fig. 1, the configurations of the corresponding $\{y_j^{k} \}$  points look similar).

The expansions (\ref{e:DUAL-AVE-0-approx-2}) and (\ref{e:DUAL-AVE-0-approx-1-Associate-1})
that define
solutions of the  $(N,M)$-approximating averaged and $(N,M)$-approximating associated dual problems take the form
\begin{equation}\label{e:near-opt-SP-10-7-per}
\zeta^{35,35}(z)=\sum_{1\leq j_1+j_2\leq 5} \lambda_{j_1,j_2}^{35,35}z^{j_1}_1z^{j_2}_2 , \ \ \ \ \ \ \ \ \ \ \eta^{35,35}_z(y)=\sum_{1\leq i_1+i_2\leq 5} \omega_{z,i_1,i_2}^{35,35} \ y_1^{i_1}y_2^{i_2},
\end{equation}
where the coefficients  $\{ \lambda_{j_1,j_2}^{35,35} \}  $ and $\{\omega_{z,i_1,i_2}^{35,35}  \} $ are obtained as a part of the solution with the above mentioned algorithm.
Using $\zeta^{35,35}(z)$ and $\eta^{35,35}_z(y) $, one can compose the problem (\ref{e-NM-minimizer-0}):
$$
\min_{u_i\in [-1,1]}\{0.1u_1^2+ 0.1u_2^2- z_1^2+ \frac{\partial \zeta^{35,35}(z)}{\partial z_1}z_2 + \frac{\partial \zeta^{35,35}(z)}{\partial z_2}(-4z_1-0.3 z_2 -y_1u_2 + y_2u_1) + \frac{\partial\eta^{35,35}_z(y)}{\partial y_1}(-y_1 +u_1)
$$
\vspace{-.2in}
\begin{equation}\label{e-NM-minimizer-0-example-per}
\ \ \ \ \ \ \ \ \ \ \ \ \ \ \ \ \ \ \ \ \ \ \ +\ \frac{\partial\eta^{35,35}_z(y)}{\partial y_2}(-y_2 +u_2)\},
\end{equation}
the solution of which is written in the form
\begin{equation}\label{feedbackFinalVel-SP-per}
u^{35,35}_i(y,z)=\left\{
\begin{array}{rrrl}
- 5 b_i^{35,35}(y,z) \ & \ \ \ \ \ if&  \ \ \ \ \ \ \ \ \ \ \ \  |5 b_i^{35,35}(y,z)|\leq 1,  
\\
-1  \ & \ \ \ \  \  if&  \ \ \ \ \ \ \ \ \ \ \ \  \ -5 b_i^{35,35}(y,z) < -1,
\\
1  \ & \ \ \ \ \  if&  \ \ \ \ \ \ \ \ \ \ \ \  \
-5 b_i^{35,35}(y,z) > 1,
\\
\end{array}
\right\},
\ \ \ i=1,2,
\end{equation}
where
$\ b_1^{35,35}(y,z)\BYDEF \frac{\partial \zeta^{35,35}(z)}{\partial z_2} y_2 + \frac{\partial\eta^{35,35}_z(y)}{\partial y_1} \ $
and $\ b_2^{35,35}(y,z)\BYDEF -\frac{\partial \zeta^{35,35}(z)}{\partial z_2} y_1 + \frac{\partial\eta^{35,35}_z(y)}{\partial y_2}\ $.

 Using the feedback controls $u^{35,35}_i(y,z), \ i=1,2,$  with fixed $\ z = z_k\approx (1.07, -0.87)$  and integrating the associated system with MATLAB from the initial conditions defined by one of the points marked in Figure 2, one obtains a periodic solution $\ y_z^{35,35}(\tau)=(y_{1,z}^{35,35}(\tau), y_{2,z}^{35,35}(\tau))  $. The corresponding square like state trajectory of the associated system  is also depicted in Figure 2. Note that this trajectory is located in a close vicinity of the marked points, this being consistent with the comments made after the statement of Theorem \ref{Main-SP-SP-Nemeric}.

 Using the same controls $u^{35,35}_i(y,z), \ i=1,2,$  in the SP
  system (\ref{e:ex-4-2-repeat})-(\ref{e:ex-4-1-rep-101-1}) and integrating the latter (taken with $\epsilon =0.01 $
and $\epsilon = 0.001$) with  MATLAB from the initial conditions defined by one of the points marked in Figure 1 and one of the points marked in Figure 2, one obtains  visibly periodic solutions, the images of which are depicted in Figures 3 and 4, with the state trajectory of the slow dynamics $\ z^{35,35}_{\epsilon} (t)= (z_{1\epsilon}^{35,35}(t), z_{2, \epsilon}^{35,35}(t)) $ being also depicted in  Figure 1.
The slow $z $-components appear to be  moving periodically along an  ellipse like figure on the plane $(z_1,z_2)$, with the period being approximately equal
 to $3.16 $. Note that this figure  and the period  appear to be the same for $\epsilon =0.01 $
and for $\epsilon = 0.001$, with the marked points  being located on or very close to the ellipse like figure in Fig. 1 (which again is consistent with the comments made after Theorem \ref{Main-SP-SP-Nemeric}). In Figures 3 and 4, the fast $y $-components  are moving along  square like figures (similar to that in Fig. 1) centered around the points on  the \lq\lq ellipse", with about $50$  rounds for the case  $\epsilon = 0.01$ (Fig. 3) and about $500$  rounds for the case  $\epsilon = 0.001$ (Fig. 4). The  values of the objective functions obtained for these two cases
are approximately the same and $\approx -1.177$, the latter being close to the value of $\tilde{G}^{35,35}$ (see (\ref{e:near-opt-SP-10-6-per})).
Due to (\ref{lim-inequalities}) and due to (\ref{e:graph-w-M-103-101}), this indicates that the found solution is close to the optimal one.

\begin{center}
{\it Graphs of $\ z^{35,35}_{\epsilon} (t) $ and $\ y_z^{35,35}(\tau)$  }
\end{center}
\bigskip
\begin{center}
\vspace{-.1in}
\includegraphics[scale=0.36]{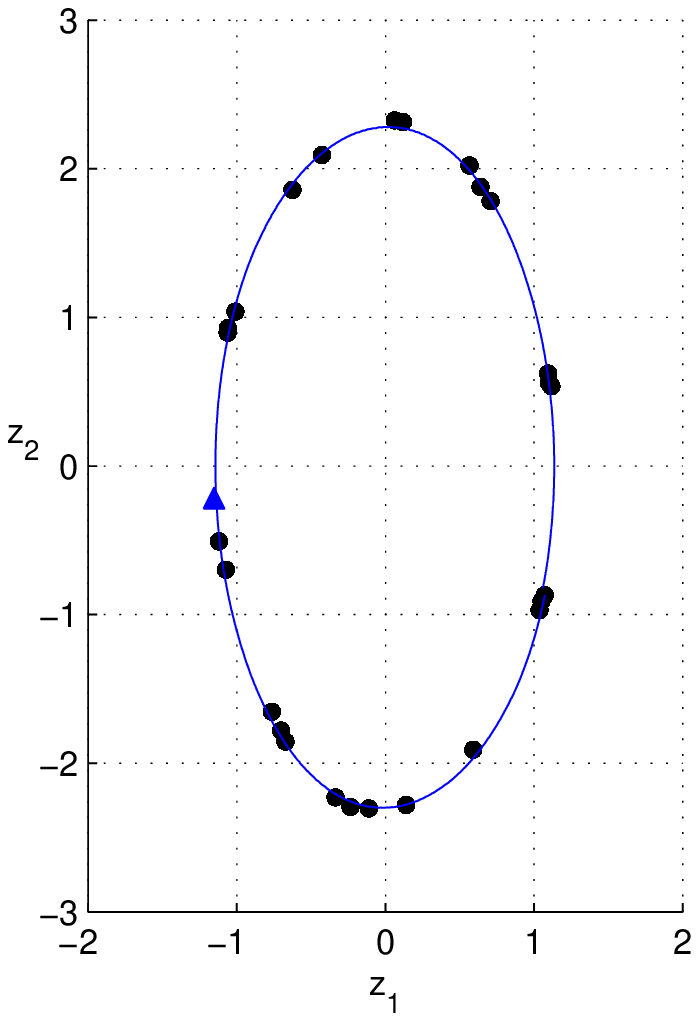}
\hspace{+.18in}
\includegraphics[scale=0.36]{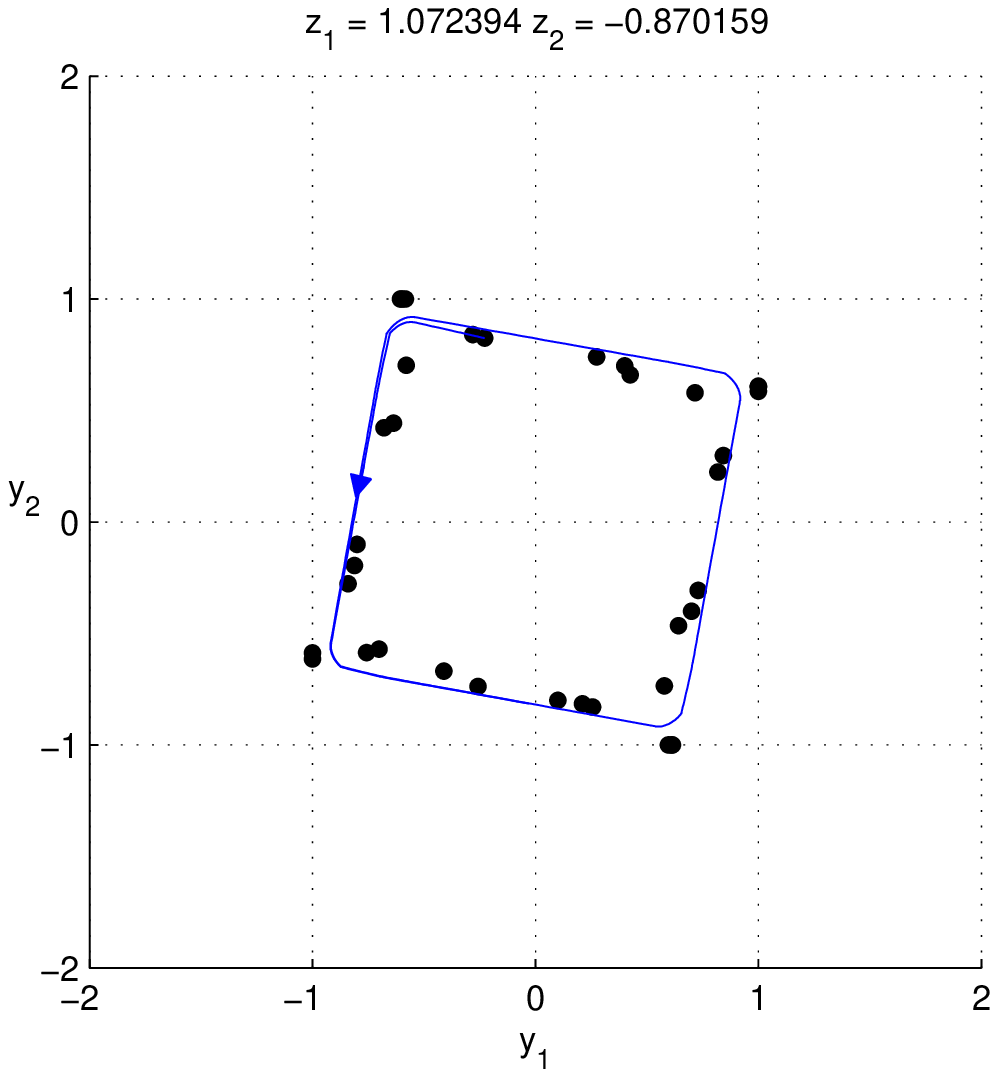}
\end{center}
 \vspace{-.1in}
 \ \ \ \ \ \ \ \ \ \ \ \   \ \ \ \ \ \ \   Fig. 1:  $\ z^{35,35}_{\epsilon} (t)= (z_{1\epsilon}^{35,35}(t), z_{2, \epsilon}^{35,35}(t)) $   \ \ \ \ \ \ \ \ \
 Fig.2: \ $\ y_z^{35,35}(\tau)=(y_{1,z}^{35,35}(\tau), y_{2,z}^{35,35}(\tau)) \  $
\bigskip
 \medskip

\begin{center}
{\it Images of the state trajectories of the SP system for $\ \epsilon = 0.01$  and $\ \epsilon = 0.001$  }
\end{center}
\begin{center}
\vspace{-.1in}
\includegraphics[scale=0.36]{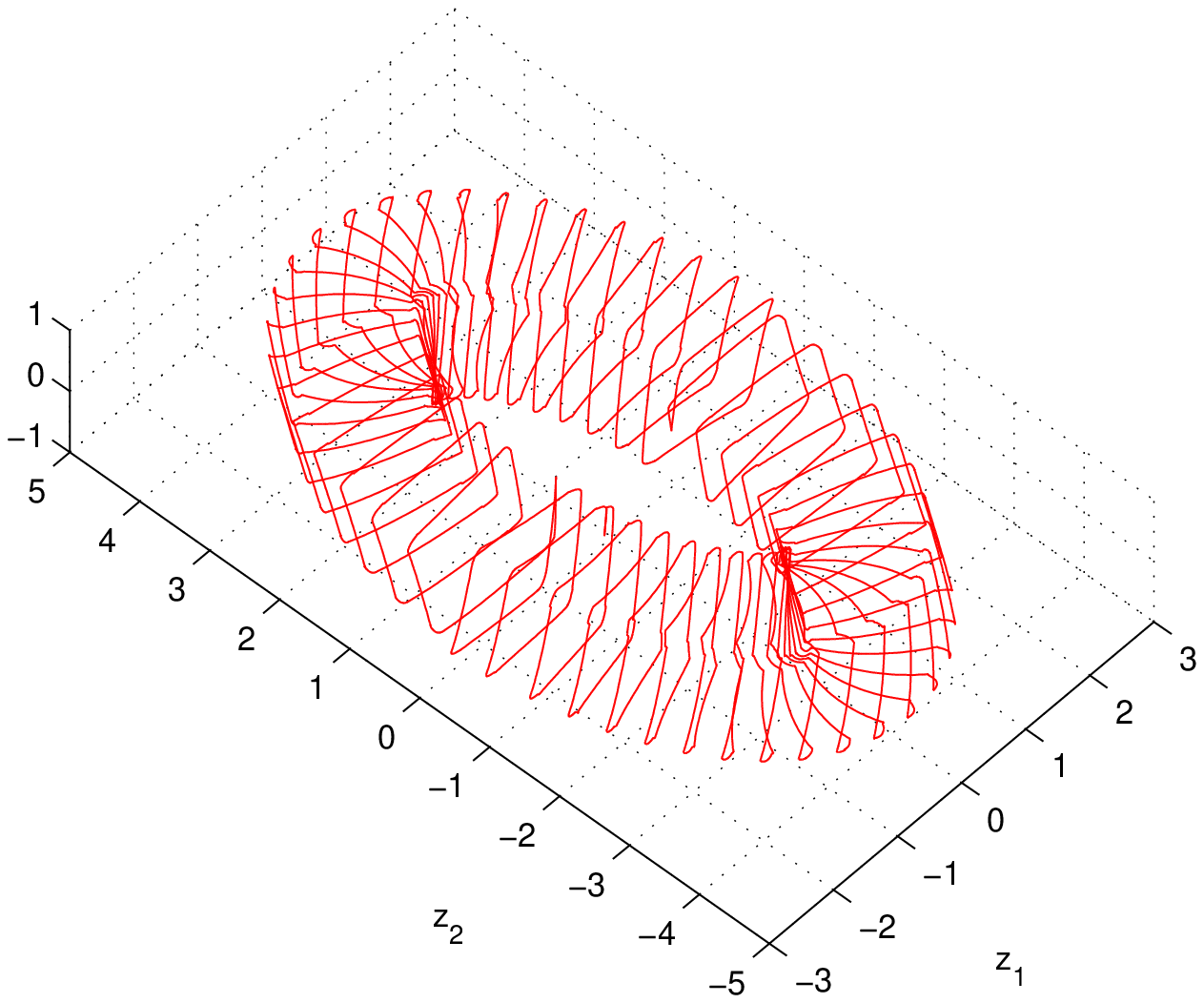}
\hspace{+.18in}
\includegraphics[scale=0.36]{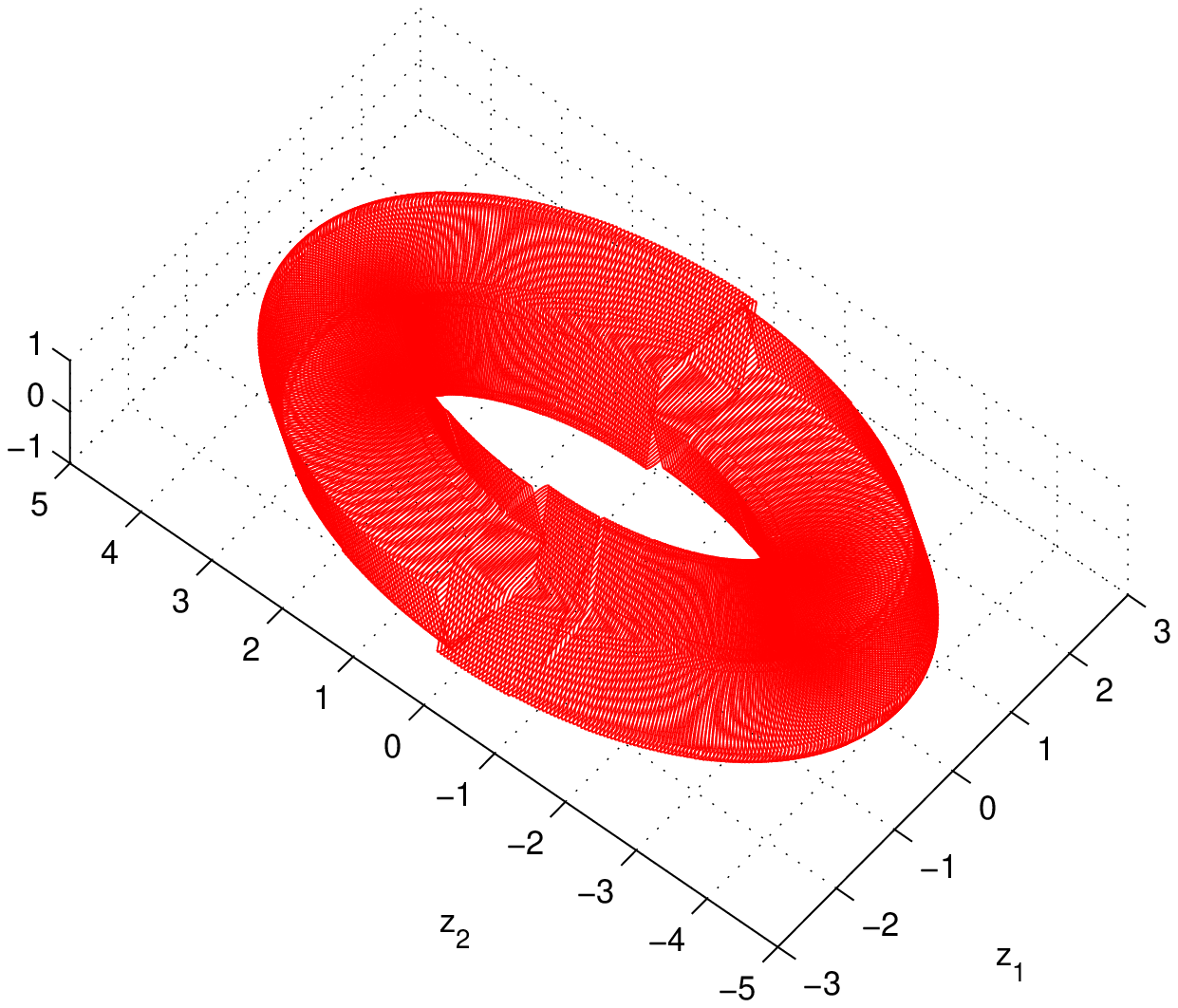}
\end{center}
 \vspace{-.1in}
\ \ \ \ \ \ \ \ \ \ \ \ \ \ \ \ \ \   \ Fig. 3:  $\ (y_{ \epsilon}^{35,35}(t), z_{\epsilon}^{35,35}(t))\  $  for $\epsilon = 0.01$ \ \ \ \ \ \ \  \  \ \
 Fig.4: \ $\ (y_{ \epsilon}^{35,35}(t), z_{\epsilon}^{35,35}(t))$ for $\epsilon = 0.001$

Note, in conclusion, that by taking $\epsilon = 0$ in (\ref{e:ex-4-2-repeat}),
one obtains  $y_i(t) = u_i(t), \ i=1,2, $ and, thus, arrives at the equality
$$
  -y_1(t)u_2(t)+y_2(t)u_1(t)= 0 \ \ \ \forall t ,
$$
which  makes the slow dynamics  uncontrolled and leads to the optimality of
the \lq\lq trivial" steady state  regime:
$ \ u_1(t)=u_2(t)=y_1(t)=y_2(t)=z_1(t)=z_2(t)= 0 \ \  \forall t \ $ implying that
 $V^*(0)=0 $. Thus, in the present example
 $$
 \lim_{\epsilon\rightarrow 0}V^*(\epsilon)\approx -1.177 \ < \ 0 \ = \ V^*(0).
 $$

\section{Proof of Theorem \ref{Main-SP-Nemeric}}\label{Sec-Selected-proofs}
Note, first of all, that there exists an optimal solution $p^{N,M}$ of the problem (\ref{e-ave-LP-opt-di-MN}) which is
presented as a convex combination of (no more than $N+1$)
 Dirac measures (see, e.g., Theorems A.4 and A.5 in \cite{Rubio}).
  That is,
  \begin{equation}\label{e-NM-minimizer-proof-2}
p^{N,M}= \sum^{K^{N,M}}_{k=1}p_k^{N,M} \delta_{(\mu_k^{N,M},z_k^{N,M})},
\end{equation}
  where $\delta_{(\mu_k^{N,M},z_k^{N,M})}$
is the Dirac measure concentrated at $(\mu_k^{N,M},z_k^{N,M})$ and
 \begin{equation}\label{e-NM-minimizer-proof-1}
 (\mu_k^{N,M},z_k^{N,M})\in F_M, \ \ \  \  p_k^{N,M}>0, \ \ k=1,...,K^{N,M}\leq N+1;\ \ \ \ \   \   \sum^{K^{N,M}}_{k=1}p_k^{N,M}=1 .
\end{equation}

\begin{Lemma}\label{opt-1}
For any $ k=1,...,K^{N,M} $,
\begin{equation}\label{e-NM-minimizer-proof-2-2}
\mu_k^{N,M} = {\rm argmin}_{\mu\in W_M(z_k^{N,M})}\{\tilde G(\mu, z_k^{N,M})  + \nabla \zeta^{N,M} (z_k^{N,M})^T  \tilde g (\mu , z_k^{N,M}) \}.
\end{equation}
That is, $\mu_k^{N,M}$ is a minimizer of the problem
\begin{equation}\label{e-NM-minimizer-proof-2-3}
\min_{\mu\in W_M(z_k^{N,M})}\{ \tilde G(\mu, z_k^{N,M}) + \nabla \zeta^{N,M} (z_k^{N,M})^T  \tilde g (\mu , z_k^{N,M})\ \}.
\end{equation}
\end{Lemma}
\begin{proof}
The proof is similar to that of Lemma 5.1 in \cite{GR-long}
\end{proof}
\begin{Lemma}\label{opt-2}
In the presentation (\ref{e-NM-minimizer-proof-2}) of an optimal solution $p^{N,M} $ of the problem (\ref{e-ave-LP-opt-di-MN}), $\mu_k^{N,M} $
can be chosen as follows:
\begin{equation}\label{e-NM-minimizer-proof-3}
\mu_k^{N,M}= \sum_{j=1}^{J^{N,M,k}} q_j^{N,M,k}\delta_{(u_j^{N,M,k},y_j^{N,M,k})}, \ \ \ \ k= 1,...,K^{N,M},
\end{equation}
where
\begin{equation}\label{e-NM-minimizer-proof-4}
 \ q_j^{N,M,k}> 0, \ \ j=1,..., J^{N,M,k}, \ \ \ \  \sum_{j=1}^{J^{N,M,k}} q_j^{N,M,k}=1,
\end{equation}
and
\begin{equation}\label{e-NM-minimizer-proof-3-1}
J^{N,M,k}\leq N+M+2 .
\end{equation}
In (\ref{e-NM-minimizer-proof-3}), $\delta_{(u_j^{N,M,k},y_j^{N,M,k})}\in \mathcal{P}(U\times Y)$  are the Dirac measures  concentrated at $\ (u_j^{N,M,k},y_j^{N,M,k})\in U\times Y, \ \ j=1,..., J^{N,M,k}$, with
$$
u_j^{N,M,k}= argmin_{u\in U}\{G(u,y_j^{N,M,k},z_k^{N,M})+\nabla \zeta^{N,M} (z_k^{N,M})^T  g (u,y_j^{N,M,k} , z_k^{N,M})
$$
\vspace{-.2in}
\begin{equation}\label{e-NM-minimizer-proof-5}
 + \nabla \eta^{N,M} (y_j^{N,M,k})^T  f(u,y_j^{N,M,k},z_k^{N,M})\}.
\end{equation}
\end{Lemma}
\begin{proof}
The proof is similar to that of Lemma 5.2 in \cite{GR-long}
\end{proof}
\begin{Lemma}\label{convergence-important}
For any $t$ such that (\ref{e:convergence-important-3}) is satisfied, there exists a sequence
\begin{equation}\label{e:convergence-important-0}
(\mu_{k^{N,M}}^{N,M},  z_{k^{N,M}}^{N,M})\in  \{(\mu_k^{N,M}, z_k^{N,M}),\ k=1,...,K^{N,M} \}, \ \ N=1,2,..., \ \ M=1,2,...,
\end{equation}
(with $\{(\mu_k^{N,M}, z_k^{N,M}),\ k=1,...,K^{N,M} \} $ being the set of  concentration points of the
Dirac measures  in  (\ref{e-NM-minimizer-proof-2}))
such that
\begin{equation}\label{e:convergence-important-1}
 \lim_{N\rightarrow \infty}\limsup_{M\rightarrow \infty} [\rho(\mu^*(t),\mu_{k^{N,M}}^{N,M}) + ||z^*(t)-z_{k^{N,M}}^{N,M}||]=0.
\end{equation}

Let $\ t\ $ be such that (\ref{e:convergence-important-1}) is valid and let $(\mu_{k^{N,M}}^{N,M},  z_{k^{N,M}}^{N,M}) $ be as in (\ref{e:convergence-important-1}),
  then for  any $\tau $ such that (\ref{e:convergence-important-7}) is satisfied, there exists a sequence
    \begin{equation}\label{e:convergence-important-2-2}
\ (u^{N,M,k^{N,M}}_{j^{N,M}}, y^{N,M,k^{N,M}}_{j^{N,M}})\in \{(u_j^{N,M,k^{N,M}},y_j^{N,M,k^{N,M}}),\ j=1,...,J^{N,M,k^{N,M}} \}, \ \ N=1,2,..., \ \ M=1,2,...,
\end{equation}
 ($\{(u_j^{N,M,k^{N,M}},y_j^{N,M,k^{N,M}}),\ j=1,...,J^{M,N,k^{N,M}} \}$ being the set of
  concentration points of the Dirac measures in (\ref{e-NM-minimizer-proof-3}) taken with $k=k_{N,M} $)
  such that
  \begin{equation}\label{e:convergence-important-2}
\lim_{N\rightarrow \infty}\limsup_{M\rightarrow \infty}[ ||u^*_t(\tau)-u^{N,M,k^{N,M}}_{j^{N,M}}|| +
||y^*_t(\tau))- y^{N,M,k^{N,M}}_{j^{N,M}}||] = 0. \end{equation}

\end{Lemma}
\begin{proof}
Assume that (\ref{e:convergence-important-1}) is not true. Then there exists  $r>0$ and sequences $N_i $ , $M_{i,j} $ with
$i=1,..., \ j=1,...,\ $ and with
$\ \lim_{i\rightarrow\infty}N_i = \infty,
\ \ \ \lim_{j\rightarrow\infty}M_{i,j}= \infty  $ such that
\begin{equation}\label{e:convergence-important-4}
dist ((\mu^*(t) ,z^*(t)), \Theta^{N_i, M_{i,j}})\geq r ,
\end{equation}
where $\Theta^{N, M}$ is the set of the concentration points of the
Dirac measures  in  (\ref{e-NM-minimizer-proof-2}), that is,
$$
\Theta^{N, M}\BYDEF \{(\mu_k^{N,M}, z_k^{N,M}), k=1,...,K^{N,M} \},
$$
taken with $N=N_i $ and $M=M_{i,j} $, and where
$$
dist((\mu,z), \Theta^{N, M})\BYDEF \min_{(\mu',z')\in \Theta^{N, M}}\{\rho(\mu,\mu') + ||z-z'|| \}.
$$
 Hence,
$$
(\mu_k^{N_i,M_{i,j}}, z_k^{N_i,M_{i,j}}) \notin \mathcal{B}_r( \mu^*(t),  z^*(t)), \ \ k=1,...,K^{N_i,M_{i,j}}, \ \ i, j= 1,2,... \ .
$$
The latter implies that
\begin{equation}\label{e:convergence-important-5}
 p^{N_i, M_{i,j}} (\mathcal{B}_r(\mu^*(t),  z^*(t)) = 0, \ \ i, j= 1,2,... \ ,
\end{equation}
where $p^{N,N}$ is defined by (\ref{e-NM-minimizer-proof-2}). Due to the fact that the optimal solution $p^*$ of the  IDLP problem (\ref{e-ave-LP-opt-ave})
 is unique (Assumption \ref{SET-1}(i)), the relationship  (\ref{e:graph-w-M-102-101}) is valid. Consequently,
\begin{equation}\label{e:convergence-important-6}
\lim_{i\rightarrow\infty }\limsup_{j\rightarrow\infty}\rho( p^{N_i, M_{i,j}},p^*) = 0.
\end{equation}
From (\ref{e:convergence-important-5}) and (\ref{e:convergence-important-6}) it follows that
$$
p^* (\mathcal{B}_r(\mu^*(t),  z^*(t)))\leq \lim_{i\rightarrow\infty }\limsup_{j\rightarrow\infty}p^{N_i, M_{i,j}}(\mathcal{B}_r(\mu^*(t),  z^*(t)))=0.
$$
The latter contradicts to (\ref{e:convergence-important-3}). Thus, (\ref{e:convergence-important-1}) is proved.

Assume now that (\ref{e:convergence-important-2}) is not valid.
Then there exists  $r>0$ and sequences $N_i $ , $M_{i,j} $ with
$i=1,..., \ j=1,...,\ $ and with
$\ \lim_{i\rightarrow\infty}N_i = \infty,
\ \ \ \lim_{j\rightarrow\infty}M_{i,j}= \infty  $ such that
\begin{equation}\label{e:convergence-important-4-(u,y)}
dist ((u^*_t(\tau) , y^*_t(\tau)), \theta^{N_i, M_{i,j}})\geq r , \ \ \ \ i, j= 1,2,... \ ,
\end{equation}
where $\theta^{N, M} $ is  the set of the concentration points of the
Dirac measures  in  (\ref{e-NM-minimizer-proof-3}),
$$
\theta^{N, M}\BYDEF \{(u_j^{N,M,k^{N,M}},y_j^{N,M,k^{N,M}}),\ \ j=1,...,J^{N,M,k^{N,M}} \},
$$
taken with $k=k^{N,M} $ and with $N=N_i $, $M=M_{i,j} $, and where
$$
dist((u,y), \theta^{N, M})\BYDEF \min_{(u',y')\in \theta^{N, M}}\{||u-u'|| + ||y-y'|| \}.
$$
From (\ref{e:convergence-important-4-(u,y)}) it follows that
$$
(u_j^{N_i,M_{i,j},k^{N_i,M_{i,j}}}, y_j^{N_i,M_{i,j},k^{N_i,M_{i,j}}}) \notin B_r(u^*_t(\tau),y^*_t(\tau)), \ \ j=1,...,J^{N_i,M_{i,j},k^{N_i,M_{i,j}}}, \ \ \ \ \ \ \ i, j= 1,2,... \ .
$$
The latter implies that
\begin{equation}\label{e:convergence-important-5-(u,y)}
\mu_{k^{N_i,M_{i,j}}}^{N_i, M_{i,j}} (B_r(u^*_t(\tau),y^*_t(\tau)) = 0, \ \ \ \ \ \ \ i, j= 1,2,... \ .
\end{equation}
where $\mu_k^{N,M}$ is defined by (\ref{e-NM-minimizer-proof-3}) (taken with $k=k^{N_i,M_{i,j}}, \ N=N_i $ and $M=M_{i,j} $).

From (\ref{e:convergence-important-1}) it follows, in particular, that
\begin{equation}\label{e:convergence-important-1-(u,y)}
 \lim_{N\rightarrow \infty}\limsup_{M\rightarrow \infty} \rho(\mu^*(t),\mu_{k^{N,M}}^{N,M})=0 \ \ \ \ \ \Rightarrow \ \ \ \ \
 \lim_{i\rightarrow \infty}\limsup_{j\rightarrow \infty}\rho(\mu^*(t),\mu_{k^{N_i,M_{i,j}}}^{N_i, M_{i,j}} )=0.
\end{equation}
The later and (\ref{e:convergence-important-5-(u,y)}) lead to
$$
\mu^*(t)(B_r(u^*_t(\tau),y^*_t(\tau))\leq \lim_{i\rightarrow \infty}\limsup_{j\rightarrow \infty}\mu_{k^{N_i,M_{i,j}}}^{N_i, M_{i,j}}(B_r(u^*_t(\tau),y^*_t(\tau))) = 0,
$$
which contradicts to (\ref{e:convergence-important-7}). Thus (\ref{e:convergence-important-2}) is proved.

\end{proof}

\begin{Lemma}\label{fast-convergence-extra}
For any $\ t\in [0,\infty)\ $ such that $\ Q_t \ $ is not empty and (\ref{e:convergence-important-3}) is valid for an arbitrary $\ r >0$, and
  for any $\ \tau \in [0,\infty)\ $ such that $\ B_{t,\tau}\ $  is not empty and
(\ref{e:convergence-important-7})  is valid for an arbitrary $\ r >0\ $,
\begin{equation}\label{e:convergence-important-12}
 \lim_{N\rightarrow \infty}\limsup_{M\rightarrow \infty}||u^*_t(\tau) - u^{N,M}(y^*_t(\tau), z^*(t))|| = 0.
 \end{equation}
\end{Lemma}

\begin{proof}
By Lemma \ref{convergence-important}, there exist $(\mu_{k^{N,M}}^{N,M},  z_{k^{N,M}}^{N,M}) $ such that (\ref{e:convergence-important-1}) is satisfied and there exist\\  $(u^{N,M,k^{N,M}}_{j^{N,M}}, y^{N,M,k^{N,M}}_{j^{N,M}}) $ such that (\ref{e:convergence-important-2}) is satisfied.

 Note that, due to (\ref{e-NM-minimizer-proof-5}),
\begin{equation}\label{e:convergence-important-11}
u^{N,M,k^{N,M}}_{j^{N,M}}= u(y^{N,M,k^{N,M}}_{j^{N,M}}, z_{k^{N,M}}^{N,M}),
\end{equation}
where $u(y,z) $ is as in (\ref{e-NM-minimizer-1}).  From (\ref{e:convergence-important-1}) and (\ref{e:convergence-important-2}) it follows that
$ \ z_{k^{N,M}}^{N,M}\in Q_t $ and $\ y^{N,M,k^{N,M}}_{j^{N,M}}\in B_{t,\tau} $ for $N $ and $M$ large enough. Hence, one can use (\ref{e:HJB-1-17-0-per-Lip-u-NM})
to obtain
$$
||u^*_t(\tau) - u^{N,M}(y^*_t(\tau), z^*(t))||\leq ||u^*_t(\tau)- u^{N,M,k^{N,M}}_{j^{N,M}}  || + ||u(y^{N,M,k^{N,M}}_{j^{N,M}}, z_{k^{N,M}}^{N,M}) -
u^{N,M}(y^*_t(\tau), z^*(t))||
$$
\vspace{-.2in}
$$
\leq ||u^*_t(\tau)- u^{N,M,k^{N,M}}_{j^{N,M}}  || + L(||y^*_t(\tau)-y^{N,M,k^{N,M}}_{j^{N,M}}|| + ||z^*(t)- z_{k^{N,M}}^{N,M}||).
$$
By (\ref{e:convergence-important-1}) and (\ref{e:convergence-important-2}), the latter implies (\ref{e:convergence-important-12}).
\end{proof}

{\it Proof of Lemma \ref{fast-convergence}.}
Let $\ t\in [0,\infty)\ $ be such that $\ Q_t \ $ is not empty and (\ref{e:convergence-important-3}) is satisfied for an arbitrary $\ r>0 $. Note that, from the assumptions made, it follows that (\ref{e:convergence-important-12})
is valid for almost all $\tau\in [0,\infty) $.

Take an arbitrary $\tau\in [0,\infty) $ and subtract the equation
 \begin{equation}\label{e:L-1-4}
y^*_t(\tau)=y^*_t(0) + \int_0^{\tau} f(u_t^*(\tau'), y^*_t(\tau'), z^*(t)) d\tau'
\end{equation}
from the equation
  \begin{equation}\label{e:L-1-5}
 y^{N,M}_{t} (\tau) =y^*_t(0) + \int_0^{\tau} f(u^{N,M}(y^{N,M}_{t} (\tau'),z^*(t)),  y^{N,M}_{t} (\tau'), z^*(t)) d\tau'.
\end{equation}
We will obtain
$$
|| y_{t}^{N,M}(\tau)-y^*_t(\tau)||\leq \int_0^{\tau} ||f(u^{N,M}( y^{N,M}_{t} (\tau'),z^*(t)),  y_{t}^{N,M}(\tau'), z^*(t))
$$
\vspace{-.2in}
$$
-  f(u_t^*(\tau'), y^*_t(\tau'), z^*(t))|| d\tau'\leq \int_0^{\tau} ||f(u^{N,M}( y^{N,M}_{t} (\tau'),z^*(t)),  y_{t}^{N,M}(\tau'), z^*(t))
$$
\vspace{-.2in}
$$
-f(u^{N,M}(y^*_t(\tau'),z^*(t)), y^*_t(\tau'), z^*(t))|| d\tau'
$$
\vspace{-.2in}
 \begin{equation}\label{e:L-1-6}
 +  \int_0^{\tau} || f(u^{N,M}(y^*_t(\tau'),z^*(t)), y^*_t(\tau'), z^*(t))- f(u_t^*(\tau'), y^*_t(\tau'), z^*(t))|| d\tau'.
\end{equation}
 Using Assumption \ref{SET-3} (ii),(iii), one can derive that
$$
\int_0^{\tau} ||f(u^{N,M}( y^{N,M}_{t} (\tau'),z^*(t)),  y_{t}^{N,M}(\tau'), z^*(t))
-  f(u^{N,M}(y^*_t(\tau'),z^*(t)), y^*_t(\tau'), z^*(t))|| d\tau'
$$
\vspace{-.2in}
$$
\leq \int_{\tau'\notin P_{t,\tau}(N,M)} ||f(u^{N,M}( y^{N,M}_{t} (\tau'),z^*(t)),  y_{t}^{N,M}(\tau'), z^*(t))
-  f(u^{N,M}(y^*_t(\tau'),z^*(t)), y^*_t(\tau'), z^*(t))|| d\tau'
$$
\vspace{-.2in}
$$
+ \int_{\tau'\in P_{t,\tau}(N,M)}[||f(u^{N,M}( y^{N,M}_{t} (\tau'),z^*(t)),  y_{z^*(t)}^{N,M}(\tau'), z^*(t))||
$$
\vspace{-.2in}
$$
+ ||f(u^{N,M}(y^*_t(\tau'),z^*(t)), y^*_t(\tau'), z^*(t))||] d\tau'
$$
\vspace{-.2in}
\begin{equation}\label{e:L-1-6-MN}
\leq L_1 \int_0^{\tau}|| y^{N,M}_{t} (\tau') - y^*_t(\tau')|| d\tau' + L_2 meas \{P_{t,\tau}(N,M)\},
\end{equation}
where $L_1$ is a constant defined (in an obvious way) by Lipschitz constants
of $f(\cdot)$ and $u^{N,M}(\cdot) $, and $\ L_2\BYDEF 2\ max_{(u,y,z)\in U\times Y\times Z}\{||f(u,y,z)||\} $.

Also, due to (\ref{e:convergence-important-12})  and the dominated convergence theorem (see, e.g., p. 49 in \cite{Ash})
 \begin{equation}\label{e:L-1-8}
 \lim_{N\rightarrow \infty}\limsup_{M\rightarrow \infty}\int_0^{\tau} || f(u^{N,M}(y^*_t(\tau'),z^*(t)), y^*_t(\tau'), z^*(t))- f(u_t^*(\tau'), y^*_t(\tau'), z^*(t))|| d\tau' = 0.
\end{equation}
Let us introduce the notation
$$
\kappa_{t,\tau}(N,M)\BYDEF L_2\ meas\{P_{t,\tau}(N,M)\}
$$
\vspace{-.2in}
$$
+\int_0^{\tau} || f(u^{N,M}(y^*_t(\tau'),z^*(t)), y^*_t(\tau'), z^*(t))- f(u_t^*(\tau'), y^*_t(\tau'), z^*(t))|| d\tau'
$$
and rewrite the inequality (\ref{e:L-1-6}) in the form
 \begin{equation}\label{e:L-1-9}
|| y_{t}^{N,M}(\tau)-y^*_t(\tau)||\leq L_1 \int_0^{\tau}||y^{N,M}_{t} (\tau') - y^*_t(\tau')|| d\tau' + \kappa_{t,\tau}(N,M).
\end{equation}
By Gronwall-Bellman lemma, it follows that
 \begin{equation}\label{e:L-1-10}
 \max_{\tau'\in [0,\tau]}|| y_{t}^{N,M}(\tau')-y^*_t(\tau')||  \leq \kappa_{t,\tau}(N,M)e^{L_1\tau}.
\end{equation}
Since, by (\ref{e:HJB-1-17-1-N-z-const}) and (\ref{e:L-1-8}),
\begin{equation}\label{e:L-1-11}
\lim_{N\rightarrow \infty}\limsup_{M\rightarrow \infty}\kappa_{t,\tau}(N,M) = 0,
\end{equation}
the inequality (\ref{e:L-1-10}) implies (\ref{e:HJB-1-17-1-N-z-const-fast-1}).

By (\ref{e:HJB-1-17-1-N-z-const-fast-1}), $\  y_{t}^{N,M}(\tau)\in B_{t,\tau} $ for $N $ and $M$ large enough (for $\tau\in [0,\infty) $  such that the ball
 $B_{t,\tau}$ is not empty). Hence,
$$
|| u^{N,M}(  y_{t}^{N,M}(\tau), z^*(t)) -u_t^*(\tau)||\leq || u^{N,M}(  y_{t}^{N,M}(\tau), z^*(t)) - u^{N,M}(y^*_t(\tau), z^*(t))||
$$
\vspace{-.2in}
$$
+||u^{N,M}(y^*_t(\tau), z^*(t)) - u_t^*(\tau)||\leq L ||y_{z^*(t)}^{N,M}(\tau) - y^*_t(\tau)|| +||u^{N,M}(y^*_t(\tau), z^*(t)) - u_t^*(\tau)||.
$$
The latter implies (\ref{e:HJB-1-17-1-N-z-const-fast-2}) (by (\ref{e:HJB-1-17-1-N-z-const-fast-1}) and (\ref{e:convergence-important-12})).
\endproof

{\it Proof of Theorem \ref{Main-SP-Nemeric}.} Let $\ t\in [0,\infty)\ $ be such that $\ Q_t \ $ is not empty and (\ref{e:convergence-important-3}) is satisfied for an arbitrary $\ r>0 $.  By (\ref{e-opt-OM-1-0-NM-star}) and (\ref{e-opt-OM-1-0-NM}), for any continuous $h(u,y) $ and
for an arbitrary small $\alpha > 0 $, there exists
 $S> 0$ such that
 \begin{equation}\label{e:LA-Airport-2}
 | S^{-1}\int_0^S h(u_t^*(\tau),y_{t}^{*}(\tau))d\tau - \int_{U\times Y}h(u,y)\mu^{*}(t)(du,dy)|\leq \frac{\alpha}{2}
\end{equation}
and
\begin{equation}\label{e:LA-Airport-1}
 | S^{-1}\int_0^S h(u^{N,M}(  y_{t}^{N,M}(\tau), z^*(t)),y_{t}^{N,M}(\tau))d\tau - \int_{U\times Y}h(u,y)\mu^{N,M}(z^*(t))(du,dy)|\leq \frac{\alpha}{2} .
\end{equation}
Using (\ref{e:LA-Airport-1}) and (\ref{e:LA-Airport-2}), one can obtain
$$
|\int_{U\times Y}h(u,y)\mu^{N,M}(z^*(t))(du,dy) - \int_{U\times Y}h(u,y)\mu^{*}(t)(du,dy)|
$$
\vspace{-.2in}
$$
\leq \ | S^{-1}\int_0^S h(u^{N,M}(  y_{t}^{N,M}(\tau), z^*(t)),y_{t}^{N,M}(\tau))d\tau - S^{-1}\int_0^S h(u_t^*(\tau),y_{t}^{*}(\tau))d\tau| + \alpha .
$$
Due to Lemma \ref{fast-convergence}, the latter implies the following inequality
$$
\lim_{N\rightarrow\infty}\limsup_{M\rightarrow\infty}|\int_{U\times Y}h(u,y)\mu^{N,M}(z^*(t))(du,dy) - \int_{U\times Y}h(u,y)\mu^{*}(t)(du,dy)|\leq \alpha ,
$$
which, in turn, implies
\begin{equation}\label{e:LA-Airport-3}
 \lim_{N\rightarrow\infty}\limsup_{M\rightarrow\infty}|\int_{U\times Y}h(u,y)\mu^{N,M}(z^*(t))(du,dy) - \int_{U\times Y}h(u,y)\mu^{*}(t)(du,dy)|=0
\end{equation}
(due to the fact that $\alpha $ can be arbitrary small). Since $h(u,y) $ is an arbitrary continuous function, from (\ref{e:LA-Airport-3}) it follows that
\begin{equation}\label{e:HJB-16-NM-proof-2}
\lim_{N\rightarrow\infty}\limsup_{M\rightarrow\infty }\rho(\mu^{N,M}(z^*(t)), \mu^*(t)) = 0 .
\end{equation}
Note that from the assumptions made it follows that $\ Q_t \ $ is not empty and (\ref{e:convergence-important-3}) is satisfied for an arbitrary $\ r>0 $ for almost all $t\in [0,\infty) $. Hence, (\ref{e:HJB-16-NM-proof-2}) is valid for almost all $t\in [0,\infty) $.

 Taking an arbitrary $t\in [0,\infty) $ and subtracting the equation
 \begin{equation}\label{e:HJB-16-NM-proof-3}
z^*(t)=z_0 + \int_0^t \tilde g(\mu^*(t'), z^*(t')) dt'
\end{equation}
from the equation
 \begin{equation}\label{e:HJB-16-NM-proof-4}
z^{N,M}(t)=z_0 + \int_0^t \tilde g(\mu^{N,M}(z^{N,M}(t')), z^{N,M}(t')) dt',
\end{equation}
one obtains
$$
||z^{N,M}(t)-z^*(t)||\leq \int_0^t ||\tilde g(\mu^{N,M}(z^{N,M}(t')), z^{N,M}(t'))- \tilde g(\mu^*(t'), z^*(t'))|| dt'
$$
\vspace{-.2in}
$$
\leq \int_0^t ||\tilde g(\mu^{N,M}(z^{N,M}(t')), z^{N,M}(t')) - \tilde g(\mu^{N,M}(z^{*}(t')), z^{*}(t'))||dt'
$$
\vspace{-.2in}
 \begin{equation}\label{e:HJB-16-NM-proof-5}
 +  \int_0^t || \tilde g(\mu^{N,M}(z^{*}(t')), z^{*}(t')) - \tilde g(\mu^*(t'), z^*(t'))|| dt'.
\end{equation}
From (\ref{e:HJB-1-17-1-per-Lip-g}) and from the definition of the set $A_t(N,M)$ (see (\ref{e:intro-0-4-N-A-t})), it follows that
$$
\int_0^t ||\tilde g(\mu^{N,M}(z^{N,M}(t')), z^{N,M}(t')) - \tilde g(\mu^{N,M}(z^{*}(t')), z^{*}(t'))||dt'
$$
\vspace{-.2in}
$$
\leq \int_{t'\notin A_t(N,M)}||\tilde g(\mu^{N,M}(z^{N,M}(t')), z^{N,M}(t')) - \tilde g(\mu^{N,M}(z^{*}(t')), z^{*}(t'))||dt'
$$
\vspace{-.2in}
$$
+ \int_{t'\in A_t(N,M)}[||\tilde g(\mu^{N,M}(z^{N,M}(t')), z^{N,M}(t'))|| + ||\tilde g(\mu^{N,M}(z^{*}(t')), z^{*}(t'))||\ ]dt'
$$
\vspace{-.2in}
 \begin{equation}\label{e:HJB-16-NM-proof-6}
\leq L\int_0^t ||z^{N,M}(t')- z^{*}(t') || + 2 L_g meas\{A_t(N,M)\},
\end{equation}
where $L_g\BYDEF  \max_{(u,y,z)\in U\times Y\times Z}||g(u,y,z||$.
This and (\ref{e:HJB-16-NM-proof-5}) allows one to obtain the inequality
\begin{equation}\label{e:HJB-16-NM-proof-6-11}
||z^{N,M}(t)-z^*(t)||\leq L\int_0^t ||z^{N,M}(t')- z^{*}(t') ||dt' +\kappa_t(N,M),
\end{equation}
 where
 $$
 \kappa_t(N,M)\BYDEF  2 L_g meas\{A_t(N,M)\} + \int_0^t || \tilde g(\mu^{N,M}(z^{*}(t')), z^{*}(t')) - \tilde g(\mu^*(t'), z^*(t'))|| dt'.
 $$
Note that, by (\ref{e:HJB-16-NM-proof-2}),
\begin{equation}\label{e:HJB-16-NM-proof-7}
 \lim_{N\rightarrow\infty}\limsup_{M\rightarrow\infty }\int_0^t || \tilde g(\mu^{N,M}(z^{*}(t')), z^{*}(t')) - \tilde g(\mu^*(t'), z^*(t'))|| dt'= 0,
\end{equation}
which, along with (\ref{e:HJB-1-17-1-N}), imply that
\begin{equation}\label{e:HJB-16-NM-proof-9}
\lim_{N\rightarrow\infty}\limsup_{M\rightarrow\infty }\kappa_t(N,M)=0.
\end{equation}
By Gronwall-Bellman lemma, from (\ref{e:HJB-16-NM-proof-6-11}) it follows that
$$
max_{t'\in [0,t]}||z^{N,M}(t')-z^*(t')||\leq  \kappa_t(N,M) e^{Lt}.
$$
The latter along with (\ref{e:HJB-16-NM-proof-9})  imply (\ref{e:HJB-1-19-2-NM}).

Let us now establish the validity of (\ref{e:HJB-1-19-1-NM}). Let  $t\in [0,\infty)$ be such that the ball $Q_t $ introduced in Assumption
\ref{SET-2}  is not empty. By triangle inequality,
\begin{equation}\label{e:summarizing-N-M-convergence}
\rho (\mu^{N,M}(z^{N,M}(t)),\mu^*(t))\leq \rho (\mu^{N,M}(z^{N,M}(t)),\mu^{N,M}(z^*(t))) + \rho (\mu^{N,M}(z^*(t)), \mu^*(t))
\end{equation}
Due to (\ref{e:HJB-1-19-2-NM}), $z^{N,M}(t)\in Q_t\ $ for $M$ and $N$ large enough. Hence, by (\ref{e:HJB-1-17-1-per-cont-mu-NM-g-1}),
$$
     \rho(\mu^{N,M}(z^{N,M}(t)), \mu^{N,M}(z^*(t)))\leq \kappa(||z^{N,M}(t')-z^*(t')||),
   $$
which implies that
$$
\lim_{N\rightarrow\infty}\limsup_{M\rightarrow\infty}\rho (\mu^{N,M}(z^{N,M}(t)),\mu^{N,M}(z^*(t)))=0.
$$
The latter, along with (\ref{e:HJB-16-NM-proof-2}) and (\ref{e:summarizing-N-M-convergence}), imply (\ref{e:HJB-1-19-1-NM}).

Finally, let us  prove (\ref{e:HJB-19-NM}).
By (\ref{Vy-ave-h-def-opt-1}) and (\ref{Vy-ave-h-def-averaged-measure-p-NM}), for any continuous function $\ \tilde{h}(\mu,z): \mathcal{P}(\mathcal{P}(U\times Y)\times Z) \rightarrow \R^1$, and
for an arbitrary small $\alpha > 0 $, there exists
 $\tilde{\mathcal{T}}> 0$ such that
 \begin{equation}\label{e:LA-Airport-2-ave}
 |\tilde{\mathcal{T}}^{-1}\int_0^{\tilde{\mathcal{T}}}  \tilde h(\mu^*(t),z^*(t))dt - \int_F \tilde h(\mu,z)p^*(d\mu,dz)|\leq \frac{\alpha}{2}
\end{equation}
and
\begin{equation}\label{e:LA-Airport-1-ave}
 |\tilde{\mathcal{T}}^{-1}\int_0^{\tilde{\mathcal{T}}} \tilde h(\mu^{N,M}(t),z^{N,M}(t))dt- \int_{F}\tilde{h}(\mu,z)\lambda^{N,M}(d\mu,dz)|\leq \frac{\alpha}{2} .
\end{equation}
Using (\ref{e:LA-Airport-1-ave}) and (\ref{e:LA-Airport-2-ave}), one can obtain
$$
|\int_{F}\tilde{h}(\mu,z)\lambda^{N,M}(d\mu,dz) - \int_F \tilde h(\mu,z)p^*(d\mu,dz)|
$$
\vspace{-.2in}
$$
\leq \ | \tilde{\mathcal{T}}^{-1}\int_0^{\tilde{\mathcal{T}}} \tilde h(\mu^{N,M}(t),z^{N,M}(t))dt - \tilde{\mathcal{T}}^{-1}\int_0^{\tilde{\mathcal{T}}}  \tilde h(\mu^*(t),z^*(t))dt| + \alpha .
$$
Due to (\ref{e:HJB-1-19-2-NM}) and (\ref{e:HJB-1-19-1-NM}), the latter implies the following inequality
$$
\lim_{N\rightarrow\infty}\limsup_{M\rightarrow\infty}|\int_{F}\tilde{h}(\mu,z)\lambda^{N,M}(d\mu,dz) - \int_F \tilde h(\mu,z)p^*(d\mu,dz)|\leq \alpha ,
$$
which, in turn, implies
\begin{equation}\label{e:LA-Airport-3-ave}
 \lim_{N\rightarrow\infty}\limsup_{M\rightarrow\infty}|\int_{F}\tilde{h}(\mu,z)\lambda^{N,M}(d\mu,dz) - \int_F \tilde h(\mu,z)p^*(d\mu,dz)|=0
\end{equation}
(due to the fact that $\alpha $ can be arbitrary small). This proves (\ref{e:HJB-1-17-1-N-z-const-def-near-opt-new}). Taking now $\tilde{h}(\mu,z) = \tilde{G}(\mu,z) $  in (\ref{e:LA-Airport-3-ave}) and having in mind that
$$
\int_{F}\tilde{G}(\mu,z)\lambda^{N,M}(d\mu,dz)=\lim_{\mathcal{T}\rightarrow\infty}\mathcal{T}^{-1}\int_0^\mathcal{T} \tilde G(\mu^{N,M}(t),z^{N,M}(t))dt
$$
(see (\ref{Vy-ave-h-def-averaged-measure-p-NM})) and that
$$
\int_F \tilde G(\mu,z)p^*(d\mu,dz) = \tilde G^*,
$$
one proves the validity of (\ref{e:HJB-19-NM}). This completes the proof of the theorem.
\endproof

\section{Proof of Theorem \ref{Main-SP-SP-Nemeric}}\label{Sec-Selected-proofs-new} Denote by
$z(t) $ the solution of the differential equation
\begin{equation}\label{e:SP-ACG-1}
z'(t)=\tilde{g}(\mu^{N,M} (t),z(t))
\end{equation}
considered on the interval $[\mathcal{T}_0, \mathcal{T}_0+\mathcal{T}] $ that satisfies the initial condition
\begin{equation}\label{e:SP-ACG-1-1}
 z(\mathcal{T}_0)=z\in Z.
\end{equation}
Also, denote by
$\bar z(t) $ the solution of the differential equation
\begin{equation}\label{e:SP-ACG-3}
z'(t)=\tilde{g}(\bar \mu^{N,M} (t), z(t))
\end{equation}
considered on the same interval $[\mathcal{T}_0, \mathcal{T}_0+\mathcal{T}] $ and satisfying the same initial condition (\ref{e:SP-ACG-1-1}), where
 $\bar \mu^{N,M}(t)$ is the piecewise constant function defined as follows
\begin{equation}\label{e:SP-ACG-2}
\bar \mu^{N,M}(t) \BYDEF \mu^{N,M}(t_l) \ \ \ \forall  t\in [t_l, t_{l+1}), \ \ \ l = 0,1,...\ .
\end{equation}
Using the piecewise continuity property (\ref{e:contr-rev-100-3-1002}), it can be readily established (using a standard argument, see, e.g., the proof of Theorem 4.5 in \cite{GR-long}) that
\begin{equation}\label{e:SP-ACG-4}
\max_{t\in [\mathcal{T}_0 , \mathcal{T}_0 + \mathcal{T}]}||\bar z(t)- z(t)||\ \leq \ \kappa_1 (\epsilon , \mathcal{T}), \ \ \ \ \ {\rm where} \ \ \ \ \ \lim_{\epsilon\rightarrow 0}
\kappa_1 (\epsilon , \mathcal{T}) = 0.
\end{equation}
The latter implies, in particular,
\begin{equation}\label{e:SP-ACG-5}
\max_{t\in [0,\mathcal{T}_0]}||z^{N,M}(t)- \bar z^{N,M}(t)||\leq \kappa_1 (\epsilon , \mathcal{T}_0),
\end{equation}
where $\bar z^{N,M}(t) $ is the solution of (\ref{e:SP-ACG-3}) that satisfies the initial condition $\ \bar z^{N,M}(0)= z^{N,M}(0) $.

Choose  now
$\mathcal{T}_0 $ in such a way that
\begin{equation}\label{e:SP-ACG-6}
c_3 e^{-c_4 \mathcal{T}_0}\BYDEF a < 1
\end{equation}
and denote by $\bar z^{N,M}_1(t) $ the solution of the system (\ref{e:SP-ACG-3}) considered on the interval $\ [\mathcal{T}_0,2 \mathcal{T}_0] $ with the initial condition $\bar z^{N,M}_1(\mathcal{T}_0) = z^{N,M}(\mathcal{T}_0)$.
From (\ref{e-opt-OM-2-106-z}) and (\ref{e:SP-ACG-5}) it follows that
\begin{equation}\label{e:SP-ACG-7}
||\bar z^{N,M}_1(2\mathcal{T}_0)- \bar z^{N,M}(2\mathcal{T}_0)||\leq a ||z^{N,M}(\mathcal{T}_0)- \bar z^{N,M}(\mathcal{T}_0) ||\leq a \kappa_1 (\epsilon , \mathcal{T}_0).
\end{equation}
Also, taking into account the validity of (\ref{e:SP-ACG-4}), one can write down
$$
||z^{N,M}(2\mathcal{T}_0)- \bar z^{N,M}(2\mathcal{T}_0)||\ \leq \ ||z^{N,M}(2\mathcal{T}_0)- \bar z^{N,M}_1(2\mathcal{T}_0)|| \ + \
||\bar z^{N,M}_1(2\mathcal{T}_0)- \bar z^{N,M}(2\mathcal{T}_0)||
$$
\vspace{-.2in}
$$
\leq \ \kappa_1 (\epsilon , \mathcal{T}_0)\ + \ a \kappa_1 (\epsilon , \mathcal{T}_0) \ \leq \ \frac{\kappa_1 (\epsilon , \mathcal{T}_0)}{1-a}.
$$
By continuing in a similar way, one can prove that, for any $k=1,2,... \  $,
$$
||z^{N,M}(k\mathcal{T}_0)- \bar z^{N,M}(k\mathcal{T}_0)||\ \leq \ (1 + a +... +a^{k-1}  ) \kappa_1 (\epsilon , \mathcal{T}_0)\ \leq \ \frac{\kappa_1 (\epsilon , \mathcal{T}_0)}{1-a} .
$$
Hence, by (\ref{e-opt-OM-2-106-z}),
$$
 \max_{t\in [k\mathcal{T}_0,(k+1)\mathcal{T}_0]}||z^{N,M}(t)- \bar z^{N,M}(t)||\  \leq \  c_3 \frac{\kappa_1 (\epsilon , \mathcal{T}_0)}{1-a}\ \ \ \forall \ k=0, 1,...
$$
and, consequently,
\begin{equation}\label{e:SP-ACG-8}
\sup_{t\in [0, \infty)}||z^{N,M}(t)- \bar z^{N,M}(t)||\ \leq \ c_3\frac{\kappa_1 (\epsilon , \mathcal{T}_0)}{1-a}\BYDEF \kappa_2(\epsilon), \ \ \ \ \ \ \ \
\lim_{\epsilon\rightarrow 0}\kappa_2(\epsilon) = 0.
\end{equation}
Using (\ref{e:SP-ACG-8}), one can obtain
$$
|\frac{1}{\mathcal{T}}\int_0^{\mathcal{T}} \tilde G(\mu^{N,M}(t),z^{N,M}(t))dt - \frac{1}{\mathcal{T}}\int_0^{\mathcal{T}} \tilde G(\bar \mu^{N,M}(t),\bar z^{N,M}(t))dt|
$$
\vspace{-.2in}
$$
\leq \ \frac{1}{\mathcal{T}}\int_0^{\mathcal{T}} |\tilde G(\mu^{N,M}(t),z^{N,M}(t)) - \tilde G(\bar \mu^{N,M}(t), z^{N,M}(t))|dt \ + \ L \kappa_2(\epsilon)
$$
\vspace{-.2in}
\begin{equation}\label{e:SP-ACG-9}
\leq \ \frac{1}{\mathcal{T}}\sum_{l=0}^{\lfloor \frac{\mathcal{T}}{\Delta(\epsilon)} \rfloor}\int_{t_l}^{t_{l+1}}|\tilde G(\mu^{N,M}(t),z^{N,M}(t)) - \tilde G( \mu^{N,M}(t_l), z^{N,M}(t))|dt \ + \ L \kappa_2(\epsilon)  \ + \ 2M\Delta(\epsilon) \ \ \ \ \ \ \forall \mathcal{T} \geq 1,
\end{equation}
where  $\lfloor \cdot \rfloor $ stands for the floor function   ($\lfloor x \rfloor $ is the maximal integer number that is less or equal than  $x $), $L$ is a Lipschitz constant (for simplicity it is assumed that $G(u,y,z) $ is Lipschitz continuous in $z $) and $\ M\BYDEF \max_{(u,y,z)\in U\times Y\times Z}|G(u,y,z)|$.

Without loss of generality, one may assume  that
$r_{\delta}$ is  decreasing with $\delta $ and that
$r_{\delta}\leq \delta$ (the later can be achieved by replacing $r_{\delta}$ with  $\min\{\delta , r_{\delta} \} $  if necessary). Having this
in mind, define $\delta(\epsilon)$ as the solution of the problem
\begin{equation}\label{e-implicit-1-0}
\min \{ \delta \ : \ r_{\delta}\geq \Delta^{\frac{1}{2}}(\epsilon)\}.
\end{equation}
That is,
\begin{equation}\label{e-implicit-1}
r_{\delta(\epsilon)} = \Delta^{\frac{1}{2}}(\epsilon).
\end{equation}
Note that, by construction,
\begin{equation}\label{e-implicit-2}
\lim_{\epsilon\rightarrow 0}\delta(\epsilon) = 0, \ \ \ \ \ \ \ \ \ \ \     \delta(\epsilon)\geq \Delta^{\frac{1}{2}}(\epsilon).
\end{equation}
By (\ref{e:contr-rev-100-3-1002}),
\begin{equation}\label{e:SP-ACG-10}
\int_{t_l}^{t_{l+1}}|\tilde G(\mu^{N,M}(t),z^{N,M}(t))dt - \tilde G( \mu^{N,M}(t_l), z^{N,M}(t))|dt\leq \Delta(\epsilon)\nu(\Delta(\epsilon))
 \ \ \ {\rm if} \ \ \ \ t_l\notin\cup_{i=1}^k ( \mathcal{T}_i - \delta(\epsilon) , \mathcal{T}_i+ \delta (\epsilon)).
\end{equation}
Also,
\begin{equation}\label{e:SP-ACG-11}
\int_{t_l}^{t_{l+1}}|\tilde G(\mu^{N,M}(t),z^{N,M}(t))dt - \tilde G( \mu^{N,M}(t_l), z^{N,M}(t))|dt\leq 2M\Delta(\epsilon)
 \ \ \ {\rm if} \ \ \ \ t_l\in\cup_{i=1}^k ( \mathcal{T}_i - \delta(\epsilon) , \mathcal{T}_i+ \delta (\epsilon)).
\end{equation}
 Taking (\ref{e:contr-rev-100-3-1002-11}), (\ref{e:SP-ACG-10}) and (\ref{e:SP-ACG-11}) into account, one can use (\ref{e:SP-ACG-9}) to obtain the following estimate
$$
|\frac{1}{\mathcal{T}}\int_0^{\mathcal{T}} \tilde G(\mu^{N,M}(t),z^{N,M}(t))dt - \frac{1}{\mathcal{T}}\int_0^{\mathcal{T}} \tilde G(\bar \mu^{N,M}(t),\bar z^{N,M}(t))dt|
$$
\vspace{-.2in}
$$
\leq \ \frac{1}{\mathcal{T}}
\lfloor
\frac{\mathcal{T}}{\Delta(\epsilon)} \rfloor \Delta(\epsilon)\nu(\Delta(\epsilon)) \ + \ \frac{1}{\mathcal{T}}  (c \mathcal{T}) [\frac{2\delta(\epsilon)}{\Delta(\epsilon)}+2] (2M\Delta(\epsilon))
\ + \ L \kappa_2(\epsilon)  \ + \ 2M\Delta(\epsilon)\BYDEF \kappa_3(\epsilon) \ \ \ \ \ \ \forall \mathcal{T} \geq 1.
$$
Thus,
\begin{equation}\label{e:SP-ACG-12}
\sup_{\mathcal{T} \geq 1}|\frac{1}{\mathcal{T}}\int_0^{\mathcal{T}} \tilde G(\mu^{N,M}(t),z^{N,M}(t))dt - \frac{1}{\mathcal{T}}\int_0^{\mathcal{T}} \tilde G(\bar \mu^{N,M}(t),\bar z^{N,M}(t))dt|\leq \kappa_3(\epsilon), \ \ \  \ \ \ \ \lim_{\epsilon\rightarrow 0}\kappa_3(\epsilon) =0.
\end{equation}
Denote by
$\ \bar{\bar{z}}(t) $ the solution of the differential equation (\ref{e:SP-ACG-3})
considered on the interval $[\mathcal{T}_0, \mathcal{T}_0+\mathcal{T}] $ and satisfying the  initial condition
$\bar{\bar{z}}(\mathcal{T}_0)= z_{\epsilon}^{N,M}(\mathcal{T}_0) $, where $\mathcal{T}_0\BYDEF  l_0\Delta (\epsilon) $ for some $l_0\geq 0 $.
Subtracting the equation
$$
\bar{\bar{z}}(t_{l+1})=\bar{\bar{z}}(t_l) + \int_{t_l}^{t_{l+1}}\tilde g(\mu^{N,M}(t_l),
\bar{\bar{z}}(t))dt, \ \  \ \  \ l\geq l_0,
$$
from the equation
$$
z_{\epsilon}^{N,M}(t_{l+1})=z_{\epsilon}^{N,M}(t_l) + \int_{t_l}^{t_{l+1}}g(u_{z^{N,M}(t_l)}(\frac{t-t_l}{\epsilon}),y_{\epsilon}^{N,M}(t),
z_{\epsilon}^{N,M}(t))dt, \ \  \ \  \ l\geq l_0 ,
$$
one can obtain
$$
||z_{\epsilon}^{N,M}(t_{l+1}) - \bar{\bar{z}}(t_{l+1})||\ \leq \ ||z_{\epsilon}^{N,M}(t_{l}) - \bar{\bar{z}}(t_{l})||
$$
\vspace{-.2in}
$$
+ \ \int_{t_l}^{t_{l+1}}||g(u_{z^{N,M}(t_l)}(\frac{t-t_l}{\epsilon}),y_{\epsilon}^{N,M}(t),
z_{\epsilon}^{N,M}(t))dt\ - \ g(u_{z^{N,M}(t_l)}(\frac{t-t_l}{\epsilon}),y_{\epsilon}^{N,M}(t),
z_{\epsilon}^{N,M}(t_l))|| dt
$$
\vspace{-.2in}
$$
+ \ ||\int_{t_l}^{t_{l+1}}g(u^{N,M}_{z^{N,M}(t_l)}(\frac{t-t_l}{\epsilon}),\ y_{\epsilon}^{N,M}(t),\
z_{\epsilon}^{N,M}(t_l))dt\ - \ \Delta(\epsilon)\tilde g(\mu^{N,M}(t_l),
z_{\epsilon}^{N,M}(t_l))||
$$
\vspace{-.2in}
$$
 + \  \int_{t_l}^{t_{l+1}}||\tilde g(\mu^{N,M}(t_l),
z_{\epsilon}^{N,M}(t_l)) - \tilde g(\mu^{N,M}(t_l),
\bar{\bar{z}}(t))||dt
$$
\vspace{-.2in}
\begin{equation}\label{e-SP-10011}
\leq \ ||z_{\epsilon}^{N,M}(t_{l}) - \bar{\bar{z}}(t_{l})||\ +\ L_1\Delta(\epsilon)||z_{\epsilon}^{N,M}(t_{l})\  -\ \bar{\bar{z}}(t_{l})||\ + \ L_2\Delta^2(\epsilon)
+\ \Delta(\epsilon)\bar{\bar{\phi}}_g(\frac{\Delta(\epsilon)}{\epsilon}),
\end{equation}
where $L_i, \ i=1,2,\  $ are positive constants and $\bar{\bar{\phi}}_g(\cdot)$ is defined in (\ref{e-opt-OM-2-106-1-0}).
Note that, in order to obtain the estimate above, one needs to take into account the fact that
\begin{equation}\label{e-SP-10012}
\max\{  \max_{t\in [t_l, t_{l+1}]}\{ ||z_{\epsilon}^{N,M}(t) -z_{\epsilon}^{N,M}(t_{l})||\}, \ \
\max_{t\in [t_l, t_{l+1}]}\{ ||\bar{\bar{z}}(t) -\bar{\bar{z}}(t_{l})||\}\ \} \ \leq \ L_3 \Delta(\epsilon), \ \ \ \ L_3 > 0
\end{equation}
as well as the fact that
  (see (\ref{e-opt-OM-2-106-1-0}))
$$
||\int_{t_l}^{t_{l+1}}g(u^{N,M}_{z^{N,M}(t_l)}(\frac{t-t_l}{\epsilon}),\ y_{\epsilon}^{N,M}(t),\
z_{\epsilon}^{N,M}(t_l))dt\ - \ \Delta(\epsilon)\tilde g(\mu^{N,M}(t_l),
z_{\epsilon}^{N,M}(t_l))||
$$
\vspace{-.2in}
$$
= \Delta(\epsilon)[\ (\frac{\Delta(\epsilon)}{\epsilon})^{-1}\int_0^{\frac{\Delta(\epsilon)}{\epsilon}}g(u^{N,M}_{z^{N,M}(t_l)}(\tau),\ y^{N,M}_{z^{N,M}(t_l)}(\tau, y_{\epsilon}^{N,M}(t_l)),\ z_{\epsilon}^{N,M}(t_l))d\tau \ - \ \tilde g(\mu^{N,M}(t_l), z_{\epsilon}^{N,M}(t_l)) ]
$$
\vspace{-.2in}
\begin{equation}\label{e-SP-10012-1}
\leq \ \Delta(\epsilon)\bar{\bar{\phi}}_g(\frac{\Delta(\epsilon)}{\epsilon}),
\end{equation}
where $ \ \ \tau = \frac{t-t_l}{\epsilon}  \ \ $ and $ \ \ y^{N,M}_{z^{N,M}(t_l)}(\tau, y_{\epsilon}^{N,M}(t_l)) = y_{\epsilon}^{N,M}(t_l + \epsilon \tau)$.
From (\ref{e-SP-10011}) it follows (see Proposition 5.1 in \cite{Gai1})
$$
||z_{\epsilon}^{N,M}(t_{l}) - \bar{\bar{z}}(t_{l})||\leq \kappa_4(\epsilon, \mathcal{T}), \ \ \ l=l_0,\ l_0+1, ...,\ l_0 +\lfloor
\frac{\mathcal{T}}{\Delta(\epsilon)} \rfloor, \ \ \ \ \ \lim_{\epsilon\rightarrow 0}\kappa_4(\epsilon) =0.
$$
This   (due to (\ref{e-SP-10012})) leads to
\begin{equation}\label{e:SP-ACG-4-11}
\max_{t\in [\mathcal{T}_0 , \mathcal{T}_0 + \mathcal{T}]}||z_{\epsilon}^{N,M}(t) - \bar{\bar{z}}(t)||\ \leq \ \kappa_5 (\epsilon , \mathcal{T}), \ \ \ \ \ {\rm where} \ \ \ \ \ \lim_{\epsilon\rightarrow 0}
\kappa_5(\epsilon , \mathcal{T}) = 0,
\end{equation}
and, in particular, to
\begin{equation}\label{e:SP-ACG-4-11-1}
\max_{t\in [0 , \mathcal{T}_0]}||z_{\epsilon}^{N,M}(t) - \bar{z}^{N,M}(t)||\ \leq \ \kappa_5 (\epsilon , \mathcal{T}_0)
\end{equation}
(since, by definition,  $ \ z_{\epsilon}^{N,M}(0) = z^{N,M}(0) $ and $ \  \bar{z}^{N,M}(0)= z^{N,M}(0) $).
Assume that $\mathcal{T}_0 $ is chosen in such a way that (\ref{e:SP-ACG-6}) is satisfied
and denote by $\bar{\bar{z}}_1(t) $ the solution of the system (\ref{e:SP-ACG-3}) considered on the interval $\ [\mathcal{T}_0,2 \mathcal{T}_0] $ with the initial condition $\bar{\bar{z}}_1(\mathcal{T}_0) = z_{\epsilon}^{N,M}(\mathcal{T}_0)$. By (\ref{e-opt-OM-2-106-z}) and (\ref{e:SP-ACG-4-11-1}),
\begin{equation}\label{e:SP-ACG-7-SP}
||\bar{\bar{z}}_1(2 \mathcal{T}_0)- \bar z^{N,M}(2 \mathcal{T}_0)||\leq a ||z_{\epsilon}^{N,M}(\mathcal{T}_0)- \bar z^{N,M}(\mathcal{T}_0) ||\leq a \kappa_5(\epsilon , \mathcal{T}_0).
\end{equation}
Also, by (\ref{e:SP-ACG-4-11}),
$$
||z_{\epsilon}^{N,M}(2 \mathcal{T}_0) -  \bar z^{N,M}(2 \mathcal{T}_0)||\ \leq \ ||z_{\epsilon}^{N,M}(2 \mathcal{T}_0)
-\bar{\bar{z}}_1(2 \mathcal{T}_0)|| \ + \ ||\bar{\bar{z}}_1(2 \mathcal{T}_0) - \bar z^{N,M}(2 \mathcal{T}_0)||
$$
\vspace{-.2in}
$$
\leq \ \kappa_5 (\epsilon , \mathcal{T}_0) \ + \ a \kappa_5 (\epsilon , \mathcal{T}_0) \ \leq \ \frac{\kappa_5 (\epsilon , \mathcal{T}_0)}{1-a}.
$$
Continuing in a similar way, one can prove that, for any  $k=1,2,... \ $,
\begin{equation}\label{e-needed}
||z_{\epsilon}^{N,M}(k\mathcal{T}_0) -  \bar z^{N,M}(k\mathcal{T}_0)||\ \leq \ (1+a +... +a^{k-1}) \kappa_5 (\epsilon , \mathcal{T}_0)\ \leq \ \frac{\kappa_5 (\epsilon , \mathcal{T}_0)}{1-a}.
\end{equation}
Denote by $\bar{\bar{z}}_k(t) $ the solution of the system (\ref{e:SP-ACG-3}) considered on the interval $\ [k\mathcal{T}_0, (k+1)\mathcal{T}_0] $ with the initial condition $\bar{\bar{z}}_k(k\mathcal{T}_0) = z_{\epsilon}^{N,M}(k\mathcal{T}_0)$. By (\ref{e-opt-OM-2-106-z}) and (\ref{e:SP-ACG-4-11-1}),
$$
\max_{t\in [k\mathcal{T}_0,(k+1)\mathcal{T}_0]}||z_{\epsilon}^{N,M}(t) - \bar z^{N,M}(t) ||\ \leq  \max_{t\in [k\mathcal{T}_0,(k+1)\mathcal{T}_0]}||z_{\epsilon}^{N,M}(t) - \bar{\bar{z}}_k(t)|| \ +  \max_{t\in [k\mathcal{T}_0,(k+1)\mathcal{T}_0]}||\bar{\bar{z}}_k(t) -  \bar z^{N,M}(t)||
$$
\vspace{-.2in}
$$
\leq \ \kappa_5 (\epsilon , \mathcal{T}_0)\ + c_3 ||z_{\epsilon}^{N,M}(k\mathcal{T}_0) -  \bar z^{N,M}(k\mathcal{T}_0)|| .
$$
Thus, by (\ref{e-needed}),
\begin{equation}\label{e:SP-ACG-8-SP}
\sup_{t\in [0, \infty)}||z_{\epsilon}^{N,M}(t) -  \bar z^{N,M}(t)||\ \leq \ \ \kappa_5 (\epsilon , \mathcal{T}_0)\ + c_3 \frac{\kappa_5 (\epsilon , \mathcal{T}_0)}{1-a}\BYDEF \kappa_6(\epsilon), \ \ \ \ \ \ \ \
\lim_{\epsilon\rightarrow 0}\kappa_6(\epsilon) = 0.
\end{equation}
 From (\ref{e:SP-ACG-8}) and
(\ref{e:SP-ACG-8-SP}) it also follows that
$$
|\frac{1}{\mathcal{T}}\int_0^{\mathcal{T}}G(u_{\epsilon}^{N,M}(t), y_{\epsilon}^{N,M}(t),z_{\epsilon}^{N,M}(t))dt - \frac{1}{\mathcal{T}}\int_0^{\mathcal{T}} \tilde G(\bar \mu^{N,M}(t),\bar z^{N,M}(t))dt|
$$
\vspace{-.2in}
$$
\leq \ |\frac{1}{\mathcal{T}}\int_0^{\mathcal{T}}G(u_{\epsilon}^{N,M}(t), y_{\epsilon}^{N,M}(t),z^{N,M}(t))dt - \frac{1}{\mathcal{T}}\int_0^{\mathcal{T}} \tilde G(\bar \mu^{N,M}(t), z^{N,M}(t))dt| \ + \ L (\kappa_2(\epsilon) +\kappa_6(\epsilon))
$$
\vspace{-.2in}
$$
\leq \ \frac{1}{\mathcal{T}}\sum_{l=0}^{\lfloor \frac{\mathcal{T}}{\Delta(\epsilon)} \rfloor}|\int_{t_l}^{t_{l+1}}G(u_{ z^{N,M}(t_l)}^{N,M}(\frac{t-t_l}{\epsilon}), y_{\epsilon}^{N,M}(t), z^{N,M}(t_l))dt - \int_{t_l}^{t_{l+1}}\tilde G( \mu^{N,M}(t_l), z^{N,M}(t_l))dt| \
$$
\vspace{-.2in}
\begin{equation}\label{e:SP-ACG-7-SP-3}
+ \ L (\kappa_2(\epsilon) +\kappa_6(\epsilon))  \ + \ 2M\Delta(\epsilon)  \ + \ 2L L_3\Delta(\epsilon) \ \ \ \ \ \ \forall \mathcal{T} \geq 1,
\end{equation}
where $L$ and $M$ are as in (\ref{e:SP-ACG-9}) and  it has been taking into account that $ \ \max_{t\in [t_l, t_{l+1}]} ||z^{N,M}(t) -z^{N,M}(t_{l})||\leq \ L_3 \Delta(\epsilon) $, with
$L_3$ being the same constant as in (\ref{e-SP-10012}). Similarly to (\ref{e-SP-10012-1}), one can obtain (using (\ref{e-opt-OM-2-106-1-0}))
$$
||\int_{t_l}^{t_{l+1}}G(u^{N,M}_{z^{N,M}(t_l)}(\frac{t-t_l}{\epsilon}),\ y_{\epsilon}^{N,M}(t),\
z^{N,M}(t_l))dt\ - \ \Delta(\epsilon)\tilde G(\mu^{N,M}(t_l),
z^{N,M}(t_l))||
$$
\vspace{-.2in}
$$
= \Delta(\epsilon)[\ (\frac{\Delta(\epsilon)}{\epsilon})^{-1}\int_0^{\frac{\Delta(\epsilon)}{\epsilon}}G(u^{N,M}_{z^{N,M}(t_l)}(\tau),\ y^{N,M}_{z^{N,M}(t_l)}(\tau, y_{\epsilon}^{N,M}(t_l)),\ z^{N,M}(t_l))dt \ - \ \tilde G(\mu^{N,M}(t_l), z^{N,M}(t_l)) ]
$$
\vspace{-.2in}
$$
\leq \ \Delta(\epsilon)\bar{\bar{\phi}}_G(\frac{\Delta(\epsilon)}{\epsilon}),\ \ \ \ \ \ \ \ \ \ \ .
$$
The latter along with (\ref{e:SP-ACG-7-SP-3}) imply that
\begin{equation}\label{e-last}
|\frac{1}{\mathcal{T}}\int_0^{\mathcal{T}}G(u_{\epsilon}^{N,M}(t), y_{\epsilon}^{N,M}(t),z_{\epsilon}^{N,M}(t))dt - \frac{1}{\mathcal{T}}\int_0^{\mathcal{T}} \tilde G(\bar \mu^{N,M}(t),\bar z^{N,M}(t))dt|\ \leq \ \kappa_7(\epsilon) \ \ \ \ \forall \mathcal{T} \geq 1,
\end{equation}
where
$$
\kappa_7(\epsilon)\BYDEF \ L (\kappa_2(\epsilon) +\kappa_6(\epsilon))  \ + \ 2M\Delta(\epsilon)  \ + \ 2L L_3\Delta(\epsilon) \  + \ \bar{\bar{\phi}}_G(\frac{\Delta(\epsilon)}{\epsilon}), \ \ \ \ \ \ \ \lim_{\epsilon\rightarrow 0} \kappa_7(\epsilon) = 0.
$$
Hence, by (\ref{e:SP-ACG-12}),
\begin{equation}\label{e-last-0}
|\frac{1}{\mathcal{T}}\int_0^{\mathcal{T}}G(u_{\epsilon}^{N,M}(t), y_{\epsilon}^{N,M}(t),z_{\epsilon}^{N,M}(t))dt - \frac{1}{\mathcal{T}}\int_0^{\mathcal{T}} \tilde G( \mu^{N,M}(t), z^{N,M}(t))dt|\ \leq \ \kappa_3(\epsilon)\ + \ \kappa_7(\epsilon) \ \ \ \ \forall \mathcal{T} \geq 1,
\end{equation}
and, consequently,
$$
|\liminf_{\mathcal{T}\rightarrow\infty}\frac{1}{\mathcal{T}}\int_0^{\mathcal{T}}G(u_{\epsilon}^{N,M}(t), y_{\epsilon}^{N,M}(t),z_{\epsilon}^{N,M}(t))dt
- \tilde V^{N,M}|\ \leq \ \kappa_3(\epsilon)\ + \ \kappa_7(\epsilon),
$$
\vspace{-.2in}
$$
\Rightarrow \ \ \ \ \lim_{\epsilon\rightarrow 0}\liminf_{\mathcal{T}\rightarrow\infty}\frac{1}{\mathcal{T}}\int_0^{\mathcal{T}}G(u_{\epsilon}^{N,M}(t), y_{\epsilon}^{N,M}(t),z_{\epsilon}^{N,M}(t))dt = \tilde V^{N,M}= \tilde G^* + \beta(N,M)
$$
(see (\ref{e:HJB-16-NM}) and (\ref{e:HJB-16-NM-beta})). Due to (\ref{lim-inequalities}), the latter proves the $\beta(N,M)$-asymptotic near optimality of the control $u^{N,M}_{\epsilon}(\cdot) $. Also, the estimate (\ref{e:contr-rev-100-3-1002-101}) follows from (\ref{e:SP-ACG-8}) and (\ref{e:SP-ACG-8-SP}).

 Using an arbitrary Lipschitz continuous function $h(u,y,z) $ instead of $G(u,y,z)$, one can obtain (similarly to (\ref{e-last-0}))
 \begin{equation}\label{e-last-last}
|\frac{1}{\mathcal{T}}\int_0^{\mathcal{T}}h(u_{\epsilon}^{N,M}(t), y_{\epsilon}^{N,M}(t),z_{\epsilon}^{N,M}(t))dt - \frac{1}{\mathcal{T}}\int_0^{\mathcal{T}} \tilde h( \mu^{N,M}(t),
 z^{N,M}(t))dt|\ \leq \ \kappa_8(\epsilon) \ \ \ \ \forall \mathcal{T} \geq 1,
\end{equation}
where $\ \lim_{\epsilon\rightarrow \infty}\kappa_8(\epsilon) = 0 $.
 If the triplets $(u_{\epsilon}^{N,M}(\cdot), y_{\epsilon}^{N,M}(\cdot),z_{\epsilon}^{N,M}(\cdot)) $ generates the occupational measure $\gamma^{N,M}_{\epsilon} $, then (see (\ref{e:oms-0-1-infy}))
 $$
 \lim_{\mathcal{T}\rightarrow\infty}\frac{1}{\mathcal{T}}\int_0^{\mathcal{T}}h(u_{\epsilon}^{N,M}(t), y_{\epsilon}^{N,M}(t),z_{\epsilon}^{N,M}(t))dt
 =\int_{U\times Y\times Z}h(u,y,z)\gamma^{N,M}_{\epsilon} (du,dy,dz).
 $$
 Hence, passing to the limit in (\ref{e-last-last}) with $\mathcal{T}\rightarrow\infty $ and taking into account (\ref{Vy-ave-h-def-averaged-measure-p-NM}), one obtains
 $$
|\int_{U\times Y\times Z}h(u,y,z)\gamma^{N,M}_{\epsilon} (du,dy,dz) \ - \ \int_{F}\tilde{h}(\mu,z)\lambda^{N,M}(d\mu,dz)|\ \leq \ \kappa_8(\epsilon).
$$
 By the definition of the map $\Phi(\cdot) $ (see (\ref{e:h&th-1})), the latter implies that
 $$
 |\int_{U\times Y\times Z}h(u,y,z)\gamma^{N,M}_{\epsilon} (du,dy,dz) \ - \  \int_{U\! \times Y \times Z} \!h(u,y,z) \Phi(\lambda^{N,M}) (du,dy,dz)| \ \leq \ \kappa_8(\epsilon),
 $$
 which, in turn, implies (\ref{e:contr-rev-100-3-1002-101-101}).
 This completes the proof.
\endproof

\end{document}